\documentclass[11pt]{article}
\usepackage[utf8]{inputenc}
\usepackage[colorlinks=true]{hyperref}
\hypersetup{urlcolor=blue, citecolor=red, linkcolor=blue}
\usepackage[polish, english]{babel}
\usepackage{dutchcal} 
\usepackage{mleftright}
\usepackage{stmaryrd}
\SetSymbolFont{stmry}{bold}{U}{stmry}{m}{n}
\usepackage{microtype}

\usepackage{bbm}
\usepackage{colortbl}

\usepackage{graphicx}
\usepackage[active]{srcltx}

\usepackage{amsmath,amssymb}

\numberwithin{equation}{section}

\usepackage{physics}
\usepackage{bm}
\usepackage{enumerate}
\usepackage{amsthm}
\usepackage[all,2cell]{xy}
\usepackage{mathrsfs}
\usepackage{mathtools}
\mathtoolsset{showonlyrefs}

\usepackage{tikz-cd}
\usetikzlibrary{babel}

\usepackage{xcolor}
\DeclareUnicodeCharacter{0301}{*************************************}

\usepackage[a4paper, left=25mm, right=25mm, top=30mm, bottom=25mm]{geometry}

\usepackage{marginnote}

\usepackage{imakeidx}
\makeindex[intoc]

\usepackage[sort]{natbib}

\usepackage[nottoc]{tocbibind}
\usepackage[colorinlistoftodos]{todonotes} 

\newcommand{\bh}{\boldsymbol{h}}

\theoremstyle{plain}
\newtheorem{theorem}{Theorem}[section]
\newtheorem{proposition}[theorem]{Proposition}
\newtheorem{corollary}[theorem]{Corollary}
\newtheorem{lemma}[theorem]{Lemma}

\theoremstyle{definition}
\newtheorem{definition}[theorem]{Definition}
\newtheorem{remark}[theorem]{Remark}
\newtheorem{example}[theorem]{Example}

\AtBeginEnvironment{remark}{%
  \pushQED{\qed}%
}
\AtEndEnvironment{remark}{\popQED\endexample}

\AtBeginEnvironment{example}{%
  \pushQED{\qed}%
}
\AtEndEnvironment{example}{\popQED\endexample}

\newcommand\restr[2]{{
  \left.\kern-\nulldelimiterspace #1 \right|_{#2} 
}}

\newcommand{\R}{\mathbb{R}}

\renewcommand{\d}{\mathrm{d}}

\newcommand{\Cinfty}{\mathscr{C}^\infty}
\newcommand{\T}{\mathrm{T}}

\newcommand{\cT}{\mathrm{T}^\ast}

\newcommand{\dker}{\ker \d \bm \eta}

\newcommand*{\inn}[1]{\iota_{#1}}

\newcommand{\Lie}{\mathscr{L}}

\newcommand{\X}{\mathfrak{X}}

\newcommand{\D}{\mathcal{D}}

\newcommand{\Ip}{\mathfrak{Im}}
\newcommand{\Rp}{\mathfrak{Re}}

\newcommand{\pd}[1]{\frac{\partial}{\partial #1}}
\newcommand{\tder}[1]{\frac{\text{d} #1}{\text{d}t}}

\newcommand{\bfX}{\mathbf{X}}

\newcommand{\qeddiamond}{\hfill $\diamond$}

\newcommand{\parder}[2]{\frac{\partial #1}{\partial #2}}

\newcommand{\tparder}[2]{\partial #1/\partial #2}

\newcommand{\boplus}{{\textstyle\bigoplus}}

\DeclareMathOperator{\rk}{rk}

\DeclareMathAlphabet{\mathpzc}{OT1}{pzc}{m}{it}
\def\d{\mathrm{d}}

\newcommand\finish{\xqed{ }}
\newcommand\xqed[1]{%
  \leavevmode\unskip\penalty9999 \hbox{}\nobreak\hfill
  \quad\hbox{#1}}

\title{k-symplectic Lie systems: theory and applications}

\usepackage{xcolor}

\usepackage{fancyhdr}

\begin{document}

\vspace{5em}

{\huge\sffamily\raggedright
    $k$-contact Lie systems: theory and applications
}
\vspace{2em}

{\large\raggedright
    \today
}

\vspace{3em}

{\Large\raggedright\sffamily
    Javier de Lucas
}

{\raggedright

    Department of Mathematical Methods in Physics, University of Warsaw, \\ ul. Pasteura 5, 02-093, Warszawa, Poland.\\
        \medskip
    Centre de Recherches Math\'ematiques, Universit\'e de Montr\'eal,\\ Pavillon André-Aisenstadt, 2920, Chemin de la Tour,  Montréal (Québec) Canada  H3T 1J4.\\
    e-mail: \href{mailto:javier.de.lucas@fuw.edu.pl}{javier.de.lucas@fuw.edu.pl} --- orcid: \href{https://orcid.org/0000-0001-8643-144X}{0000-0001-8643-144X}
}

\bigskip

{\Large\raggedright\sffamily
    Xavier Rivas
}\vspace{1mm}\newline
{\raggedright
    Department of Computer Engineering and Mathematics, Universitat Rovira i Virgili\\
    Avinguda Països Catalans 26, 43007, Tarragona, Spain.\\
    e-mail: \href{mailto:xavier.rivas@urv.cat}{xavier.rivas@urv.cat} --- orcid: \href{https://orcid.org/0000-0002-4175-5157}{0000-0002-4175-5157}
}

\bigskip

{\Large\raggedright\sffamily
    Tomasz Sobczak
}\vspace{1mm}\newline
{\raggedright
    Department of Mathematical Methods in Physics, University of Warsaw, \\ ul. Pasteura 5, 02-093, Warszawa, Poland.\\
    e-mail: \href{mailto:t.sobczak2@uw.edu.pl}{t.sobczak2@uw.edu.pl} --- orcid: \href{https://orcid.org/0009-0002-9577-0456}{0009-0002-9577-0456}
}

\vspace{3em}

{\large\bf\raggedright
    Abstract
}\vspace{1mm}\newline
{\noindent
  This paper introduces a new class of Lie systems that are Hamiltonian relative to a $k$-contact manifold. We  show that a recent distributional approach to $k$-contact manifolds along with a related $k$-contact Hamiltonian vector field notion allow us to understand  relevant Lie systems as Hamiltonian relative to a $k$-contact manifold. Our procedure  is  more general than previously known methods with this aim. As a result, we find that a plethora of Lie systems related to control and physical problems can be considered in a natural manner as $k$-contact Lie systems. We study their $t$-dependent and $t$-independent constants of motion, master symmetries of higher order, and other properties of interest. Finally, we use our new techniques and findings to study PDE Lie systems with a compatible $k$-contact manifold, some of which become Hamilton--De Donder--Weyl equations. 
}
\bigskip

{\large\bf\raggedright
    Keywords: } Lie system, superposition rule, $k$-contact manifold, $k$-contact distribution, control system, master symmetry, generalised constant of motion,  Hamilton--De Donder--Weyl equations. 
\medskip

{\large\bf\raggedright
    MSC2020 codes: {\it 34A05, 
    34A26 
    (primary),  17B66, 
    22E70 
     (secondary) }
}

\bigskip


\newpage

{\setcounter{tocdepth}{2}
\def\baselinestretch{1}
\small
\def\addvspace#1{\vskip 1pt}
\parskip 0pt plus 0.1mm
\tableofcontents
}

\section{Introduction}
\label{ch:Lie-systems}

Towards the end of the XIX century, there was much interest in determining $t$-dependent systems of first-order ordinary differential equations whose solutions could be described as a $t$-independent function of a generic family of particular solutions and some constants: the \textit{Lie systems} \cite{CL_11,CGM_00,LS_20,Lie1}. In particular, Lie proved that Lie systems are related to a certain curve in a finite-dimensional Lie algebra of vector fields: a {\it Vessiot--Guldberg Lie algebra}.

Lie systems have being studied along with different compatible geometric structures as this allowed one to study more easily their properties (see \cite{LS_20} for a survey on the topic). In particular, \textit{Lie--Hamilton systems} are Lie systems admitting a Vessiot--Guldberg Lie algebra of Hamiltonian vector fields relative to a Poisson bivector \cite{CLS_13,CGM_00}. Lie--Hamilton systems were the first class of Lie systems admitting a Vessiot--Guldberg Lie algebra of Hamiltonian vector fields relative to a geometric structure that was studied. In this case, the Poisson bivector enabled the derivation of their superposition rules and constants of motion in a simpler manner than previous methods via, for instance, the Poisson coalgebra method \cite{CLS_13,CGM_00,LS_20,BCHLS13}. Nevertheless, Lie--Hamilton systems were insufficient for studying many types of Lie systems \cite{LS_20}. Several types of Hamiltonian Lie systems relative to different geometric structures have been posteriorly analysed. They not only aided in comprehending the characteristics of various Lie systems but also facilitated the development of innovative results and techniques in differential geometry. For instance, $k$-symplectic Lie systems introduced Poisson brackets of certain families of functions in $k$-symplectic geometry to obtain superposition rules \cite{LV_15}, or, quite relevantly, Gr\`acia, de Lucas, Rom\'an-Roy, Mu\~noz-Lecanda,  Rivas and Vilari\~no introduced  multisymplectic Lie systems, which were used to study different types of multisymplectic reduction and tensor invariants for Lie systems \cite{LGRRV_22,GLMV_19}.

Contact Lie systems \cite{LR_23,Car_25,CCH_25} were recently introduced to study Lie systems of physical relevance  which, in some cases, cannot be studied via Lie--Hamilton systems. These works were also motivated by the interest of contact geometry in the study of dynamical systems, specially in those with a dissipative behaviour \cite{CCM_18,LL_19}. Contact Lie systems appeared for the first time in \cite{LR_23}. Recently, contact Lie systems have also appeared in the study of the reduction of Lie--Hamilton systems related to thermodynamical systems \cite{CCH_25}, while the characterisation of  contact Lie systems on low-dimensional manifolds was also related to the geometric analysis of hydrodynamic equations via quasi-rectifiable families of vector fields in \cite{GL_25}.

Recently $k$-contact geometry appeared as a generalisation of contact geometry to study field theories \cite{Riv_21}.  In particular, $k$-contact geometry allows for the description of the Hamilton--De Donder--Weyl equations for field theories. The initial idea behind $k$-contact geometry was to analyse the properties of a generalisation of contact forms to a certain type of $\mathbb{R}^k$-valued differential one-forms: the {\it $k$-contact forms}. Moreover, many other mathematical and physical related features of $k$-contact forms  appeared over the years \cite{LRS_25,Riv_23,Vil_25,GRR_22,GGMRR_21}.
This work introduces \textit{$k$-contact Lie systems}, namely Lie systems admitting a Vessiot--Guldberg Lie algebra of $\bm \eta$-Hamiltonian vector fields relative to a $k$-contact form, which retrieves contact Lie systems as a particular case for $k=1$ \cite{LR_23,CCH_25}. Then, we use $k$-contact geometry to investigate $k$-contact Lie systems. 

$k$-Contact geometry was, at first, mainly aimed at studying systems of partial differential equations. Then,  \cite{LRS_25} used $k$-contact forms to define $\bm \eta$-Hamiltonian vector fields relative to a $k$-contact form $\bm \eta$, which were used to analyse ordinary differential equations (see  Section \ref{sec:k-contact-geometry} for a review on the results in \cite{LRS_25} on this matter). This allows for our definition of $k$-contact Lie systems.  Relevantly, \cite{LRS_25} introduced  {\it $k$-contact distributions}, namely distributions given locally as the kernel of a $k$-contact form, which allowed for a deeper insight into $k$-contact geometry. In particular, a $k$-contact distribution is a generalisation of contact distributions, which are recovered for $k=1$. This work shows that this $k$-contact distributional approach is very practical for determining $k$-contact Lie systems with practical applications. The idea is that obtaining a $k$-contact form $\bm\eta$ turning a Lie system into an $\bm \eta$-Hamiltonian system is difficult, but it is simpler by using $k$-contact distributions. In fact, we show that many control systems, e.g. appearing in \cite{Ram_02}, can be studied via $k$-contact Lie systems by using a quite powerful method developed in Section \ref{Sec:NewMethods}. In particular, our ideas can also be used to understand Lie systems as contact Lie systems in a simple manner. We also consider the relation of $k$-contact Lie systems with other Lie systems admitting a compatible Vessiot--Guldberg Lie algebra of Hamiltonian vector fields relative to a $k$-symplectic or presymplectic form.

It is remarkable that the use of $k$-contact geometry techniques for the study of $k$-contact Lie systems show new uses of features appearing in $k$-contact geometry. In particular, $\bm\eta$-Hamiltonian $k$-functions  play a role in the study of quantities that are constants of the motion, generators of constants of motion of order $m$, or master symmetries for $k$-contact Lie systems. It is worth noting that it is the first time that generators of constants of motion and associated master symmetries are analysed for Lie systems (see \cite{Fer_93,CFR_13} for works on such notions on symplectic and Poisson geometry). As a byproduct, we obtain an extension of contact dissipated quantities to the $k$-contact realm \cite{Lop_24}.

As an application, we show several Lie systems that admit compatible $k$-contact manifolds turning them into $k$-contact Lie systems. These examples include extensions of classical equations like the Riccati equation, and novel systems such as the complex Schwarz equation, higher-order generalisations of Brockett systems, a front-wheel driven kinematic car, and so on. Theoretical results on $k$-contact Lie systems are used to analyse their generators of order $m$ of constants of motion, constants of motion, and our results that are applied to the examined examples. It is worth noting that the study of master symmetries and generalised constants of motion of order $m$ are introduced for the first time in the realm of Lie systems.

We prove that the diagonal prolongation of a $k$-contact Lie system on $N$ to $N^s$ is a $ks$-contact Lie system. This solves a problem appearing in the Poisson coalgebra method for contact manifolds, where it was observed that the diagonal prolongations of contact Lie systems  are not, in general, contact Lie systems as they may be defined on even-dimensional manifolds. The fact that the diagonal prolongation of a contact Lie system may not be a contact Lie system posed a problem to develop a contact Poisson coalgebra method to derive superposition rules and constants of motion for contact Lie systems, which was partially solved using Jacobi geometry. In the case of $k$-contact Lie systems, the use of Jacobi manifolds is neither needed nor available as $k$-contact manifolds are not in general related to Jacobi manifolds.

In the last part of the work, we briefly recall the theory of PDE Lie systems \cite{CL_11} and show that a Lie algebra of $\bm \eta$-Hamiltonian vector fields relative to a co-oriented $k$-contact manifold $(M,\bm\eta)$ allows us to construct PDE Lie systems with a compatible $k$-contact manifold, the so-called {\it $k$-contact PDE Lie systems}. Some of these PDE Lie systems can be understood as Hamilton--De Donder--Weyl equations in the $k$-contact realm, which is illustrated by examples considered in Section \ref{sec:PDE_Lie_systems} and Theorem \ref{thm:eta_Hamiltonian_k_vector_fields}.  It is worth recalling that there are not many applications of PDE Lie systems in the literature \cite{Ram_02,GL19,GL_25} and our new applications are specially interesting due to this fact. These new applications are concerned with Floquet theory, $\mathfrak{g}$-structures, Lax pairs, while PDE Lie systems on Lie groups, Wess--Zumin--Witten--Novikov system, were previously studied \cite{GL19,CGL_19}. 
We briefly describe the potential extensions of our techniques for $k$-contact Lie systems to $k$-contact PDE Lie systems.

The structure of the work goes as follows. Section \ref{Sec:LieSystems} presents the basics on Lie systems and related notions. $k$-Vector fields and their integral submanifolds are presented in Section \ref{Sec:k-vectorfields}, while $k$-contact geometry is briefly presented in Section \ref{sec:k-contact-geometry}. Then, a theory of $k$-contact Lie systems is derived in Section \ref{Sec:kcontactLiesystem}, while a theory on $t$-dependent constants of motion for $k$-contact Lie systems and other related notions is developed in Section \ref{Sec:Constants}. The relation of $k$-contact Lie systems and other geometric structures is given in Section \ref{Sec:kcontactAndOthers}. Methods to construct $k$-contact Lie systems from Lie systems are introduced in Section \ref{Sec:NewMethods}.  In Section \ref{Sec:Diagonalkcontact} we show that the diagonal prolongation to $N^s$ of a $k$-contact Lie system on $N$ is a $ks$-contact Lie system. Afterwards, applications of $k$-contact Lie systems appear in Section \ref{Sec:applications}. The theory of PDE Lie systems, some new and known examples, as well as PDE Lie systems admitting a Vessiot--Guldberg Lie algebra of Hamiltonian vector fields relative to a $k$-contact form are introduced in Section \ref{sec:PDE_Lie_systems}. Examples of these $k$-contact PDE Lie systems with potential applications are presented. Finally, the conclusions and outlook of our work are presented in Section \ref{Sec:ConclusionsOutlook}. 

\section{Basics on Lie systems}\label{Sec:LieSystems}
Let us establish some basic definitions about Lie systems  and other related concepts to be used in this work \cite{CL_11}. Natural numbers start at one. Hereafter, manifolds are assumed to be smooth, Hausdorff, connected, and finite-dimensional, unless otherwise stated. In particular, $M$ is an $m$-dimensional manifold an $U$ is always an open subset of a manifold. Our further considerations are mainly local at generic points, and the problem of establishing whether structures are globally defined will be generally skipped. We will also focus on smooth functions. Moreover, $\{e_1,\ldots,e_k\}$ is a basis for $\mathbb{R}^k$. 
 From now on, $\pi : \T M \rightarrow M$ stands for the canonical tangent bundle projection and we set $\pi_{2} : \mathbb{R} \times M \ni (t,x) \mapsto x \in M$. Given two subsets $A_1, A_2$ of a Lie algebra $\mathfrak{g}$, we represent by $[A_1,A_2]$ the vector space spanned by the Lie brackets between the elements of $A_1$ and $A_2$, respectively. We say that vector fields $X_1, \ldots, X_r$ are \textit{linearly independent at a generic point} if $\sum_{\alpha=1}^rf_\alpha X_\alpha =0
 $ on some $U$ only for functions $f_1=\ldots =f_r=0$ on $U$.

It is essential in the theory of Lie systems to describe $t$-dependent systems of first-order ordinary differential equations in normal form as $t$-dependent vector fields. 
A \textit{t-dependent vector field} is a map $X : \mathbb{R} \times M \ni (t,x) \mapsto X(t,x) \in \T M$ such that $\pi \circ X= \pi_{2}$.
The latter amounts to the fact that each $t$-dependent vector field $X$ on a manifold $M$ is equivalent to a family $\{X_{t}\}_{t \in \mathbb{R}}$ of vector fields on $M$. 

 We call {\it smallest Lie algebra} of a $t$-dependent vector field $X$ the smallest Lie algebra (in the sense of inclusion) containing all the vector fields $\{X_t\}_{t\in \mathbb{R}}$. We denote the smallest Lie algebra of $X$ by $V^X$. Every $t$-dependent vector field $X$ on $M$ gives rise to a generalised distribution
$
\mathcal{D}^X_x=\{Y_x\mid Y\in V^X\},$ for every $x\in M$. 
For simplicity, we simply call generalised distributions. Given a vector field $Z$ on $M$, we write $[Z,\mathcal{D}]$ for the distribution spanned by the Lie brackets of $Z$ with vector fields taking values in $\mathcal{D}$. Although this notation conveys a little abuse of notation, its meaning is clear and simplifies the presentation.

Given a  $t$-dependent vector field $X$ on $M$, its \textit{associated system} is the system of first-order ordinary differential equations
\begin{equation}
\label{eq:integral_curve}
    \frac{\d x}{\d t} = X(t,x)\,, 
\end{equation}
whose solution $x : \mathbb{R} \rightarrow M$ with $x(0)=x_0$ is the \textit{integral curve} of $X$ with initial condition $x_0$ at $t=0$. Every $t$-dependent vector field has an associated $t$-dependent system of differential equations in normal form \eqref{eq:integral_curve} and vice versa. It is convenient then to use $X$ to refer both to a $t$-dependent vector field and the $t$-dependent system of ordinary differential equations determining its integral curves.

In order to introduce Lie systems, consider 
\begin{equation}
    \frac{\d x}{\d t}=b_{1}(t)+b_{2}(t)x+b_{3}(t)x^2\,, \label{eq5}
\end{equation}
where $x \in \mathbb{R}, t\in \mathbb{R}$ and $b_{1}(t),b_{2}(t),b_{3}(t) $ are arbitrary $t$-dependent functions. Differential equations of this type are called {\it Riccati equations}\footnote{It is sometimes assumed that $b_{1}(t)b_{3}(t)\neq0$, but this condition is frequently ignored in Lie systems theory \cite{CL_11,LS_20}.}. The general solution of \eqref{eq5} can be written as
\begin{equation}
    x(t)=\frac{x_{(1)}(t)(x_{(3)}(t)-x_{(2)}(t)) - kx_{(2)}(t)(x_{(3)}(t)-x_{(1)}(t))}{x_{(3)}(t)-x_{(2)}(t) - k(x_{(3)}(t)-x_{(1)}(t))}\,, \label{solutin1}
\end{equation}
where $x_{(1)}(t),x_{(2)}(t),x_{(3)}(t)$ are three different particular solutions of \eqref{eq5} and $k  \in \mathbb{R}$.

    More generally, a $t$-dependent system of first-order ordinary differential equations on $M$ of the form \eqref{eq:integral_curve}
for a $t$-dependent vector field $X$ on $M$ is called a \textit{Lie system}, if it admits a \textit{superposition rule}, namely a $t$-independent map $\Phi : M^{\ell} \times M \rightarrow M$ of the form
$\label{eq7}
    x=\Phi(x_{(1)},\ldots,x_{(\ell)};k)\,,
$
such that every generic solution of \eqref{eq:integral_curve} can be written as
\begin{equation}\label{eq:Superposition_rule}
     x(t)=\Phi(x_{(1)}(t),\ldots,x_{(\ell)}(t);k)\,,
\end{equation}
where $x_{(1)}(t),\ldots,x_{(\ell)}(t)$ is any generic family of particular solutions of system  \eqref{eq:integral_curve} and $k \in M$.

There are some subtle problems related to the above definition and the meaning of being a ``generic solution'' or a ``generic family''. They are mostly technical  (see \cite{BM_10,CL_11}), therefore we will omit discussing them in detail. We will simply say that the domain of $\Phi$ is, in reality, some open subset of $M^\ell\times M$, and, for every generic family of particular solutions, one can recover locally almost every solution \cite{CL_11,LS_20}.

 The necessary and sufficient conditions to characterise Lie systems were provided by Sophus Lie in the \textit{Lie Theorem} \cite{Lie3,Lie2,Lie1}.
\begin{theorem}[Lie Theorem]\label{Lie_theorem}
A $t$-dependent system $X$ on a manifold $M$ admits a superposition rule \eqref{eq:Superposition_rule} if, and only if, $X$ can be written as
\begin{equation}
    X(t,x)=\sum^{r}_{\alpha=1}b_{\alpha}(t)X_{\alpha}(x), \label{vec2} \qquad t\in \R, \quad x \in M,
\end{equation}
where $X_{1}, \ldots, X_{r}$ is a family of vector fields on $M$ spanning an r-dimensional real Lie algebra of vector fields $V$, a so called \textit{Vessiot--Guldberg Lie algebra} of $X$ on $M$, and $b_{1}(t),\ldots,b_{r}(t)$ are arbitrary $t$-dependent functions.
\end{theorem}
Lie proved that the existence of a superposition rule depending on $\ell$ particular solutions implies the existence of a Vessiot--Guldberg Lie algebra $V$ with a basis $X_{1},\ldots, X_{r}$ such that $r\leq \dim M \cdot \ell $, and, conversely, given a decomposition \eqref{vec2}, there exists a superposition rule depending on $\ell$ particular solutions with $r\leq \dim M \cdot \ell$. The inequality $r\leq \ell\cdot \dim M$ is referred to as \textit{Lie's condition} \cite{CL_11}.
 
A Riccati equation describes the integral curves of a $t$-dependent vector field on $\mathbb{R}$ of the form 
\begin{equation}
    X(t,x)=(b_{1}(t) + b_{2}(t)x + b_{3}(t)x^{2})\frac{\partial}{\partial x}\,. \label{vec4}
\end{equation}
One can present $X$ as a linear combination of three vector fields on $\mathbb{R}$ given by
\begin{equation}
X_{1}=\frac{\partial}{\partial x}, \qquad X_{2}=x\frac{\partial}{\partial x}, \qquad X_{3}=x^{2}\frac{\partial}{\partial x}\,,  \label{vec5}
\end{equation}
which satisfy the commutation relations
\begin{equation}
\label{eq:comm_rel_RIcc}
    [X_{1}, X_{2}] = X_{1}\,, \qquad [X_{1}, X_{3}] = 2X_{2}\,, \qquad [X_{2}, X_{3}] = X_{3}\,.
\end{equation}
By the Lie Theorem, it follows that Riccati equations admit a superposition rule. Conversely, one can also say that the existence of a superposition rule \eqref{solutin1} implies that \eqref{vec4} is such that $\{X_t\}_{t\in \mathbb{R}}$ are included in some Lie algebra of dimension three or less. 
In view of \eqref{eq:comm_rel_RIcc}, the vector fields \eqref{vec5} span a Lie algebra isomorphic to $\mathfrak{sl}(2,\mathbb{R})$, which is three-dimensional \cite{CR_02}. Since the superposition rule \eqref{solutin1} depends on three particular solutions, Riccati equations satisfy Lie's condition.

Let us recall the analysis of Lie systems via contact geometry, which gives rise to contact Lie systems \cite{LR_23}. 

\begin{definition}
    A {\it contact Lie system} is a triple $(M,\eta,X)$, where $\eta$ is a contact form on $M$ and $X$ is a Lie system on $M$ admitting a Vessiot--Guldberg Lie algebra of contact Hamiltonian vector fields relative to $\eta$. A {\it contact Lie system} is called \textit{conservative} or of \textit{Liouville type} if the Reeb vector field of $(M,\eta)$ is a Lie symmetry of $X$.
\qeddiamond\end{definition}

The term `conservative' or of `Liouville type' is coined because such contact Lie systems preserve a volume form like in the Liouville theorem in symplectic geometry \cite{LR_23}.
The contact Hamiltonian function of a contact Hamiltonian vector field is a first integral of the Reeb vector field if, and only if, the contact Hamiltonian vector field commutes with the Reeb vector field. Hence, a contact Lie system of Liouville type amounts to a Lie system $X$ on $M$ admitting a smallest Lie algebra $V^X$ of contact Hamiltonian vector fields relative to a contact form on $M$ that commute with the Reeb vector field, $R$, of the  contact form, namely $R$ is a Lie symmetry of $V^X$.

In view of Lie Theorem, a Lie system $X$ can be considered as a curve in $V^X$ parametrised by $t$. Every contact form on a manifold $M$ gives rise to an isomorphism mapping each contact Hamiltonian vector field on $M$ to a contact Hamiltonian function on $M$ and vice versa (see Corollary \ref{Cor:MorLieBrackett} for $k=1$). Therefore, the Lie algebra of contact Hamiltonian vector fields $V^X$ gives rise to an isomorphic Lie algebra of functions $\mathfrak{W}$ relative to the bracket on $\Cinfty(M)$ induced by $\eta$. Hence, $X$ also defines a curve in $\mathfrak{W}$ parametrised by $t$. This suggests us the following definition.

\begin{definition} A {\it contact Lie--Hamiltonian system} is a triple $(M,\eta,h:\mathbb{R}\times M\rightarrow \mathbb{R})$, where $(M,\eta)$ is a co-oriented contact manifold and $h$ gives rise to a $t$-dependent family of functions $h_t:x\in M
\mapsto h(t,x)\in \mathbb{R}$, with $t\in \mathbb{R}$, contained in a finite-dimensional Lie algebra of functions relative to the Lie bracket in $\Cinfty_\eta(M,\mathbb{R})$ induced by $(M,\eta)$: a {\it contact Lie--Hamilton algebra}. The function $h$ is called a {\it contact Lie--Hamiltonian} relative to $\eta$. 
\qeddiamond\end{definition}

It is worth noting that every contact Lie system related to $\eta$ gives rise a unique contact Lie--Hamiltonian system related to the same $\eta$ ` and vice versa.

\section{\texorpdfstring{$k$}--Vector fields and its integral sections}\label{Sec:k-vectorfields}

$k$-Vector fields are of great use in the geometric study of systems of partial differential equations \cite{LSV_15,RRSV_11}. Consider the Whitney sum
$\label{eq:Whitney-sum}
    \boplus^k_M\T M := \T M\oplus_M\overset{(k)}{\dotsb}\oplus_M \T M\,,
$ with the natural projections\footnote{From now on, the subindex in  $\oplus_M$ in the Whitney sum will be skipped.}
$
    \tau^\alpha\colon\boplus^k\T M\to\T M\,,\; \tau^k_M\colon\boplus^k\T M\to M\,, \; \alpha=1,\ldots,k\,,
$ where $\tau^\alpha$ denotes the projection onto the $\alpha$-th component of the Whitney sum $\boplus^k_M\T M $, and $\tau^k_M$ is the projection onto the base manifold $M$.
\vskip -0.1cm
\noindent\begin{minipage}{0.75\linewidth}
    A \textit{k-vector field} on \( M \) is a section \( X: M \to \bigoplus^k TM \) of the projection \( \tau^k_M \). 
    The space of \( k \)-vector fields on \( M \) is denoted by \( \mathfrak{X}^k(M) \). Taking into account the diagram aside, a \( k \)-vector field \( \mathbf{X} \in \mathfrak{X}^k(M) \) amounts to a family of vector fields \( X_1, \dots, X_k \in \mathfrak{X}(M) \) given by 
\( X_\alpha = \tau^\alpha \circ \mathbf{X} \) with
\end{minipage}
\hfill
\begin{minipage}{0.29\linewidth}
    \vskip -0.2cm
    \[
    \begin{tikzcd}[every cell/.append style={scale=1.9}]
        & \bigoplus^k TM \arrow[d, "\tau^\alpha"] \\
        M \arrow[ur, "X"] \arrow[r, "X_\alpha"] & TM
    \end{tikzcd}
    \]
\end{minipage}

\noindent
 \( \alpha = 1, \dots, k \). With this in mind, one can denote \( \mathbf{X} = (X_1, \dots, X_k) \). 
A \( k \)-vector field \( \mathbf{X} \) induces a decomposable contravariant totally skew-symmetric tensor field, \( X_1 \wedge \dots \wedge X_k \), which is a decomposable section of the bundle 
\( \bigsqcup_{x \in M} \Lambda^k T_x M = \Lambda^k TM \to M \), where $\bigsqcup_{x\in M}$ is the disjoint sum over $x\in M$.

    Given a mapping $\phi\colon U\subset\R^k\to M$, its {\it first prolongation} to $\bigoplus^k\T M$ is the map $\phi'\colon U\subset\R^k\to\bigoplus^k\T M$ 
    defined as follows
    $$
        \phi'(t) = \left( \phi(t); \T_t\phi\left( \parder{}{t^1}\bigg\vert_t \right),\dotsc,\T_t\phi\left( \parder{}{t^k}\bigg\vert_t \right) \right) := (\phi(t); \phi'_{1}(t),\ldots,\phi'_{k}(t))\,, \qquad t\in\R^k\,,
    $$
    where $t=(t^1,\dotsc,t^k)$, and $\{t^1,\ldots,t^k\}$ are the canonical coordinates on $\R^k$.

As for integral curves of vector fields, one can define integral sections of a $k$-vector field as follows. Let $\bfX = (X_1,\dotsc,X_k)\in\X^k(M)$ be a $k$-vector field. An {\it integral section} of $\bfX$ is a map $\phi\colon U\subset\R^k\to M$ such that $\phi' = \bfX\circ\phi\,$, namely $\T\phi \left(\parder{}{t^\alpha}\right) = X_\alpha\circ\phi$ for $\alpha=1,\ldots,k$. A $k$-vector field $\bfX\in\X^k(M)$ is said to be {\it integrable} if every point of $M$ lies in the image of an integral section of $\bfX$.

Let $\bfX = (X_1,\ldots, X_k)$ be a $k$-vector field on $M$ with local expression $ X_\alpha = \sum_{i=1}^{m}X_\alpha^i\parder{}{x^i}\,$ for $\alpha=1,\ldots ,k$.
Then, $\phi\colon U\subset\R^k\to M$, with local expression $\phi(t) = (\phi^i(t))$, is an integral section of $\bfX$ if, and only if, its components satisfy the following system of PDEs
\begin{equation}\label{eq:SysPDEs}
	\parder{\phi^i}{t^\alpha} = X_\alpha^i\circ\phi\,,\qquad i=1,\ldots,m\,,\qquad \alpha=1,\ldots,k\,.
\end{equation}

Then, $\bfX$ is integrable if, and only if, $[X_\alpha,X_\beta] = 0$ for $\alpha,\beta=1,\ldots,k$. These are necessary and sufficient conditions for the integrability of the system of PDEs \eqref{eq:SysPDEs} (see \cite{Lee_12,Ol_93} for details).

Every $k$-vector field $\bfX = (X_1,\dotsc,X_k)$ on $M$ defines a distribution $\mathcal{D}^\bfX\subset\T M$ given by $\mathcal{D}^\bfX_x = \mathcal{D}^\bfX \cap\T_xM = \langle X_1(x),\dotsc,X_k(x)\rangle$. However, the notion of an integral submanifold of the $k$-vector field $\bfX$ is stronger than the notion of an integral section of the distribution $\mathcal{D}^\bfX$. The distribution $\mathcal{D}^\bfX$ is integrable if, and only if, $[X_\alpha,X_\beta] = \sum_{\gamma=1}^{k}f_{\alpha\beta}^\gamma X_\gamma$, with $\alpha,\beta=1,\ldots,k$ for certain functions $f_{\alpha\beta}^\gamma\in \Cinfty(M)$ with $\alpha,\beta,\gamma=1,\ldots,k$, and $\mathcal{D}^{\bfX}$ is invariant relative to the one-parameter group of diffeomorphisms of any vector field taking values in $\mathcal{D}^{\bfX}$ \cite{Lav_18,Ste_74,Sus_73}. On the other hand, the $k$-vector field $\bm X$ is integrable if, and only if, $X_1,\ldots, X_k$ commute with each other, which is a stronger condition.

 \section{\texorpdfstring{$k$}{k}-contact geometry}
 \label{sec:k-contact-geometry}

 This section summarises previous known results about $k$-contact manifolds \cite{Riv_21} and a new approach, via distributions, devised in the preprint \cite{LRS_25}. We have included just some proofs of our previous results in \cite{LRS_25} for completeness. 

Let us start by the classical definition of a $k$-contact form on an open subset \cite{Riv_21}, whose motivation was to generalise the concept of a contact form. 
 
\begin{definition}\label{dfn:k-contact-manifold}
    A \textit{$k$-contact form on an open subset $U\subset M$} is a differential one-form on $U$ taking values in $\mathbb{R}^k$, let us say $\bm\eta\in\Omega^1(U,\mathbb{R}^k)$, such that
    \begin{enumerate}[(1)]
        \item $\ker \bm\eta\subset\T U$ is a regular non-zero distribution of corank $k$,
        \item $\ker \d\bm\eta\subset\T U$ is a regular distribution of rank $k$,
        \item $\ker \bm\eta\cap\ker \d\bm\eta  = 0$.
    \end{enumerate}
    If the $k$-contact form $\bm \eta$ is defined on $M$, the pair $(M,\bm\eta)$ is called a \textit{co-oriented $k$-contact manifold} and $\ker\d\bm\eta$ is called the {\it Reeb distribution} of $(M,\bm\eta)$. 
    \qeddiamond\end{definition}
    Nowadays, the only approach to $k$-contact Hamiltonian dynamics is constructed exclusively for co-oriented $k$-contact manifolds \cite{Riv_21,LRS_25}. Reeb vector fields for $k$-contact forms are the natural analogue of Reeb vector fields for contact forms, which play a relevant role in contact geometry and the study of periodic orbits of Hamiltonian systems \cite{CH_13,Pas_12}. 

\begin{theorem}[Reeb vector fields \cite{GGMRR_20}]\label{thm:k-contact-Reeb}
    Let $(M,\bm\eta = \sum_{\alpha=1}^k\eta^\alpha\otimes e_\alpha)$ be a co-oriented $k$-contact manifold. There exists a unique family of vector fields $R_1,\dotsc, R_k\in\X(M)$ such that
    \begin{equation*}\label{eq:k-contact-Reeb}
    \inn{R_\alpha}\eta^\beta = \delta_\alpha^\beta, \qquad \inn{R_\alpha}\d\eta^\beta = 0\,,    
    \end{equation*}
    for $\alpha,\beta = 1,\dotsc,k$. The vector fields $R_1,\dotsc,R_k$ commute between themselves, i.e.
    $
        [R_\alpha,R_\beta] = 0$ for $\alpha,\beta = 1,\dotsc,k\,,
    $  and  $\ker \d\bm\eta = \langle R_1,\dotsc,R_k \rangle$. 
\end{theorem}

As shown later, it is generally difficult to find a 
$k$-contact forms that transform the vector fields of a Vessiot–Guldberg Lie algebra for a Lie system into 
$k$-contact Hamiltonian vector fields, thereby enabling the application of various methods to study these vector fields (cf. \cite{LRS_25,Riv_21}). Instead, we will devise new methods to obtain such $k$-contact forms via $k$-contact distributions to be defined next (see \cite{LRS_25} for details). It is worth noting that $k$-contact distribution play the role of contact distributions in contact geometry, and they are easy to employ, practically and theoretically, to study $k$-contact manifolds \cite{LRS_25}. In particular, we will show in this paper how $k$-contact distributions are simpler to use to study dynamical systems via $k$-contact geometry.

\begin{definition}\label{def:k-contact-distribution}
    A {\it $k$-contact distribution} on $M$ is a distribution $\mathcal{D}\subset\T M$ such that, for each point $x\in M$, there is an open neighbourhood $U\ni x$ and a $k$-contact form $\bm\eta$ on $U$ such that $\restr{\mathcal{D}}{U} = \ker\bm\eta$. We say that $(M,\mathcal{D})$ is a {\it $k$-contact manifold}.
\qeddiamond\end{definition}

Note that  $k$-contact distributions are regular because the kernel of $k$-contact forms is regular too.  
Let us now develop the notions of $\bm\eta$-Hamiltonian vector fields and $\bm\eta$-Hamiltonian $k$-functions, which will be very useful in the description of Lie systems in Section \ref{Sec:applications}. 

The contraction of a vector field $X$ with a differential one-form taking values in $\mathbb{R}^k$, e.g. $\bm \eta=\sum_{\alpha=1}^k
\eta^\alpha\otimes e_\alpha$,  is defined component-wise, namely $\iota_X\bm \eta=\sum_{\alpha=1}^k\iota_X\eta^\alpha\otimes e_\alpha$. The Lie bracket or contraction of $X$ with a differential $s$-form taking values in $\mathbb{R}^k$ is defined similarly. In order to characterise $k$-distributions without determining a $k$-contact form, which can be challenging, we use Lie symmetries of distributions to be defined below. Moreover, the concept gives a natural manner to generalise to $k$-contact geometry the theory of contact Hamiltonian vector fields. 

\begin{definition}
    A {\it Lie symmetry} of a distribution $\mathcal{D}$ on $M$ is a vector field $X$ on $M$ such that $[X,\mathcal{D}]\subset \mathcal{D}$, where $[X,\mathcal{D}]$ stands for the distribution generated by the Lie brackets of $X$ with vector fields taking values in $\mathcal{D}$.
\label{def:etaHamiltonianvectorfield}A {\it $k$-contact vector field} relative to a $k$-contact manifold $(M,\mathcal{D})$ is a Lie symmetry $X$ of $\mathcal{D}$. If additionally $\mathcal{D}=\ker \bm \eta$ for a $k$-contact form $\bm \eta$, then $X$ is called an {\it $\bm\eta$-Hamiltonian vector field} and $-\inn{X}{\bm \eta}$ is its {\it $\bm\eta$-Hamiltonian $k$-function}. 
\qeddiamond\end{definition}

Denote by $\X_\mathcal{D}(M)$ the space of $k$-contact  vector fields on $M$ relative to $(M,\mathcal{D})$ and we write $\X_{\bm\eta}(M)$ for the space of $\bm\eta$-Hamiltonian vector fields relative to a co-oriented $k$-contact manifold $(M,\bm\eta)$, respectively. Meanwhile, $\Cinfty_{\bm\eta}(M,\mathbb{R}^k)$ stands for the space of $\bm\eta$-Hamiltonian $k$-functions relative to a $k$-contact form $\bm\eta$. 

The following proposition explains why $-\inn{X}{\bm \eta}$ plays the role of an $\bm\eta$-Hamiltonian $k$-function for an $\bm\eta$-Hamiltonian vector field $X$. 

\begin{proposition}
\label{prop:eta-Hamiltonian-vector-field}
A vector field $X$ on $M$ is an $\bm\eta$-Hamiltonian vector field relative to $(M,\mathcal{D}=\ker\bm\eta)$ if, and only if, 
\begin{equation}\label{eq:ConHam}
    \inn{X}\eta^\alpha=-h^\alpha\,,\qquad \inn{X}\d\eta^\alpha=\d h^\alpha-\sum_{\beta=1}^k(R_\beta h^\alpha)\eta^\beta\,,\qquad \alpha=1,\ldots,k\,,
\end{equation}
for some $k$-function ${\bm h} = \sum_{\alpha=1}^k h^\alpha  e_\alpha\in\Cinfty(M,\R^k)$.
\end{proposition}

Similarly, $k$-contact Hamiltonian $k$-vector fields can be defined as follows. It is worth noting that  $\sum_{\beta=1}^k(R_\beta h^
\alpha)\eta^\beta$, for every $\alpha$, is a linear combination of the one-forms$\eta^1,\ldots, \eta^k$ with coefficients given by the functions $R_\beta h^\alpha$.

\begin{definition} Given  $h\in \Cinfty(M)$ on a co-oriented $k$-contact manifold $(M,\bm\eta)$, called a \textit{$\bm \eta$-Hamiltonian function}, a $k$-vector field  $\bfX^c_{h} = (X_\alpha)\in\X^k(M)$ satisfying the equations
\begin{equation}\label{eq:contact-Ham-equations-trivial-1}
    \sum_{\alpha=1}^k\inn{X_\alpha}\d\eta^\alpha = \d h - \sum_{\alpha=1}^k(R_\alpha h)\eta^\alpha\,,\qquad \sum_{\alpha=1}^k\inn{X_\alpha}\eta^\alpha = -h\,,
\end{equation}
 is  called an \textit{$\bm \eta$-Hamiltonian $k$-vector field}. Equations \eqref{eq:contact-Ham-equations-trivial-1} can be rewritten as
\begin{equation}\label{eq:contact-Ham-equations-trivial-2}
    \sum_{\alpha=1}^k\Lie_{X_\alpha}\eta^\alpha = -\sum_{\alpha=1}^k(R_\alpha h)\eta^\alpha\,,\qquad \sum_{\alpha=1}^k\inn{X_\alpha}\eta^\alpha = -h\,.    
\end{equation}
We call $(M,\bm\eta,h)$   a \textit{$k$-contact Hamiltonian system}.
\qeddiamond\end{definition}

Not every element in $\Cinfty(M,\mathbb{R}^k)$ is related to an $\bm{\eta}$-Hamiltonian vector field, e.g. $\bm \eta=(\d z-p\d x)\otimes e_1 +\d w\otimes e_2$ on $\mathbb{R}^4$ has not $\bm\eta$-Hamiltonian two-function of the form $pe_1+pe_2$. But every function  
$h\in \Cinfty(M)$ gives rise to a one-form $\d h- \sum_{\alpha=1}^k(R_\alpha h)\eta^\alpha $ belonging to the annihilator of the Reeb distribution. If the codistribution $\langle \d\eta^1,\ldots,\d\eta^k\rangle $ is not equal to the annihilator of the Reeb distribution, then $\ker \d\bm \eta \cap \ker \bm \eta\neq0$, which is a contradiction. Hence, there exists a series of vector fields $X_1,\ldots,X_k$   such that $\sum_{\alpha=1}^r\iota_{X_\alpha}\d\eta^\alpha= \d h- \sum_{\alpha=1}^k(R_\alpha h)\eta^\alpha $. If $g=\sum_{\alpha=1}^k\iota_{X_\alpha}\eta^\alpha$, then, $(X_1-(h+g)R_1,X_2,\ldots,X_k)$ is an $\bm \eta$-Hamiltonian $k$-vector field related to $h$.

Let us now provide some notation  \cite{LRS_25}.

\begin{definition} Given a co-oriented $k$-contact manifold $(M,\bm \eta)$, every $k$-contact Hamiltonian $k$-function ${\bm h}\in \Cinfty_{\bm\eta}(M,\mathbb{R}^k)$ is related to its {\it Reeb derivation}, namely the vector field on $M$ of the form
$$
R_{\bm h}=\sum_{\alpha=1}^kh^\alpha R_\alpha.
$$
Moreover, let us define the so-called 
{\it Reeb tensor field} $\mathfrak{R}_{\bm\eta}:\cT M\rightarrow \cT M$ satisfying
$$
\mathfrak{R}_{\bm \eta}\theta=\sum_{\alpha=1}^k\eta^\alpha \iota_{R_\alpha} \theta,\qquad \forall \theta\in \cT M.
$$
\end{definition}
 
The advantages of the previous notation are illustrated by the conciseness of the expressions in following propositions. In particular, the following proposition will be useful to demonstrate some properties of $k$-contact Lie systems (we refer to \cite{LRS_25} for details). 

\begin{proposition}\label{Prop:Calculation}
Let $X_{\bm f}$ be the $\bm\eta$-Hamiltonian vector field of  $\bm f\in \Cinfty_{\bm\eta}(M,\mathbb{R}^k)$. Then,
\begin{enumerate}[(i)]
    \item $\Lie_{X_{\bm f }}\bm\eta  = -\sum_{\alpha,\beta=1}^{k}(R_\beta f^\alpha)\eta^\beta\otimes e_\alpha=-\mathfrak{R}_{\bm \eta}\d \bm f $,
    \item $X_{\bm f}\bm f = -\sum_{\alpha,\beta=1}^kf^\beta(R_\beta f^\alpha)  e_\alpha=-R_{\bm f}\bm f$,
    \item $[R_\beta,X_{\bm f}]=X_{ R_\beta {\bm f}},$ 
\end{enumerate}
where $R_\beta {\bm f}={\sum_{\alpha=1}^{k} R_\beta f^\alpha  e_\alpha}$ and $\beta=1,\ldots,k$. 
\end{proposition}

The following proposition follows from expression \eqref{eq:ConHam} and the fact that an $\bm \eta$-Hamiltonian vector field is a vector field that leaves invariant a $k$-contact distribution.
\begin{proposition}\label{Prop:EtaHam}
The space of $\bm\eta$-Hamiltonian vector fields
$\mathfrak{X}_{\bm \eta}(M)$ relative to a $k$-contact form $\bm \eta$ is a Lie algebra and, for $ {\bm h}, {\bm g}\in \Cinfty_{\bm \eta}(M,\mathbb{R}^{k})$, one has that $[X_{\bm h},X_{\bm g}]$ is the $\bm \eta$-Hamiltonian vector field related to the $\bm \eta$-Hamiltonian $k$-function $-\inn{[X_{\bm h},X_{\bm g}]}\bm\eta$.
\end{proposition}

In a view of the Proposition \ref{Prop:EtaHam}, one can define a following Lie bracket on $\Cinfty_{\bm \eta}(M,\mathbb{R}^{k})$.
\begin{definition}
    Given a co-oriented $k$-contact manifold $(M,\bm \eta)$, we define a bracket on $\Cinfty_{\bm \eta}(M,\mathbb{R}^{k})$ of the form \begin{equation}\label{eq:kcontactbracket}
        \{\bm h_1,\bm h_2\}_{\bm\eta} =\bm\eta([X_{\bm h_1},X_{\bm h_2}])\,,\qquad \forall \bm h_1,\bm h_2\in \Cinfty_{\bm \eta}(M,\mathbb{R}^k)\,.
\end{equation}\qeddiamond\end{definition}
The result below  immediately follows from Proposition \ref{Prop:EtaHam}  (see also \cite{LRS_25}). 
\begin{corollary}\label{Cor:MorLieBrackett}
    The bracket \eqref{eq:kcontactbracket}  induces a Lie algebra isomorphism $\phi:\bm h\in \Cinfty_{\bm \eta}(M,\mathbb{R}^{k})\mapsto -X_{\bm h}\in\X_{\bm \eta}(M)$ and an exact sequence of Lie algebra morphisms
$$\label{morph-sequence}
0\longrightarrow \Cinfty_{\bm\eta}(M,\mathbb{R}^k)\stackrel{\phi}{\longrightarrow} \X_{\bm \eta}(M)\longrightarrow 0\,,
$$
where $\Cinfty_{\bm\eta}(M,\mathbb{R}^k)$ is endowed with the Lie bracket \eqref{eq:kcontactbracket}. 
\end{corollary}

Note that
\begin{equation}
\label{eq:Lie_bracket_Ham}
\{{\bm h}_1,{\bm h}_2\}_{\bm \eta}=\bm \eta([X_{{\bm h}_1},X_{{\bm h}_2}])=(\Lie_{X_{\bm h_1}}\iota_{X_{\bm h_2}}-\iota_{X_{\bm h_2}}\Lie_{X_{\bm h_1}})\bm\eta=-\Lie_{X_{{\bm h}_1}}{\bm h_2}-R_{{\bm h}_2}{\bm h}_1,\qquad 
\end{equation}
for all ${\bm h}_1,{\bm h}_2\in \Cinfty_{\bm \eta}(M,\mathbb{R}^k)$
and 
$$
X_{\bm f}{\bm f}=-R_{\bm f}{\bm f},\qquad \forall {\bm f}\in \Cinfty_{\bm \eta}(M,\mathbb{R}^k).
$$
It is remarkable that, in the contact case, a dissipated quantity (see \cite{Lop_24} and references therein) is a function such that $X_hf=-f(Rh)$. Hence, the expression above can be used to define an analogue of dissipated quantities in $k$-contact geometry.
\begin{definition} A {\it dissipated $k$-contact $k$-function} relative to a $k$-contact Hamiltonian system $(M,\bm \eta,\bm h)$ is an  ${\bm \eta}$-contact $k$-function ${\bm f}\in \Cinfty_{\bm \eta}(M,\mathbb{R}^k)$ such that
$$
X_{\bm h}{\bm f}=-R_{\bm f}{\bm h}.
$$
\end{definition}
In particular every ${\bm f}\in \Cinfty_{\bm \eta}(M,\mathbb{R}^k)$ is dissipated relative to itself, namely $X_{
\bm f}{\bm f}=-R_{\bm f}{\bm f}$. Moreover, 
$$
\{{\bm h},{\bm f}\}=0\qquad \Longleftrightarrow   \qquad 
X_{\bm h}\bm f=-R_{\bm f}{\bm h}.
$$
In other words, dissipated quantities are those that $\bm \eta$-commute with ${\bm h}$. Moreover, an ${\bm \eta}$-contact $k$-function is dissipated if, and only if, its associated $\bm\eta$-Hamiltonian vector field is a Lie symmetry of $X_{\bm h}$.

On the other hand, note that
$$
\iota_{X_{\bm h}}\d\bm \eta=\d {\bm h}-\mathfrak{R}_{\bm \eta}{\d\bm h}=({\rm Id}_{\cT M}-\mathfrak{R}_{\bm \eta})(\d{\bm h}),\qquad \iota_{X_{\bm h}}\bm\eta=-\bm h,\qquad \Lie_{X_{\bm h }}\bm\eta=-\mathfrak{R}_{\bm \eta}\d{\bm h}.
$$
It is relevant that \eqref{eq:kcontactbracket} is not a Poisson bracket and ${\bm h}_2$ may be a constant function having a non-vanishing Lie bracket with other functions. Examples of this will be shown in Section \ref{Sec:applications}.

Let us  now study the relation of $k$-contact distributions with the maximal non-integrability notion defined via distributions \cite{Vit_15}.
\begin{definition}
\label{def:max_non_integrable_mapping}
    Let $\mathcal{D}$ be a regular distribution on $M$ and let $\pi\colon\T M\rightarrow \T M/\mathcal{D}$ be the natural vector bundle projection. Then, $\mathcal{D}$  is {\it maximally non-integrable} in a {\it distributional sense} if $\mathcal{D}\neq 0$ and the vector bundle mapping $\rho\colon \mathcal{D}\times_M \mathcal{D}\rightarrow \T M/\mathcal{D}$ over $M$ given by
    \begin{equation}\label{eq:CurvDis}
        \rho(v,v')=\pi([X,X']_x)\,,\qquad \forall v,v'\in \mathcal{D}_x\,,\qquad \forall x\in M\,,
    \end{equation}
    where $X,X'$ are vector fields taking values in $\mathcal{D}$ locally defined around $x$ such that $X_x=v$ and $X'_x=v'$, is non-degenerate.
\qeddiamond\end{definition}

The mapping $\rho$ is well defined as shown in Proposition \ref{Prop:WellDefined} in the Appendix or \cite{LRS_25}.
It is worth noting that the terms ``maximally non-integrable" and ``completely non-integrable" have been used in the literature with equal, different, or even related meanings  \cite{Ada_21,BH_16,LRS_25}. 

\begin{proposition}\label{Prop:MaxNonkCon}
    If $(M,\bm\eta)$ is a co-oriented $k$-contact manifold, then $\d\bm\eta$ is non-degenerate when restricted to $\ker \bm \eta$. In other words, $\ker\bm\eta$ is maximally non-integrable.
\end{proposition}

Let us recall the main result  in \cite{LRS_25} used in this work.
\begin{theorem}\label{Prop:LocalEqu} 
    A distribution $\mathcal{D}$ on $M$ is a $k$-contact distribution if, and only if, it is maximally non-integrable and, around an open neighbourhood $U$ of every $x\in M$, admits an integrable $k$-vector field ${\bf S} = (S_1,\ldots, S_k)$ of Lie symmetries of $\mathcal{D}|_U$ such that
    \begin{equation}\label{eq:Dec}
        \langle S_1,\ldots,S_k\rangle \oplus \mathcal{D}|_U = \T U\,.
    \end{equation}
\end{theorem}

The proof of the previous theorem is indeed the method used in our applications to obtain a $k$-contact form $\bm \eta$ turning out the vector fields of a Vessiot--Guldberg Lie algebra into $\bm \eta$-Hamiltonian vector fields. The reason for the use of the previous theorem is that it is difficult to obtain a $k$-contact form compatible with a Lie system, but it is pretty much simpler in many cases to obtain a compatible $k$-contact distribution.  This will be the key in this paper to obtain the applications in control theory shown in \cite{Ram_02}. Indeed, the difference between the abundance of examples in this work and \cite{LR_23} illustrates that it is difficult to obtain Lie systems with a compatible contact or, even more difficult, $k$-contact form. 

\section{\texorpdfstring{$k$}{}-contact Lie systems}\label{Sec:kcontactLiesystem}
Let us introduce $k$-contact Lie systems and determine their fundamental properties. It is worth noting that the notion is a generalisation of contact Lie systems \cite{LR_23}, and it follows the general idea of Lie systems admitting compatible geometric structures \cite{HLT_17,GLMV_19,LS_20}.

\begin{definition}
    A {\it $k$-contact Lie system} is a triple $(M,\bm\eta,X)$, where $\bm\eta$ is a $k$-contact form on $M$ and $X$ is a Lie system on $M$ whose smallest 
    Lie algebra, $V^X$, consists of $\bm\eta$-Hamiltonian vector fields. A {\it $k$-contact Lie system} $(M,\bm\eta,X)$ is called \textit{projectable}  if the $\bm\eta$-Hamiltonian $k$-functions associated with the vector fields in $V^X$ are first integrals of the Reeb vector fields of $(M,\bm\eta)$. 
\qeddiamond\end{definition}

$k$-Contact Lie systems are called {\it projectable} because they can be projected onto $k$-symplectic Lie systems \cite{LV_15}, as it will be shown soon. Sometimes, $X$ is simply said to be a $k$-contact Lie system if there exists some $k$-contact form turning it into a $k$-contact Lie system $(M,\bm \eta,X)$. Moreover, ${\bm \eta}$-Hamiltonian vector fields, and other names of notions containing ${\bm \eta}$ like ${\bm \eta}$-Hamiltonian $k$-functions, may be called $k$-contact Hamiltonian vector fields or $k$-contact Hamiltonian $k$-functions when the knowledge of the particular ${\bm \eta}$ is not important.

Proposition \ref{Prop:Calculation} yields that  an $\bm\eta$-Hamiltonian $k$-function of an $\bm\eta$-Hamiltonian vector field is a first integral of all the Reeb vector fields of $\bm \eta$ if, and only if, the $\bm\eta$-Hamiltonian vector field commutes with all the Reeb vector fields of $\bm \eta$. Hence, a projectable $k$-contact Lie system $(M,\bm\eta,X)$ amounts to a Lie system $X$ on a manifold $M$ such that each element of $V^X\subset \mathfrak{X}_{\bm \eta}(M)$ is invariant relative to the flows of all the Reeb vector fields of $\bm\eta$. In other words, $R_1,\ldots,R_k$ are Lie symmetries of $V^X$ and, therefore, are Lie symmetries of the $k$-contact Lie system itself.

Again due to the isomorphism of Lie algebras 
\begin{equation}\label{eq:IsomorLie}
\Cinfty_{\bm \eta}(M,\mathbb{R}^k) \ni \bm f\longmapsto-X_{\bm f} \in \mathfrak{X}_{\bm \eta}(M),
\end{equation} 
 every $k$-contact Lie system $(M,\bm\eta,X)$ is related to a unique $t$-dependent $\bm\eta$-Hamiltonian $k$-function $\bm h:\mathbb{R}\times M
\rightarrow \mathbb{R}^k$ whose $\{{\bm h}_t\}_{t\in \mathbb{R}}$ are contained in a finite-dimensional Lie algebra of $\bm\eta$-Hamiltonian $k$-functions, and vice versa.  This motivates the following definition.

\begin{definition} A {\it $k$-contact Lie--Hamiltonian system} is a triple $(M,\bm\eta,{\bm h}:\mathbb{R}\times M\rightarrow \mathbb{R}^k)$, where $(M,\bm \eta)$ is a $k$-contact manifold and $\bm h$ is a $t$-dependent family of $k$-functions ${\bm h}_t:x\in M
\mapsto {\bm h}(t,x)\in \mathbb{R}^k$, with $t\in \mathbb{R}$, contained in a finite-dimensional Lie algebra $\mathfrak{W}$ of $\bm\eta$-Hamiltonian $k$-functions relative to the Lie bracket in $\Cinfty_{\bm \eta}(M,\mathbb{R}^k)$. We call $\mathfrak{W}$ an  ${\bm \eta}$-Lie--Hamilton algebra.
\qeddiamond\end{definition}

As the function ${\bm h}:\mathbb{R}\times M\rightarrow \mathbb{R}^k$ in $(M,\bm\eta,{\bm h}:\mathbb{R}\times M\rightarrow \mathbb{R}^k)$ gives rise to a $t$-dependent family of $\bm\eta$-Hamiltonian $k$-functions contained in an $\bm\eta$-Lie--Hamilton algebra  $\mathfrak{W}$, the isomorphism \eqref{eq:IsomorLie} yields that the $t$-dependent $\bm\eta$-Hamiltonian vector field $X_{\bm h}$ is a Lie system with a Vessiot--Guldberg Lie algebra given by the $\bm \eta$-Hamiltonian vector fields of $\mathfrak{W}$. Then, one has the following proposition.

\begin{proposition} Every $k$-contact Lie system with respect to a $k$-contact form $\bm \eta$ gives rise to a unique ${\it \bm \eta}$-Lie--Hamiltonian system, and vice versa. 
\end{proposition}

Anyhow, it is remarkable that $\mathfrak{W}$ need not be unique, as shown in following examples.

\begin{example}\label{Ex:TDepIsotropic}
(\textbf{$t$-dependent frequency isotropic harmonic oscillators}) Consider the system of harmonic isotropic oscillators on $\mathbb{R}^2$ with a $t$-dependent frequency  described by the $t$-dependent system of ordinary differential equations on $\T\mathbb{R}^2\simeq\mathbb{R}^4$ of the form
\begin{equation}\label{eq:Iso2DimLie}\left\{
\begin{aligned}
 \frac{\d x_i}{\d t} &=v_i,\\
\frac{\d v_i}{\d t}&=-\nu^2(t)x_i, 
\end{aligned}\right.\qquad i=1,2,   
\end{equation}
for an arbitrary $t$-dependent frequency $\nu(t)$ and coordinates $x_1,v_1,x_2,v_2$. Define three vector fields on $\T\mathbb{R}^2\simeq \mathbb{R}^4$ given by
$$
X_1=\sum_{i=1}^2v_i\frac{\partial}{\partial x_i}\,\qquad X_2=\frac 12\sum_{i=1}^2\left(x_i\frac{\partial}{\partial x_i}-v_i\frac{\partial}{\partial v_i}\right)\,,\qquad X_3 =- \sum_{i=1}^2 x_i\frac{\partial}{\partial v_i}\,.
$$
The vector fields $X_1,X_2,X_3$ span (over the real numbers) a Lie algebra $V_{is}$ isomorphic to $\mathfrak{sl}_2$. Indeed,
$$
[X_1,X_2]=X_1\,,\qquad [X_1,X_3]=2X_2\,,\qquad [X_2,X_3]=X_3\,,
$$
and one obtains a distribution $\mathcal{D}^{V_{is}}$ on $\T\mathbb{R}^2$ of the form
 $$
 \mathcal{D}_x^{V_{is}}=\langle X_1(x),X_2(x),X_3(x)\rangle,\qquad \forall x\in \T\mathbb{R}^2,
 $$of rank three almost everywhere. Note that system \eqref{eq:Iso2DimLie} is related to the $t$-dependent vector field
 $$
 X_{is}=\nu^2(t)X_3+X_1
 $$
 and becomes a Lie system. 
 
 Moreover, $\T\mathbb{R}^2$ admits a natural symplectic form
$$
\omega=\d x_1\wedge \d v_1+\d x_2\wedge \d v_2\,,
$$
and $X_1,X_2,X_3$ are Hamiltonian vector fields relative to $\omega$ with Hamiltonian functions $h^1,h^2,h^3\in \Cinfty(\T\mathbb{R}^2)$ because $\Lie_{X_\alpha}\omega=0$ for $\alpha=1,2,3$ and $\T\mathbb{R}^2$ is simply connected, which implies that every closed one-form is exact. Additionally, $\d h^1\wedge \d h^2\wedge \d h^3\neq 0$.

Moreover, there exists another vector field on $\T\mathbb{R}^2$, namely 
$$
X_4=\sum_{i=1}^2\left(x_i\frac{\partial}{\partial x_i}+v_i\frac{\partial}{\partial v_i}\right),
$$ commuting with $X_1,X_2,X_3$. Hence, $X_1,\ldots,X_4$ span a Lie algebra $V_e$ isomorphic to the matrix Lie algebra of $2\times 2$ matrices. The vector field $X_4$ is not Hamiltonian relative to $\omega$ because $\Lie_{X_4}\omega=2\omega\neq 0$.  Let us prove that we can turn all vector fields $X_1,\ldots,X_4$ into $\bm\eta_{is}$-Hamiltonian relative to a two-contact form $\bm \eta_{is}$. 

 The vector fields $X_1,X_2,X_3,X_4$ are linearly independent at the manifold $\mathcal{O}_{is}$ of points in $\T \mathbb{R}^2$ where $v_2x_1-x_2v_1\neq 0$. Moreover, $X_1,X_2,X_3,X_4$ are the fundamental vector fields of a locally transitive linear Lie group action of $\mathrm{GL}_2$, the Lie group of invertible linear real automorphisms on $\mathbb{R}^2$, on the manifold $\mathcal{O}_{is}\subset \T\mathbb{R}^2\backslash\{(0,0,0,0)\}$. As proved in \cite{GLMV_19}, the space $\mathcal{O}_{is}$  is locally diffeomorphic to the Lie group $\mathrm{GL}_2$ in such a manner that $X_1,\ldots,X_4$ become mapped into a basis of right-invariant vector fields. Then, there exists, at least locally around a generic point, a Lie algebra $\langle Y_1,\ldots,Y_4\rangle $ of Lie symmetries of the vector fields of $V_e$, i.e. $[Y_i,X_j]=0$ for $i,j=1,\ldots,4$, that is isomorphic to $\mathfrak{gl}_2$ with $Y_1\wedge\ldots\wedge Y_4\neq 0$. In fact, consider
 $$
 Y_1=x_2\frac{\partial }{\partial x_1}+v_2\frac{\partial}{\partial v_1},\qquad Y_2=\frac 12 \sum_{i=1}^2(-1)^{i}\left(x_i\frac{\partial}{\partial  x_i}+v_i\frac{\partial}{\partial v_i}\right),\qquad Y_3=-x_1\frac{\partial}{\partial x_2}-v_1\frac{\partial}{\partial v_2}
 $$
 and $X_4=Y_4$. The vector fields $Y_1,\ldots,Y_4$ commute with $X_1,\ldots,X_4$, while
 $$
 [Y_1,Y_2]=-Y_1,\qquad [Y_1,Y_3]=-2Y_2,\qquad [Y_2,Y_3]=-Y_3.
 $$
 Let $\Upsilon_1,\ldots,\Upsilon_4$ be the dual one-forms to $Y_1,\ldots,Y_4$, namely
 $$
 \Upsilon^1=\frac{v_1\d x_1-x_1\d v_1}{v_1x_2-x_1v_2}\,,\qquad \Upsilon^2=\frac{v_2\d x_1-x_2\d v_1+v_1\d x_2-x_1\d v_2}{v_1x_2-x_1v_2}\,,\qquad \Upsilon^3=\frac{v_2 \d x_2-x_2\d v_2}{v_1x_2-x_1v_2} 
 $$
  and
 $$
 \Upsilon^4=\frac{\d (v_1x_2-x_1v_2)}{2(v_1x_2-x_1v_2)}.
 $$
 Then,
$$
\d \Upsilon^2=2\Upsilon^1\wedge\Upsilon^3\,,\qquad \d\Upsilon^4=0\,.
$$
Hence, $\bm\eta_{is}=\Upsilon^2\otimes e_1+\Upsilon^4\otimes e_2$ defines a two-contact form that is invariant relative to $X_1,\ldots,X_4$. Indeed, 
$$
(\Lie_{X_i}\eta_{is}^\alpha)(Y_j)=X_i(\eta_{is}^\alpha(Y_j)-\eta_{is}^\alpha([X_i,Y_j])=0,\qquad i=1,\ldots,4\,,\qquad \alpha=2,4\,,
$$
and $X_1,\ldots,X_4$ are locally $\bm\eta_{is}$-Hamiltonian relative to $\bm\eta_{is}$. Define $\bm h_i=-\iota_{X_i}\bm\eta_{is}$ for $i=1,2,3,4$, which read
$$
\bm h_1=\frac{-2v_2v_1e_1}{v_1x_2-x_1v_2}\,,\qquad \bm h_2=\left(1-\frac{2v_1x_2}{v_1x_2-x_1v_2}\right)e_1\,,\qquad \bm h_3=-\frac{2x_2x_1e_1}{v_1x_2-x_1v_2}\,,\qquad \bm h_4=-e_2\,.
$$
One has that the following non-vanishing commutation relations
$$
\{{\bm h}_1,\bm h_{2}\}=-\bm h_{1}, \qquad\{\bm h_{1}, \bm h_{3}\}=-2\bm h_{2}, \qquad \{\bm h_{2}, \bm h_{3}\}=-\bm h_{3}
$$
and ${\bm h}_1,\ldots,{\bm h}_4$ close a finite-dimensional Lie algebra of $\bm\eta_{is}$-Hamiltonian functions, which is isomorphic to $\mathfrak{gl}(2,\mathbb{R})$. Note that $(\mathcal{O}_{is}\subset \T\mathbb{R}^2,\bm \eta_{is},{\bm h_1}+\nu^2(t){\bm h}_3)$ is a two-contact Lie--Hamiltonian system for the two-contact Lie system \eqref{eq:Iso2DimLie}.  Indeed, $\langle {\bm h}_1,\ldots,{\bm h}_4\rangle$ and $\langle {\bm h}_1,\ldots,{\bm h}_3\rangle$ are  $\bm \eta_{is}$-Lie--Hamilton algebras for $X_{is}$. 

Note that the Reeb vector fields are $Y_2$ and $Y_4$, which are Lie symmetries of $X_1,X_2,X_3,X_4$. Hence, their $\bm\eta_{is}$-Hamiltonian two-functions are first integrals of $Y_2,Y_4$. Then, $(\mathcal{O}_{is},\bm \eta_{is},X_{is})$ is a projectable two-contact Lie system and projects onto a new  Lie system. This will be analysed in detail in forthcoming sections. Note that we do not really include $X_4$ in the Vessiot--Guldberg Lie algebra of $X_{is}$: it was just a tool to transform the initial Lie system into a two-contact one. 
\end{example}

To complete the presentation of $k$-contact Lie systems, let us give an example of n

\begin{example}\label{Ex:NonConHam}Consider the manifold $\R^5$ equipped with linear coordinates $\{q,z_1,z_2,p_1,p_2\}$. Then,  $\R^5$ has a natural two-contact form given by  $\bm\eta_J = (\d z_1 - p_1\,\d q)\otimes e_1+(\d z_2 - p_2\,\d q)\otimes e_2$. Its associated Reeb vector fields are $R_1 = \tparder{}{z_1}$ and $R_2 = \tparder{}{z_2}$. Then, $\mathbb{R}^5$ is diffeomorphic to the first jet bundle $J^1(\mathbb{R},\mathbb{R}\times \mathbb{R}^2)$ associated with the fibre bundle $(q,z_1,z_2)\in \mathbb{R}\times \mathbb{R}^2\mapsto q\in \mathbb{R}$  and $\bm \eta_J$ is, essentially, the natural two-contact form with the adapted variables $\{q,z_1,z_2,p_1,p_2\}$. 

Consider the vector fields on $\mathbb{R}^5$ given by
$$ X_1 = \parder{}{z_1}\,,\qquad X_2 = \parder{}{z_2}\,,\qquad X_3 = \parder{}{q}\,,\qquad X_4 = q\parder{}{q} - p_1\parder{}{p_1}- p_2 \parder{}{p_2}\,,
$$
$$
X_5=z_1\frac{\partial}{\partial z_1}+\frac 14 z_2\frac{\partial}{\partial z_2}+\frac 12\left(q\frac{\partial}{\partial q}+ p_1\frac{\partial}{\partial p_1}-\frac 12 p_2\frac{\partial}{\partial p_2}\right).$$
Note that $X_1\wedge\ldots\wedge X_5$ is different from zero almost everywhere in $\mathbb{R}^5$. 
The vector fields $X_1,\ldots,X_5$ are $\bm\eta_J$-Hamiltonian relative to $(\R^5,\bm\eta_J)$ with $\bm\eta_J$-Hamiltonian two-functions
$$ \bm h_1=-e_1\,,\qquad \bm h_2 = -e_2\,,\qquad \bm h_3 =p_1e_1+p_2e_2\,,\qquad\bm  h_4 = p_1qe_1+p_2qe_2\,,$$
$$
\bm h_5 = (-z_1+qp_1/2)e_1+(-z_2/4+qp_2/2)e_2,$$
and span a five-dimensional Vessiot--Guldberg  Lie algebra with commutation relations
$$ [X_1,X_2]=0\,,\qquad [X_1,X_3]=0\,,\qquad [X_1,X_4]=0\,,\qquad [X_1,X_5]=X_1,$$
$$
[X_2,X_3] = 0\,,\qquad [X_2,X_4] =0\,,\qquad [X_2,X_5]=\frac14 X_2\,,
$$
$$ [X_3,X_4] = X_3\,,\qquad [X_3,X_5]=\frac12 X_3\,,$$
$$\qquad[X_4,X_5]=0\,.$$
This allows us to define a two-contact Lie system on $\mathbb{R}^5$ relative to $\bm\eta_J$ given by $(\mathbb{R}^5,\bm\eta_J,X_J)$ with
\begin{equation}\label{eq:vector-field-example-nonconservative}
    X_J=\sum_{\alpha=1}^5b_\alpha(t)X_\alpha\,.
\end{equation}
where $b_1(t),\ldots,b_5(t)$ are any $t$-dependent functions. If $b_1(t),\ldots,b_5(t)$ are such that the vectors $(b_1(t),\ldots,b_5(t))$, with $t\in\mathbb{R}$, span $\mathbb{R}^5$, then $V^{X_J}=\langle X_1,\ldots,X_5\rangle$ and $\mathcal{D}^{X_J}_x =\T_x \mathbb{R}^5$ for  $x$ in an open dense subset of $\mathbb{R}^5$. 

Since the $\bm\eta_J$-Hamiltonian two-function of $X_5$ is not a first integral of the Reeb vector fields $R_1=X_1,R_2=X_2$, we have that $(\mathbb{R}^5,\bm \eta_J,X_J)$ is a not a projectable two-contact Lie system when $(b_1(t),\ldots,b_5(t))$, for $t\in \mathbb{R}$, span $\mathbb{R}^5$. Note also that $X_J$ is associated with the $t$-dependent $\bm\eta_J$-Hamiltonian two-function
$$
    \bm h = \sum_{\alpha = 1}^5 b_\alpha(t)\bm h_\alpha\,,
$$
namely each $(X_J)_t$ is the $\bm \eta_J$-Hamiltonian vector field related to $\bm h_t$ for every $t\in \mathbb{R}$. Finally, $\langle {\bm h}_1,\ldots,{\bm h}_5\rangle$ span an ${\bm \eta}_J$-Lie--Hamilton algebra for $X_J$. In fact, the non-vanishing commutation relations read
$$
\{\bm h_{1}, \bm h_{5}\}=-\bm h_{1}, \qquad \{\bm h_{2}, \bm h_{5}\}=-\frac{1}{4}\bm h_{2}, \qquad \{\bm h_{3}, \bm h_{4}\}=-\bm h_{3}, \qquad \{\bm h_{3},\bm h_{5}\}=-\frac{1}{2}\bm h_{3}.
$$
Then, $(\mathbb{R}^5,\bm \eta_J,$$
    \bm h = \sum_{\alpha = 1}^5 b_\alpha(t)\bm h_\alpha\,)
$ is a two-contact Lie--Hamiltonian system for the two-contact Lie system \eqref{eq:vector-field-example-nonconservative}.  
\finish
\end{example}
\section{ Generalised {\it t}-dependent constants of the motion for {\it k}-contact Lie systems}\label{Sec:Constants}

Given a $k$-contact Lie system, one obtains that a Vessiot--Guldberg Lie algebra of $\bm\eta$-Hamiltonian vector fields gives rise to a Lie algebra of $\bm \eta$-Hamiltonian $k$-functions, which in turn gives rise to a momentum map 
$
{\bf J}\colon M\to \mathfrak{g}^{*}\otimes \mathbb{R}^k
$. This mapping will be relevant so as to determine some properties of the initial $k$-contact Lie system. In particular, it is used to obtain generalised constants of motion \cite{CFR_13} for $k$-contact Lie systems. This will illustrate some of the potential uses of $k$-contact geometry to analyse $k$-contact Lie systems. Let us start by a motivating example.

\begin{example}Let us go back to Example \ref{Ex:TDepIsotropic} of the $t$-dependent frequency isotropic oscillator on $\mathcal{O}_{is}$. Consider the presymplectic form $\langle \d \bm \eta_{is},e^1\rangle$. This gives rise to a Poisson bracket on the space of its admissible functions, i.e. functions  $f\in \Cinfty( \mathcal{O}_{is})$ such that $\d f$ belongs to the image of $\d \langle \bm \eta_{is},e^1\rangle$.   Since $X_1,\ldots,X_4$ are invariant relative to the Reeb vector fields, they become Hamiltonian with respect to $\d\langle \bm \eta_{is},e^1\rangle$. This latter fact will be shown in detail in Proposition \ref{prop:ReductSympl} and, for the time being, it is enough to verify this fact by a short calculation. Note that 
$$
h_\alpha^1=\langle \bm h_\alpha,e^1\rangle,\qquad \alpha=1,2,3,4,
$$
are Hamiltonian functions for $X_1,\ldots,X_4$, respectively, and span a Lie algebra of admissible functions isomorphic to $\mathfrak{sl}_2$. Then, $h_1^1h_3^1-(h_2^1)^2$ Poisson commutes with $h_1^1,h_2^1,h_3^1$  and becomes a constant of motion for $X_{is}$.
    
\end{example}

Let us generalise the properties of the above example for every projective $k$-contact Lie system. It is worth noting that if $\bm f,\bm g\in\Cinfty_{\bm \eta}(M,\mathbb{R}^k)\cap \ker \bm R$, then
\begin{multline}\label{eq:RelkConPre}
   \{ 
   \bm  f,\bm g\}_{\bm \eta}=\bm \eta([X_{\bm f},X_{\bm g}])=X_{\bm f}\iota_{X_{\bm g}}\bm \eta-{X_{\bm g}}\iota_{X_{\bm f}}\bm \eta-\d \bm \eta(X_{\bm f},X_{\bm g})\\=-\iota_{X_{\bm f}}\iota_{X_{\bm g}}\d \bm \eta +\iota_{X_{\bm g}}\iota_{X_{\bm f}}\d \bm \eta -\d \bm \eta(X_{\bm f},X_{\bm g})=\d\bm\eta(X_{\bm f},X_{\bm g}).
    \end{multline}
This explains the relation of the bracket of functions in $\ker \bm R$ and the bracket of presymplectic forms defined in the following proposition. 

\begin{proposition}\label{prop:ReductSympl}
If $\bm \eta$ is a $k$-contact form and $\theta\in \mathbb{R}^{k*}$, then $\omega=\langle\d \bm\eta,\theta\rangle$ is a presymplectic form. Every Lie algebra of $k$-contact Hamiltonian $k$-functions $\mathfrak{W}=\langle \bm h_1,\ldots,\bm h_r\rangle$ belonging to $\ker \bm R$ gives rise to a Lie algebra of Hamiltonian functions $\mathfrak{W}^\theta=\langle \bm \langle h_1,\theta\rangle ,\ldots,\langle h_r,\theta\rangle \rangle$ relative to  $\omega$. In particular, there exists a Lie algebra surjection
$$
\bm f\in \mathfrak{W}\mapsto { f}^\theta=\langle \bm f,\theta\rangle\in \mathfrak{W}^\theta
$$
If $C$ is an admissible function of $\langle \d \bm \eta,\theta\rangle$ and Poisson commutes with the elements of $\mathfrak{W}^\theta$ relative to the Poisson bracket induced by $\omega$, then $C$ is a constant of motion of every Lie system $X$ admitting an ${\bm \eta}$-Lie--Hamilton algebra $\mathfrak{W}$. 
\end{proposition}
\begin{proof} Since $\d\bm\eta$ is closed, so is $\langle\d\bm\eta,\theta\rangle$, which becomes presymplectic\footnote{Note that its rank may not be constant.}. Moreover, the functions $h^\theta_\alpha=\langle \bm h_\alpha,\theta\rangle$ become Hamiltonian functions for $\langle \d\bm \eta,\theta\rangle$ because $R_\alpha\bm h=0$. Indeed,
$$
\iota_{X_{\bm f}}\d \bm \eta=\d \bm f\Rightarrow \iota_{X_{\bm f}}\d \langle \bm \eta,\theta\rangle=\d\langle \bm f,\theta\rangle. 
$$Let us show that expression \eqref{eq:RelkConPre} shows that  there exists a Lie algebra surjection from $\mathfrak{M}$ onto $\mathfrak{M}^\theta$ 
and $X_{\bm f}$ is a presymplectic Hamiltonian vector field of $  f^\theta$. We consider $X_{ f^\theta}=X_{\bm f}$, although ${  f}^\theta$ does not determine a Hamiltonian vector field relative to $\d\bm\langle \bm\eta,\theta\rangle$ uniquely. More in detail, 
$$
\langle\{\bm f,\bm g\}_{\bm \eta},\theta\rangle=\d \langle \bm \eta,\theta\rangle(X_{\bm f},X_{\bm g})=\iota_{X_{\bm g}}\d    f^\theta=\d \langle \bm \eta,\theta\rangle(X_{ f^\theta},X_{\bm g})=\d \langle \bm \eta,\theta\rangle(X_{{ f}^\theta},X_{{ g}^\theta})=\{{  f}^\theta,{  g}^\theta\}_{\omega}.
$$
The fact that $C$ is a constant of motion of $X$ follows from standard presymplectic geometry and the fact that $X$ admits a $t$-dependent Hamiltonian of the form $h=\sum_{\alpha=1}^rb_\alpha(t)\langle {\bm h}_\alpha,\theta\rangle$ for certain $t$-dependent functions $b_1(t),\ldots,b_r(t)$. 
    
\end{proof}

\begin{theorem} Consider the Lie algebra $\Cinfty_{\bm \eta}(M,\mathbb{R}^k)$ relative to $\{\cdot,\cdot\}_{\bm\eta}$ and the space $\Cinfty_{\bm \eta,\theta}(M)$ of functions of the form $   \langle \bm f,\theta\rangle$ for a certain $\bm f\in \Cinfty_{\bm \eta}(M,\mathbb{R}^k)$ and a fixed $\theta\in \mathbb{R}^{k*}$. 
The morphism 
\begin{equation}
\label{eq:Poisson_algebra_morphism}
\bm f \in \Cinfty_{\bm\eta}(M,\R^k)\,\cap \, \ker \bm R \mapsto f^\theta
=-\langle\bm f,\theta\rangle \in \Cinfty_{\bm\eta,\theta}(M)\,\cap \, \ker \bm R 
\end{equation}
is a Poisson algebra morphism relative to the product $\bm f\star \bm g =\sum_{\alpha=1}^r(f^\alpha g^\alpha)   e_{\alpha}$ and the bracket $\{\cdot,\cdot\}_{\bm \eta}$ in $\Cinfty_{\bm\eta}(M,\R^k)\,\cap \, \ker \bm R $ and the Poisson bracket of $\langle \d\bm \eta,\theta\rangle$ and the standard multiplication of functions in $\Cinfty_{\bm\eta,\theta}(M)\,\cap \, \ker \bm R $.
\end{theorem}
\begin{proof}

To prove that $\Cinfty_{\bm\eta}(M,\R^k)\,\cap \, \ker \bm R$ is a Poisson algebra, take $\bm f,\bm g,\bm h$ belonging to $\ker \bm R$ and see that
\begin{align}
\{\bm f, \bm g\star\bm h\} &= -X_{\bm f}\left(\sum_{\alpha=1}^k g^\alpha h^\alpha  e_{\alpha}\right) =-  \sum_{\alpha=1}^k X_{\bm f}\left(g^\alpha\right)h^\alpha   e_{\alpha}
 -\sum_{\beta=1}^k g^\beta X_{\bm f}(h^\beta)  e_{\beta}= \notag\\
&=\{\bm f,\bm g\}\star \bm h+\bm g\star\{\bm f,\bm h\}.
\end{align}
Moreover, since $\Cinfty_{\bm \eta,\theta}(M) \cap \ker \bm R$ is  an $\mathbb{R}$-algebra and the morphism \eqref{eq:Poisson_algebra_morphism} is an $\mathbb{R}$-algebra morphism, the morphism is a Poisson algebra morphism. 
\end{proof}

One important thing about $k$-contact Lie systems is the determination of regions of the manifold that are invariant relative to its dynamics. Let us provide a first result about this.

\begin{proposition} Let $X$ be a $k$-contact Lie system with an $r$-dimensional Vessiot--Guldberg Lie algebra $V=\langle X_1,\ldots,X_r\rangle$ of $\bm \eta$-Hamiltonian vector fields with $\bm \eta$-Hamiltonian $k$-functions ${\bm h}_1,\ldots,{\bm h}_r$ for $X_1,\ldots,X_r$, respectively. Let $\mathfrak{W}=\langle {\rm h}_1,\ldots,{\bm h}_r\rangle$ be a Lie--Hamilton algebra of $\bm\eta$-Hamiltonian $k$-functions isomorphic to an abstract $r$-dimensional Lie algebra $\mathfrak{g}$. Then, one obtains a map ${\bm J}:x\in M\mapsto \sum_{\alpha=1}^r{\bm h}_{\alpha}(x)v^\alpha \in \mathfrak{g}^*\otimes \mathbb{R}^k$, where $\{v^1,\ldots,v^r\}$ is a basis for $\mathfrak{g}^*$. For every $\theta\in \mathbb{R}^{k*}$ such that $R_{{\bm h}_1},\ldots,R_{{\bm h}_r}$ are tangent to the zero level of $ J^\theta=\langle {\bm J},\theta\rangle$, which is assumed to be a submanifold, one has that $({J}^\theta)^{-1}(0)$ is an invariant subset of the dynamics of the system, provided it is a submanifold.
    
\end{proposition}
\begin{proof} Note that $X_{\bm h}=\sum_{\alpha=1}^kb_\alpha(t)X_{{\bm h}_\alpha}$ and
$$
\frac{\d{\bm J}^\theta}{\d t}=\frac{\partial {\bm J}^\theta}{\partial t}+X_{\bm h} {\bm J}^\theta=\!\sum_{\alpha,\beta=1}^kb_\alpha(t)\langle X_{{\bm h}_\alpha}{\bm h}_\beta,\theta\rangle =-\sum_{\alpha,\beta=1}^kb_\beta(t)\langle\{{\bm h}_\alpha,{\bm h_\beta}\}+R_{{\bm h}_\beta}{\bm h}_\alpha,\theta\rangle.
$$
Recall that ${\bm h}_1,\ldots, {\bm h}_r$ close a Lie algebra of functions, namely $\{{\bm h}_\alpha,{\bm h}_\beta\} =\sum_{\gamma=1}^rc_{\alpha\beta}\,^\gamma {\bm h}_\gamma$ for certain constants $c_{\alpha\beta}\,^\gamma$, and one has 
$$
\langle R_{{\bm h}_\beta}{\bm h}_\alpha, \theta\rangle=\sum_{\mu=1}^kh_\beta^\mu R_\mu \langle{\bm h}_\alpha,\theta\rangle.
$$
Take into account that $\langle {\bm J},\theta\rangle=\sum_{\alpha=1}^r\langle {\bm h }_\alpha v^\alpha,\theta\rangle=\sum_{\alpha=1}^r\langle {\bm h}_\alpha,\theta\rangle v^\alpha=0$ and $R_1,\ldots,R_k$ are tangent to the level set $({  J}^\theta)^{-1}(0)$. Hence, on such a submanifold,  $\langle {\bm h}_\alpha,\theta\rangle=0$ and $R_{{\bm h}_\beta}\langle {\bm h}_\alpha,\theta\rangle=0$ for $\alpha,\beta=1,\ldots,r$.  Thus, $X_{\bm h}$ is tangent to the zero level set of ${  J}^\theta$.

\end{proof}

\begin{example} Consider Example \ref{Ex:NonConHam} and the $t$-dependent vector field $X_J=\sum_{\alpha=3}^5b_\alpha(t)X_\alpha$. Assume $\theta=e_2$. One has the mapping
$$
{J}^\theta=\langle\bm J,\theta\rangle=(p_2,p_2 q,-z_2/4+qp_2/2).
$$
Then $({ J}^\theta)^{-1}(0)$ is the submanifold given by $p_2=0$ and $z_2=0$.  Note that
$$
R_{{\bm h}_3}= p_1\frac{\partial}{\partial z_1}+p_2\frac{\partial}{\partial z_2},\quad R_{{\bm h}_4}= p_1q\frac{\partial}{\partial z_1}+p_2q\frac{\partial}{\partial z_2},\quad R_{{\bm h}_5}= \left(-z_1 +\frac{qp_1}{2}\right)\frac{\partial}{\partial z_1}+\left(-\frac{z_2}{4}+\frac{qp_2}{2}\right)\frac{\partial}{\partial z_2}
$$
and $R_{{\bm h}_\beta}\langle \bm h_\alpha,\theta\rangle=0$ for $\alpha,\beta=3,4,5$ on $(J^\theta)^{-1}(0)$. Then, $({  J}^\theta)^{-1}(0)$ is an invariant submanifold of the dynamics of the system.
\end{example}

Let us now describe a theory to derive  $t$-dependent constants of motion for $k$-contact Lie systems via $k$-contact Lie--Hamilton algebras. The main problem to extend the results of Lie--Hamilton systems \cite{BCHLS13} to this new realm is the fact that the Lie bracket on $\mathfrak{W}$ is not a Poisson bracket and describing the variation of a $t$-dependent function in time does not depend only on the Lie--Hamilton algebra, but it also depends on the action of the Reeb vector fields on them, which can be difficult to analyse. A possible solution is given below.

\begin{proposition}\label{Prop:ConsMotkCon} Let $(M,\bm \eta,X)$ be a $k$-contact Lie system such that the Reeb derivations of the elements of an ${\bm \eta}$-Lie--Hamilton algebra leave invariant the ${\bm \eta}$-Lie--Hamilton algebra, namely  $R_{{\bm h}_\alpha}{\bm h}_{\beta}=\sum_{\gamma=1}^r\lambda_{\alpha\beta}\,^\gamma {\bm h}_{\gamma}$ for certain constants $\lambda_{\alpha\beta}\,^\gamma$ with $\alpha,\beta,\gamma=1,\ldots,r$. Then, the $t$-dependent constants of motion of $X$ of the form $I_\theta=\langle \sum_{\alpha=1}^rf^\alpha(t){\bm h}_{\alpha},\theta\rangle$ for some $\theta\in \mathbb{R}^{k*}$ and $t$-dependent functions $f_1(t),\ldots,f_r(t)$ are solutions to
\begin{equation}\label{eq:Final}
\frac{df^\alpha}{dt}=-\sum_{\nu,\beta=1}^rf^\nu b_\beta [c_{\nu\beta}\,^\alpha-\lambda_{\nu\beta}\,^\alpha],
\end{equation}
if $\langle h_1^\theta,\ldots,h_r^\theta\rangle$ are linearly independent.
\end{proposition}
\begin{proof} If ${\bm I}=\sum_{\alpha=1}^rf^\alpha(t){\bm h}_\alpha$, then one has
\begin{equation}\label{eq:Derivation}
\frac{\d {\bm I}}{\d t}(t)=\frac{\partial  {\bm I}}{\partial t}(t)+(X_{\bm h})_t {\bm I}_t=\sum_{\alpha=1}^r\left(\frac{\d f^\alpha}{\d t}(t){\bm h}_\alpha+\sum_{\beta=1}^rb_\beta(t)f^\alpha (t)X_{{\bm h}_\beta}{\bm h}_\alpha\right)
\end{equation}
Hence,
$$
\frac{\d  {\bm I}}{\d t}=\sum_{\alpha=1}^r\left(\frac{\d f^\alpha}{\d t}{\bm h}_{\alpha}-\sum_{\beta=1}^rb_\beta(t)f^\alpha (\{{\bm h}_\beta,{\bm h}_\alpha\}+R_{{\bm h}_\alpha}{\bm h}_\beta)\right).
$$
For $t$-dependent constants of the motion and composing both sides of the above expression with $\theta$, one has that 
$$
\sum_{\alpha=1}^r\frac{\d f^\alpha}{\d t}{ h}^\theta_{\alpha}=\sum_{\alpha,\beta=1}^rb_\beta(t)f^\alpha (\{{\bm h}_\beta,{\bm h}_\alpha\}^\theta+R_{{\bm h}_\alpha}{ h}^\theta_\beta)=-\sum_{\alpha,\beta=1}^rb_\beta(t)f^\alpha \sum_{\nu=1}^r(c_{\alpha\beta}\,^\nu{h}^\theta_\nu-\lambda_{\alpha\beta}
\,^\nu {h}^\theta_\nu). $$
If the functions $h_\alpha^\theta$ are linearly independent, it follows \eqref{eq:Final}.
\end{proof}
The main result of Proposition \ref{Prop:ConsMotkCon} is the fact that some $t$-dependent constants of motion for a  $k$-contact Lie system can be described via another Lie system \eqref{eq:Final}. This illustrates one of the reasons why $k$-contact geometry is interesting for $k$-contact Lie systems. On the other hand, these examples also showed that it is natural to associated every $k$-contact $\bm \eta$-Hamiltonian function with a vector field taking values in the Reeb distribution. 

\begin{remark}\label{Rem:SpecialCases} Note that \eqref{eq:Derivation} and the conditions in Proposition \ref{Prop:ConsMotkCon} are the key to many new methods. Assume that $I_\theta$ is $t$-independent, the Lie algebra $\langle [\mathfrak{W},\mathfrak{W}],\theta\rangle$ of a $k$-contact Lie--Hamilton algebra $\mathfrak{W}$ consists of constants and the $R_{\bm h_\alpha} \langle {\bm h}_\beta,\theta\rangle$ are constants (as it will happen in some examples of this work). Under the given conditions, one has, in view of \eqref{eq:Derivation}, that 
$
b(t)=\frac{dI_\theta}{dt}
$ for a $t$-dependent function $b(t)$. Hence, $\int^tb(t')\d t'-I_\theta$ is a constant of motion for $(M,\bm \eta,X)$.
 
    Recall that a {\it generalised master symmetry of order $m$} for a system of differential equations associated with a vector field $X$ on a manifold $M$ is a vector field $Z$ on $M$ such that ${\rm ad}_X^mZ=0$, where ${\rm ad}^m_X=(m-{\rm times}) {\rm ad}_X\circ \ldots\circ {\rm ad}_X$ (see \cite{Fer_93}).  Generalised master symmetries allow one to obtain functions, the {\it generators of constants of motion of order $m$}, whose derivative of order $m$ in terms of the time is zero, but the derivative of order $m-1$ is not. The same idea, which appears in symplectic mechanics, can be applied to $t$-independent projectable $k$-contact Lie systems. For instance, assume that $\mathfrak{W}_\theta$ is nilpotent of order $m$ consisting of first integrals of the Reeb vector fields of a $k$-contact Lie system $(M,\bm\eta,X)$. Consider also the ${\rm \eta}$-Lie Hamilton algebra $\mathfrak{W}$ for $V^X$. In such a case,
$$
\frac{\d^m I_\theta}{\d t^m}=(-1)^m\stackrel{m-{\rm times}}{\overbrace{\{h_\theta,\{h_\theta \ldots,\{h_\theta,I_\theta\}\ldots\}}}=0,\qquad I_\theta \in \mathfrak{W}^\theta,
$$
    and $I_\theta$ is a $t$-dependent function of order, at most $m-1$,  on solutions. Meanwhile, $X_{I_\theta}$ becomes a master symmetry of order $m$ of $X$.
    
\end{remark}
 \section{\texorpdfstring{$k$}{}-contact  Lie systems and other  geometric structures}\label{Sec:kcontactAndOthers}

It is worth noting that a Lie system may potentially be Hamiltonian relative to different types of geometric structures. In some cases, some of them can be more appropriate than others. The following proposition shows how projectable $k$-contact Lie systems induce some $k$-symplectic Lie  systems on other spaces.

\begin{proposition}\label{prop:conservative-contact-project-symplectic} If $(M,\bm \eta,X)$ is a projectable  $k$-contact Lie system, the space of integral curves of the Reeb distribution $\mathcal{D}^R:=\ker \d\bm \eta$, let us say $M/\mathcal{D}^R$, is a manifold, and $\mathfrak{p}_R:M\rightarrow M/\mathcal{D}^R$ is the canonical projection that becomes a submersion, then $(M/\mathcal{D}^R,\bm\omega,\mathfrak{p}_{R*}X)$, where $\mathfrak{p}_{R}^*\bm\omega=\d\bm \eta$, is a $k$-symplectic Lie system relative to the $k$-symplectic form $\bm\omega$ on $M/\mathcal{D}^R$, namely the vector fields $\mathfrak{p}_{R*}X_t$, with $t\in \mathbb{R}$, are $\bm \omega$-Hamiltonian relative to the $k$-symplectic form $\bm\omega$.
\end{proposition}
\begin{proof} Since $(M,\bm\eta,X)$ is projectable, the Reeb vector fields $R_1,\ldots,R_k$ commute with  the $\bm \eta$-Hamiltonian vector fields of $V^X$. Therefore, all the elements of $V^X$ are projectable onto $M/\mathcal{D}^R$ and their projections span a finite-dimensional Lie algebra of vector fields on $M/\mathcal{D}^R$. Moreover, the $\bm \eta$-Hamiltonian $k$-functions of the elements of $V^X$ are first integrals of the Reeb vector fields of $\bm \eta$. Hence, they are also projectable. Moreover, $\Lie_{R_\alpha}\d\bm\eta=0$ and $\iota_{R_\alpha}\d\bm\eta=0$ for $\alpha=1,\ldots,k$. Hence, $\d\bm\eta$ can be projected onto $M/\mathcal{D}^R$. In other words, there exists a unique two-form, $\bm\omega$, on $M/\mathcal{D}^R$ such that $\mathfrak{p}_R^*\bm\omega=\d\bm\eta$. Note that $\bm\omega$ is closed. Moreover, if $\iota_{Y_{[x]}}\bm\omega_{[x]}=0$ for a tangent vector $Y_{[x]}\in \T_{[x]}(M/\mathcal{D}^R)$ where $[x]$ is the leave of $\mathcal{D}^R$ passing through a point $x\in M$, then there exists a tangent vector $\widetilde{Y}_x\in\T_xM$ such that $\T_x\mathfrak{p}_R(\widetilde{Y}_x)=Y_{[x]}$. Hence, $\T_{x}\mathfrak{p}_R^T\iota_{Y_{[x]}}\bm\omega_{[x]}=\iota_{\widetilde{Y}_x}(\d\bm\eta)_x=0$ and  $\widetilde{Y}_x\in \ker(\d\bm\eta)_x$. Moreover, $\mathfrak{p}_{R*x}\widetilde{Y}_x=0$ and $\bm\omega$ is non-degenerate. Since $\bm \omega$ is closed too, it becomes a $k$-symplectic form and the vector fields of $\mathfrak{p}_{R*}V^X$ span a finite-dimensional Lie algebra of Hamiltonian vector fields relative to $\bm \omega$. In fact, given a basis of ${\bm \eta}$-Hamiltonian $k$-functions $\bm h_1,\ldots,\bm h_r$ for a basis $X_1,\ldots,X_R$ of vector fields of $V^X$, one may write that $\bm h_\alpha=\mathfrak{p}_R^*\bm g_\alpha$ for $\alpha=1,\ldots,k$ and some $k$-functions $\bm g_1,\ldots,\bm g_r\in \Cinfty(M/\mathcal{D}^R,\mathbb{R}^k)$ . Moreover,
$$
\mathfrak{p}_R^*(\iota_{\mathfrak{p}_{R*}X_\alpha}\bm \omega)=\iota_{X_\alpha}\d\bm\eta=\d\bm h_\alpha-\sum_{\mu,\beta=1}^k(R_\beta h_\alpha^\mu)\eta^\beta \otimes e_\mu=\d\mathfrak{p}_{R}^*\bm g_\alpha=\mathfrak{p}_R^*(\d\bm g_\alpha).
$$
Hence, $\mathfrak{p}_{R*}X_\alpha$ is $\bm\omega$-Hamiltonian relative to $\bm\omega$. Therefore, the time-dependent vector field $\mathfrak{p}_{R*}X$, namely the $t$-parametric family of vector fields $(\mathfrak{p}_{R*}X)_t=\mathfrak{p}_{R*}X_t$ for every $t\in \mathbb{R}$, becomes a $k$-symplectic Lie system relative to $\bm\omega$. 
\end{proof}

Let us provide a proposition that allows one to study $k$-contact Lie systems via presymplectic Lie systems. Notwithstanding, this implies that one has to study a problem on a higher-dimensional manifold with different properties than the initial one, which may make the problem harder to study in some cases \cite{LR_23}, but easier in others. It is worth noting that Proposition \ref{Eq:ExtHam} is a particular case of \cite[Proposition 7.14]{LRS_25} focusing on a particular case of relevance for us.

\begin{proposition}\label{Eq:ExtHam}
Every $k$-contact Lie system $(M,\bm\eta,X_h)$ gives rise to a presymplectic Lie system  $(\mathbb{R}^k\times M,\bm \omega=\d(\sum_{\alpha=1}^kz_\alpha \widehat{\eta}^\alpha),\sum_{\alpha,\beta=1}^kz_\alpha(R_\beta h^\alpha)\partial/\partial z^\beta+X_{\bm h})$, where $\hat \eta^\alpha$ is the lift to $\mathbb{R}^k\times M$, via the natural projection from $\mathbb{R}^k\times M$ onto $M$, of $\eta^\alpha$ for $\alpha=1,\ldots,k$. 
\end{proposition}

As it happens for the similar proposition in the contact case, Proposition \ref{Eq:ExtHam} may be inappropriate to study $k$-contact Hamiltonian systems on $M$ via Hamiltonian systems on presymplectic manifolds $\mathbb{R}^k\times M$ since the dynamics of a $k$-contact Hamiltonian vector field on $M$ may significantly differ from the presymplectic Hamiltonian system on $\mathbb{R}^k\times M$ used to investigate it. For instance, a $k$-contact Hamiltonian vector field $X$ on $M$ may have stable points, while $\sum_{\alpha,\beta=1}^kz_\alpha (R_\beta h^\alpha )\partial/\partial z^\beta +X_{\bm h}$, which is its associated Hamiltonian vector field on $\mathbb{R}^k\times M$, has not. This has relevance in certain theories, like the energy-momentum method \cite{MS_88}. Additionally, enlarging the manifold where the dynamics of a system is considered poses a problem that should have more advantages than drawbacks.

 \begin{example} Note that Example \ref{Ex:TDepIsotropic} satisfies the conditions of Proposition \ref{prop:conservative-contact-project-symplectic}. Then, the Reeb vector fields have constants of motion given by the common first integrals of the vector fields
$$
x_1\frac{\partial}{\partial x_1}+v_1\frac{\partial}{\partial v_1},\qquad x_2\frac{\partial}{\partial x_2}+v_2\frac{\partial}{\partial v_2}.
$$
We can define the first integrals  $f_1=v_1/x_1$ and $f_2=v_2/x_2$ of the previous vector fields.  Hence, the integral leaves of the distribution spanned by $Y_4,Y_2$ can be described by means of the surfaces with $f_1,f_2$ equal to constants. Hence, $f_1,f_2$ are coordinates in the space of leaves. In those coordinates, the $\bm\eta
_{is}$-Hamiltonian vector fields  $X_1,\ldots,X_4$ read
$$
X_{1}=f_1x_1\parder{}{x_1}+f_{2}x_{2}\parder{}{x_{2}}-f_{1}^2\parder{}{f_{1}}-f_{2}^2\parder{}{f_{2}}, \qquad X_{2}=\frac{1}{2}x_1\parder{}{x_{1}}+\frac{1}{2}x_{2}\parder{}{x_{2}}-f_{1}\parder{}{f_{1}}-f_{2}\parder{}{f_{2}},
$$
$$
X_{3}=-\parder{}{f_1}-\parder{}{f_{2}}, \qquad X_{4}=x_{1}\parder{}{x_{1}}+x_{2}\parder{}{x_{2}},
$$
and the differential of two-contact form becomes
$$
\d\bm\eta_{is}=2\Upsilon^1\wedge\Upsilon^3\otimes e_{1}=\left(\frac{2}{(f_{2}-f_{1})^{2}}\d f_{1} \wedge \d f_{2}\right)\otimes e_{1}\,.
$$
We can define a mapping $\mathfrak{p}_R\colon (f_1,f_2,x_1,x_2)\in \mathbb{R}^4\mapsto(f_1,f_2)\in \mathbb{R}^2$ that projects the initial space onto the space of leaves, which, by Proposition \ref{prop:conservative-contact-project-symplectic}, is the two-symplectic manifold with two-symplectic form $\bm \omega$ defined by $\mathfrak{\bm p}_{R}^*\bm\omega=\d\bm\eta_{is}$. Moreover, $\mathfrak{\bm p}_{R_*} X_{i}$, where $i=1,\ldots,4$, are given by
$$
\mathfrak{\bm p}_{R_*} X_{1}=-f^{2}_{1}\parder{}{f_{1}}-f_{2}^{2}\parder{}{f_{2}}, \qquad \mathfrak{\bm p}_{R_*} X_{2}=-f_{1}\parder{}{f_{1}}-f_{2}\parder{}{f_{2}}, \qquad \mathfrak{\bm p}_{R_*} X_{3}=-\parder{}{f_{1}}-\parder{}{f_{2}},
$$
and $\mathfrak{\bm p}_{R_*} X_{4}=0$, are Hamiltonian vector fields relative to the two-symplectic  form $\d\bm\eta_{is}$, with Hamiltonian $k$-functions on $\mathbb{R}^2$ of the form 
$$
\bm h'_{1}=-\frac{2f_{1}f_{2}}{f_{1}-f_{2}}e_{1}, \qquad\bm h'_{2}=\left(1+\frac{2f_{1}}{f_{2}-f_{1}}\right)e_{1}, \qquad \bm h'_{3}=-\frac{2}{f_{1}-f_{2}}e_{1}, \qquad \bm h'_{4}=-e_{2}.
$$
Note that $\mathfrak{\bm p}_{R}^* \bm h'_{i}=\bm h_{i}$ for $i=1,\ldots, 4$.
\end{example}

It is interesting to analyse how to recover the solution to the initial projectable $\bm \eta$-contact Lie system from the information of the projected $k$-symplectic one. Some particular results in this direction have recently appeared in the literature \cite{Car_25}. We expect to investigate this topic further in the future.

\section{Methods to construct \texorpdfstring{$k$}{k}-contact Lie systems}\label{Sec:NewMethods}

Let us develop a method to turn a Lie system $X$ into a $k$-contact Lie system by finding a compatible $k$-contact form, i.e. one turning $V^X$ into $\bm\eta$-Hamiltonian vector fields for a certain $k$-contact form $\bm\eta$. This will be more or less explicitly used in further sections to determine potential applications of $k$-contact Lie systems. Our method for deriving $\bm\eta$ relies on distributions, a strategy that generally simplifies the process compared to directly calculating $\bm\eta$.

Recall that a {\it locally automorphic Lie system} is a triple $(M,X,V)$, where $X$ is a Lie system on a manifold $M$ so that $V$ is a Lie algebra of dimension $\dim M$ whose vector fields span the tangent space to $M$. It is known \cite{LS_20,GLMV_19} that locally automorphic Lie systems are locally diffeomorphic to a Lie system on a Lie group $G$ of the form
\begin{equation}\label{eq:AutomorphicLie}
X^R=\sum_{\alpha=1}^rb_\alpha(t)X_\alpha^R,
\end{equation}
where $X_1^R,\ldots,X_r^R$ is a basis of right-invariant vector fields on $G$ spanning a Lie algebra isomorphic to $V$ and $b_1(t),\ldots,b_r(t)$ are $t$-dependent functions. This is structure is extremely useful, as one can treat locally automorphic Lie systems on general manifolds as Lie systems of a very specific type on Lie groups. Moreover, there are many applications of locally automorphic Lie systems \cite{CL_11,CGM_00,LS_20}. For instance, there are many locally automorphic Lie systems related to control systems (cf. \cite{Ram_02}) and we will study examples of this type in following sections.

Consider now the Maurer--Cartan equations on $G$, which satisfy that
\begin{equation}\label{eq:MauG}
\d\eta_L^\alpha+\frac 12\sum_{\beta,\gamma=1}^rc_{\beta\gamma}\!^\alpha\eta_L^\beta\wedge\eta_L^\gamma=0\,,\qquad  \alpha=1,\ldots,r\,,
\end{equation}
where $c_{\beta\gamma}\!^\alpha$, with $\alpha,\beta,\gamma=1,\ldots,r$, are the constants of structure of a basis of left-invariant vector fields on $G$ given by $X^L_1,\ldots,X^L_r$, while $\eta^1_L,\ldots,\eta^r_L$ is its  dual basis, which consists of left-invariant one-forms on $G$.  
The above expression \eqref{eq:MauG}
follows immediately by evaluating $\d\eta^\alpha_L$ on pairs of left-invariant vector fields chosen among $X^L_1,\ldots,X_r^L$. Then, the differentials of left-invariant one-forms can be determined by the Lie algebra structure of $V$. 

In particular, it is simple to find Lie algebras admitting $k$ particular different indexes $\alpha_1,\ldots,\alpha_k\in \{1,\ldots,r\}$ so that $c_{\alpha_j\beta}\,\!^{\alpha_i}=0$ for every $\beta$ and $i,j=1,\ldots,k$. In view of \eqref{eq:MauG}, this implies that the elements of $\langle X^L_{\alpha_1},\ldots,X^L_{\alpha_k}\rangle$ take values in $ \bigcap_{i=1}^k\ker \d\eta_L^{\alpha_i}$.  If additionally $\bigcap_{i=1}^k\ker \d\eta_L^{\alpha_i}=\langle X^L_{\alpha_1},\ldots,X^L_{\alpha_k}\rangle$, then $\bm \eta=\sum_{i=1}^k\eta^{\alpha_i}_L\otimes e_{\alpha_i}$ is a $k$-contact form that is invariant relative to the vector fields $X^R_1,\ldots,X^R_r$. Recall that Lie groups of this type are related to the so called {\it $k$-contact Lie groups} \cite{LRS_25}, namely Lie groups with a left-invariant $k$-contact form\footnote{Note that the approach that follows could be developed similarly with a basis of left-invariant vector fields and a right-invariant $k$-contact form.}. In practice, it is relatively simple to determine Lie groups of this type, which will be shown in latter examples. Moreover, one-contact Lie groups were analysed in \cite{LR_23}. On the opposite, Lie groups no admitting a left-invariant contact form were discussed in \cite{GL_25} in the context of hydrodynamic equations.

Right-invariant vector fields leave invariant $\ker \bm \eta$, which can be spanned by left-invariant vector fields, which implies that the vector fields of the Vessiot--Guldberg Lie algebra $\langle X^R_1,\ldots,X^R_r\rangle$ consists of $\bm\eta$-Hamiltonian vector fields relative to $\bm \eta$. 
With no loss of generality (see \cite{LRS_25} for details), the basis of left-invariant vector fields of a $k$-contact Lie group can be chosen so that the Reeb vector fields are given by $X^L_1,\ldots,X_k^L$, which commute with $X^R_1,\ldots,X^R_r$. Then, the Lie systems of the form
$$
X^R=\sum_{\alpha=1}^rb_\alpha(t)X_\alpha^R,
$$
are projectable $k$-contact Lie systems regardless of the functions $b_1(t),\ldots,b_r(t)$. Indeed, it follows that the $\bm \eta$-Hamiltonian $k$-functions of $X^R_1,\ldots,X^R_r$ are first integrals of the Reeb vector fields, which are left-invariant, and the Lie system can be projected onto $G/\ker\d \bm \eta$, where $\ker\d\bm\eta $ is the distribution given by the Reeb vector fields. More specifically, it turns out that the $k$-contact Lie system is projectable onto a $k$-symplectic one \cite{LV_15}.

There is another manner to obtain, in general non-projectable, $k$-contact Lie systems from locally automorphic ones. As before, we analyse a locally automorphic Lie system via its associated Lie system \eqref{eq:AutomorphicLie} on a Lie group. Consider that the Lie algebra of the Lie group is such that there exists a linear subspace spanned by left-invariant vector fields, $X^L_1,\ldots,X^L_{r-k}$, that is maximally non-integrable, i.e. it admits no element leaving invariant the subspace via the commutator of vector fields, which is relatively straightforward to determine due to the Lie algebra structure of left-invariant vector fields. Consider also that there exists a supplementary abelian Lie subalgebra assumed to have, with no loss of generality, a basis of the form $X^L_{r-k+1},\ldots,X^L_r$. Then, let us analyse the maximally non-integrable distribution spanned by
$$
\mathcal{D}=\langle X^L_1,\ldots, X^L_{r-k}\rangle
$$
and the commutative Lie algebra
$$
\langle X^R_{r-k+1},\ldots,X^R_{r}\rangle.
$$
Note that the right-invariant vector fields are chosen so that $X^R_\alpha(e)=X^L_\alpha(e)$ for $\alpha=1,\ldots,r$, which makes them to have opposite commutation relations than left-invariant vector fields. 
The latter right-invariant vector fields $X^R_{r-k+1},\ldots,X^R_r$ leave $\mathcal{D}$ invariant and commute among themselves. 
This gives rise to a $k$-contact distribution $\mathcal{D}$ and the vector fields $X_{r-k+1}^R,\ldots,X_r^R$ leave it invariant. Hence, $X_1^R,\ldots,X_r^R$ become $k$-contact vector fields and one obtains a $k$-contact Lie system related to \eqref{eq:AutomorphicLie}. In this approach, the $k$-contact form must be obtained as in the proof of Theorem \ref{Prop:LocalEqu}. Namely, one has to derive the $k$ one-forms $\eta^{r-k+1},\ldots,\eta^{r}$ that annihilate $\langle X^L_1,\ldots,X^L_{r-k}\rangle$ and are duals to the vector fields $X^R_{r-k+1},\ldots,X^R_{r}$. It is worth noting that this approach works, at least locally at a point where $X^L_1,\ldots,X^L_{r-k},X^R_{r-k+1},\ldots,X^R_{r}$ are linearly independent, which can be always found locally at almost every point. The obtained one-forms are the components of a $k$-contact form $\bm \eta=\sum_{\alpha=r-k+1}^{r}\eta_L^\alpha\otimes e_\alpha$ turning the vector fields $X^R_1,\ldots,X^R_r$ into $\bm\eta$-Hamiltonian vector fields. Note that the $k$-contact form $\bm \eta$ described above is not, in general, a left-invariant one-form, while  $X^R$ does not need to be projectable.

There are several manners to obtain a maximally non-integrable $\mathcal{D}$. For instance, if $[X_1^L,X_2^L]\wedge X_1^L\wedge X_2^L$ is not vanishing, then $\mathcal{D}$ is maximally non-integrable. This is a relatively straightforward case to identify, as it suffices to locate two left-invariant vector fields $X$ and $Y$, such that their commutator $[X,Y]$ does not take values within the distribution spanned by $X$ and $Y$ at any point. Note that if $[X,Y]$ does not take values in the distribution spanned by $X,Y$ at any point, then it is impossible that $X\wedge Y= 0$ on an open set as, if this would happen,  
$$0=\mathcal{L}_X(X\wedge Y)=X\wedge \mathcal{L}_XY=\mathcal{L}_Y(X\wedge Y)=\mathcal{L}_YX\wedge Y\qquad\Rightarrow\qquad [X,Y]\wedge X\wedge Y=0,
$$
which is a contradiction.
Moreover, assume that $E=\langle X^L_1,X^L_2,X_3^L\rangle\subset \langle X^L_1,\ldots,X_r^L\rangle$ is such that    there is no element of $Y\in E$ such that $[Y,E]$ is not included in $E$. Then, $E$ is also maximally non-integrable. These ideas to obtain maximally non-integrable distributions will be used in our applications  in Section \ref{Sec:applications}.
\section{Diagonal prolongation of {\it k}-contact Lie systems}\label{Sec:Diagonalkcontact}
It is possible to describe superposition rules of Lie systems geometrically. To do so, it is convenient to introduce the so-called \textit{diagonal prolongation} of a $t$-dependent vector field and its basic properties \cite{CGM_07}.
Before that, let us introduce the projections
\begin{equation} \label{proj1}
    \mathfrak{pr}\colon M^{ \ell+1 } \ni (x_{(0)}, \ldots,x_{(\ell)}) \longmapsto (x_{(1)}, \ldots,x_{(\ell)}) \in M^{\ell}\,,
\end{equation}
\begin{equation} \label{proj2}
\mathfrak{pr}_{0}^{\alpha} \colon M^{ \ell+1 } \ni (x_{(0)},\ldots,x_{(\ell)}) \longmapsto x_{(\alpha)} \in M\,,
\end{equation}
for $\alpha=0,\ldots,\ell$. Recall that the group $S^{\ell+1}$ of permutations $x_{(\alpha)}\leftrightarrow x_{(\beta)}$, with $0\leq \alpha\neq \beta\leq \ell $ acts on $M^{\ell+1}$. 
\begin{definition}
    Given a $t$-dependent vector field on $M$ locally of the form
    \begin{equation}
    X(t,x_{(0)})=\sum^{n}_{i=1}X^{i}(t,x_{(0)})\frac{\partial}{\partial x^{i}_{(0)}}\,,
    \end{equation}
its \textit{diagonal prolongation} to $M^{\ell+1}$ is the $t$-dependent vector field on $M^{\ell+1}$ given
by
\begin{equation}
{X}^{[\ell+1]}(t,x_{(0)},\ldots,x_{(\ell)})=\sum^{\ell}_{a=0}\sum^{n}_{i=1}X^{i}(t,x_{(a)})\frac{\partial}{\partial x^{i}_{(a)}}\,.
\end{equation}
\qeddiamond\end{definition}
This definition can be rewritten in the following manner.
\begin{definition}
    Given a $t$-dependent vector field $X$ on $M$, its \textit{diagonal prolongation} to $M^{\ell+1}$ is the unique $t$-dependent vector field ${X}^{[\ell+1]}$ on $M^{\ell+1}$ such that:
\begin{enumerate}[(i)]
    \item The $t$-dependent vector field ${X}^{[\ell+1]}$ is invariant under the action of the symmetry group $S^{\ell+1}$ on $M^{\ell+1}$;
    \item The vector fields ${X}^{[\ell+1]}_{t}$ are projectable under the projections $\mathfrak{pr}^\alpha_{0}$ given by \eqref{proj2} and $\mathfrak{pr}^\alpha_{0*}{X}^{[\ell+1]}_t=X_t$ for every $t\in \mathbb{R}$.
\end{enumerate}
\qeddiamond\end{definition}

Several lemmas, covered in \cite{CL_11,CGM_07}, describe properties of the diagonal projection necessary to formulate the following proposition (see \cite[p.\,18--21]{CL_11} for details).

\begin{proposition}
\label{proposition}
    For every family of linearly independent (over $\mathbb{R}$) vector fields $X_{1},\ldots , X_{r}$\newline$ \in  \mathfrak{X}(M)$, there exists an integer $\ell$ with $\ell \dim M  \geq r$ such that their prolongations to $M^{\ell}$ are linearly
independent at a generic point.
\end{proposition}

In the contact setting, the diagonal prolongation of a contact Lie  system is not a contact Lie system because the manifold of the diagonal prolongation is not always odd. Nevertheless, $k$-contact distributions allow  for a diagonal prolongation procedure as follows.

\begin{proposition} The diagonal prolongation to $N^{\ell+1}$ of a $k$-contact Lie system on $N$ is a $k(\ell+1)$-contact Lie system on $N^{\ell+1}$.
\end{proposition}
\begin{proof}
Let $X_{1}, \ldots, X_{r}$ span the smallest Lie algebra $V^{X}$ of the $k$-contact Lie system $(N,\bm\eta_{N},X)$. 
 There exists the $k(\ell+1)$-contact form 
$$
\bm\eta_{N^{\ell+1}}=\sum_{\alpha=0}^{\ell}\bm\eta_{N}^{\alpha}=\sum_{\alpha=0}^{\ell}\sum_{\beta=1}^{k}\eta^{\alpha}_{\beta}\otimes e_{\alpha k+\beta},
$$
where $\bm\eta^{\alpha}_{N}=\mathfrak{pr}^{\alpha*}_0\bm\eta_N$. Indeed, $$\ker\bm\eta_{N^{\ell+1}}=\ker \bm\eta^0_{N}\cap \ker \bm\eta^1_{N}\cap\dotsb\cap \ker\bm\eta_N^{\ell},$$ 
$$\ker \d \bm\eta_{N^{\ell+1}}=\ker \d\bm\eta^0_{N}\cap \ker \bm\d\eta^1_{N}\cap\dotsb\cap \ker\d\bm\eta_N^{\ell}$$ 
and $\ker\bm\eta_{N^{\ell+1}} \cap\ker \d \bm\eta_{N^{\ell+1}}=0$. Moreover, $\ker \bm \eta_{N^{\ell+1}}$ has corank $(\ell+1) k$ and $\ker \d \bm \eta_{N^{\ell+1}}$ has rank $k(\ell+1)$. Hence, we have obtained a $k(\ell+1)$-contact form.

By contracting each $X_i^{[\ell+1]}=\sum_{\alpha=0}^\ell X^{(\alpha)}_i$, where  $X^{(\alpha)}_i=X_i(x_{(\alpha)})$ and $i=1,\ldots, r$, with $\bm\eta_{N^{\ell+1}}$ and its differential

$$
\iota_{X_i^{[\ell ]}}\bm\eta_{N^{\ell }}=\sum_{\alpha, \beta=0}^{\ell}\iota_{X_i^{(\alpha)}}\bm\eta^{\beta}_{N}= \sum_{\alpha=0}^{\ell}\iota_{X_i^{(\alpha)}}\bm\eta^{\alpha}_{N}=-\sum_{\alpha=0}^{\ell}\bm h^i_{\alpha}, \qquad i=1,\ldots, r,
$$
$$
\iota_{X_i^{[\ell]}}\d\bm\eta_{N^{\ell+1}}=\sum_{\alpha=0
}^{\ell}\iota_{X_i^{(\alpha)}}\d\bm\eta^{\alpha}_{N}=\sum_{\alpha=0}^{\ell }\d\bm h^i_{\alpha}-\sum_{\mu=0}^{\ell}\sum_{\beta=1}^{k}R_{\beta}^{\alpha}\bm h_{\alpha}^i\eta_{\beta}^{\alpha}, \qquad i=1,\ldots,r,
$$
where $R_{\beta}^{\alpha}$, for $\beta=1,\ldots, k,$ are Reeb vector fields on $(\alpha+1)$-th copy of $N$ in $N^{\ell+1}$. Thus, we obtain that each prolongation is an $\bm \eta_{N^{\ell+1}}$-Hamiltonian vector field with the $\bm \eta_{N^{\ell+1}}$-Hamiltonian $k(\ell+1)$-function $\bm h^i=\sum_{\alpha=0}^{\ell }\bm h^i_{\alpha}$ for $ i=1,\ldots,r.$ .
\end{proof}

The reason why the procedure for the diagonal prolongations of contact Lie systems relied on passing to the use of Jacobi setting was because this has a good diagonal prolongation related to Jacobi manifolds. This is no longer necessary in the above, more natural, realm. 

There is another relevant fact about the fact that the diagonal prolongations to $N_{\ell+1}$  of $\bm \eta$-Hamiltonian vector fields are $\bm \eta_{N^{\ell+1}}$-Hamiltonian. Consider a $k$-contact Lie system $(M,\bm\eta,X)$ and an $\bm \eta$-Hamiltonian $k$-function $\bm C$. If $\{h^\theta,C^\theta\}
_\theta=0$, then the $\bm \eta_{N^{\ell+1}}$-Hamiltonian functions of the diagonal prolongations to $N^{\ell+1}$ of $X_{\bm h}$ and $X_{\bm C}$ do also commute, which can be used to obtain constants of motion of  $X^{[\ell+1]}_{\bm h}$.  This may have applications to obtain superposition rules for $X_{\bm h}$, for instance.

\section{Applications of {\it k}-contact Lie systems}\label{Sec:applications}

This section provides several applications of $k$-contact Lie systems and illustrates the results obtained in previous sections.  

\subsection{A control Lie system}
Let us endow a Lie system related to control systems with a $k$-contact form. Let us consider the system of differential equations on $\mathbb{R}^5$ given by
\begin{equation}\label{eq:Con}
\frac{\d x}{\d t} = \sum_{\alpha=1}^5b_\alpha(t)X_\alpha\,,
\end{equation}
where $b_1(t),\ldots,b_5(t)$ are arbitrary $t$-dependent functions and
$$
\begin{gathered}
 X_1=\frac{\partial}{\partial x_1}\,,\qquad X_2=\frac{\partial}{\partial x_2}+x_1\frac{\partial}{\partial x_3}+x_1^2\frac{\partial}{\partial x_4}+2x_1x_2\frac{\partial}{\partial x_5}\,,\\
 X_3=\frac{\partial}{\partial x_3}+2x_1\frac{\partial}{\partial x_4}+2x_2\frac{\partial}{\partial x_5}\,,\qquad X_4=\frac{\partial}{\partial x_4}\,,\qquad X_5=\frac{\partial}{\partial x_5}.
 \end{gathered}
$$

The above vector fields span a nilpotent Lie algebra $V_c$ of vector fields whose non-vanishing commutation relations read
 \[
    [X_1,X_2]=X_3\,,\qquad [X_1,X_3]=2X_4\,,\qquad [X_2,X_3]=2X_5\,.
 \]
This makes  \eqref{eq:Con} into a Lie system related to a $t$-dependent vector field $\sum_{\alpha=1}^5b_\alpha(t)X_\alpha$ as proved in \cite{Ram_02}. The initial motivation to study \eqref{eq:Con} comes from the fact that it covers as a particular case the system of differential equations on $\mathbb{R}^5$ of the form 
\begin{equation}\label{eq:controlnewold}
    \frac{\d x_1}{\d t}=b_1(t)\,,\quad \frac{\d x_2}{\d t}=b_2(t)\,,\quad
    \frac{\d x_3}{\d t}=b_2(t)x_1\,,\quad \frac{\d x_4}{\d t}=b_2(t)x_1^2\,, \quad \frac{\d x_5}{\d t}=2b_2(t)x_1x_2\,,
\end{equation}
where $b_1(t)$ and $b_2(t)$ are arbitrary $t$-dependent functions whose interest is due to their relation to certain control problems \cite{Nik_00,Ram_02}. Let us define $X_c=b_1(t)X_1+b_2(t)X_2$.

Since $X_1\wedge\dotsb\wedge X_5\neq 0$ and $\dim V_c=5$, one obtains a locally automorphic Lie system $(\mathbb{R}^5,X_c,V_c)$. 
Then, it can be proved that the Lie algebra of Lie symmetries of the vector fields of $V_c$ is isomorphic to it \cite{GLMV_19}. Indeed,  consider the Lie algebra of Lie symmetries of $V_c$ given by vector fields
\begin{gather}
Y_1 = \frac{\partial}{\partial x_1} + x_2\frac{\partial}{\partial x_3} + 2x_3\frac{\partial}{\partial x_4} + x_2^2\frac{\partial}{\partial x_5}\,,\qquad Y_2=\frac{\partial}{\partial x_2}+2x_3\frac{\partial}{\partial x_5}\,,\\Y_3=\frac{\partial}{\partial x_3}\,,\qquad Y_4=\frac{\partial}{\partial x_4}\,,\qquad Y_5=\frac{\partial}{\partial x_5}.
\end{gather}
More exactly, $[Y_i,X_j]=0$ for $i,j=1,\ldots,5$. Moreover, $Y_1,\ldots,Y_5$ span a Lie algebra with opposite structure constants to those of $X_1,\ldots,X_5$. 
The $Y_1,\ldots,Y_5$ admit the corresponding dual one-forms given by
\begin{gather}
    \Upsilon^1 = \d x_1\,,\qquad \Upsilon^2 = \d x_2\,,\qquad \Upsilon^3 = -x_2\d x_1 + \d x_3\,,\\
    \Upsilon^4 = -2x_3\d x_1 + \d x_4\,,\qquad 
    \Upsilon^5 = -x_2^2\d x_1-2x_3\d x_2 + \d x_5\,.
\end{gather}

The existence of these dual forms follows by the condition that $Y_1\wedge\dotsb\wedge Y_5$ is non-vanishing on a manifold of dimension five. Additionally,  $\Upsilon^1,\dotsc,\Upsilon^5$ are invariant relative to the vector fields $X_1,\dotsc,X_5$, namely $\Lie_{X_i}\Upsilon^j=0$ for $i,j=1,\ldots,5$. The differentials of $\Upsilon^1,\dotsc,\Upsilon^5$ depend on the commutation relations of the vector fields $X_1,\dotsc,X_5$, namely
$$
\d\Upsilon^\gamma = \frac 12 \sum_{\alpha,\beta=1}^5c_{\alpha\beta}\,\!^\gamma\Upsilon^\alpha\wedge\Upsilon^\beta\,,\qquad \alpha=1,\ldots,5\,,
$$
where $c_{\alpha\beta}\,\!^\gamma$ are the structure constants of our problem, namely $[X_\alpha,X_\beta]=\sum_{\gamma=1}^5c_{\alpha\beta}\,\!^\gamma X_\gamma$ for $\alpha,\beta=1,\ldots,5.$
In particular, for every differential one-form $\Upsilon \in \langle\Upsilon^1,\ldots,\Upsilon^5 \rangle$, its differential is both closed and invariant relative to the vector fields $X_1,\ldots,X_5$. Consequently, these vector fields become $\d \Upsilon$-Hamiltonian relative to such presymplectic forms $\Upsilon$.  

One has the following relations
$$
 \d \Upsilon^1=0\,,\qquad \d\Upsilon^2=0\,,\qquad\d \Upsilon^3 = \Upsilon^1\wedge\Upsilon^2\,,\qquad\d\Upsilon^4 = 2\Upsilon^1\wedge \Upsilon^3\,,\qquad\d \Upsilon^5 = 2\Upsilon^2\wedge \Upsilon^3\,.
 $$
Let us prove that $\bm\eta_c=\Upsilon^{4}\otimes e_1+\Upsilon^{5}\otimes e_2$ gives rise to a two-contact form. Indeed,
\[
\ker \bm\eta_c=\langle Y_1,Y_2,Y_3\rangle=\left\langle\frac{\partial}{\partial x_1} + x_2\frac{\partial}{\partial x_3} + 2x_3\frac{\partial}{\partial x_4} + x_2^2\frac{\partial}{\partial x_5}\ ,\ \frac{\partial}{\partial x_2}+2x_3\frac{\partial}{\partial x_5}\ ,\ \frac{\partial}{\partial x_3} \right\rangle\,. 
\]
Furthermore, $\ker\d\bm\eta_c=\ker\d \Upsilon^{4}\cap\ker\d \Upsilon^{5}$ is a distribution spanned by
\[
\ker\d\bm \eta_c=\langle Y_4,Y_5\rangle=\left<\frac{\partial}{\partial x_{4}},\frac{\partial}{\partial x_{5}} \right>.
\]
Clearly, $\ker\d\bm\eta_c\cap \ker\bm\eta_c=\{0\}$ and  $\bm\eta_c$ is a two-contact form. Moreover, $R_1=Y_4,R_2=Y_5$ are the associated Reeb vector fields. Let us now find the corresponding $\bm\eta_c$-Hamiltonian two-functions  $\bm h_{1},\ldots,\bm h_{5}$. By verifying the conditions
\begin{equation}
\label{eq:conditions}
\iota_{X_i}\d\Upsilon^\mu=\d h_i^\mu-\sum_{\alpha=4}^5(R_{\alpha-3} h^\mu_i)\Upsilon^\alpha, \qquad
\iota_{X_i}\Upsilon^\mu=-h^\mu_i,\qquad \mu=4,5,\qquad i=1,\ldots,5,
\end{equation}
 for $X_{1},\ldots,X_{5}$, we get that
\begin{gather*}
\bh_{1}=2x_{3}   e_{1} + x_{2}^{2}   e_{2}\,, \qquad 
\bh_{2}=-x_{1}^{2}   e_{1} + (2x_{3}-2x_{1}x_{2})   e_{2}\,,\\
\bh_{3}=-2x_{1}   e_{1} - 2x_{2}   e_{2}\,, \qquad 
\bh_{4}=-   e_{1}\,,  \qquad
\bh_{5}=-   e_{2}\,.
\end{gather*}

In particular, $X_4=Y_4$ and $X_5=Y_5$, therefore, $X_4$ and $X_5$ have constant Hamiltonian two-functions, as it corresponds to two-contact Reeb vector fields. It is worth noting that all the above $\bm\eta_c$-Hamiltonian two-functions are first integrals of the Reeb vector fields, thus \eqref{eq:Con} becomes a two-contact projectable Lie system relative to $\bm \eta_c$. The system \eqref{eq:Con} gives rise to a two-contact Lie--Hamiltonian system  $(\mathbb{R}^5,{\bm \eta}_c,{\bm h}=\sum_{\alpha=1}^5b_\alpha(t){\bm h}_\alpha)$ with a two-contact Lie--Hamilton algebra $\mathfrak{W}_c=\langle {\bm h}_1,\ldots,{\bm h}_5\rangle$ relative to ${\bm \eta}_c$. It is worth noting that $X_c$ is invariant relative to the Reeb vector fields of ${\bm \eta}_c$. Moreover, ${\bm h}$ is also invariant relative to the Reeb vector fields of $\bm \eta_c$. In particular, system \eqref{eq:controlnewold} becomes a projectable two-contact Lie system relative   to ${\bm \eta}_c$ and $(\mathbb{R}^5,\bm\eta_c,{\bm h}_c=b_1(t){\bm h}_1+b_2(t){\bm h}_2)$ is a two-contact Lie--Hamiltonian system for $(\mathbb{R}^5,\bm\eta_c,X_c)$ with a two-contact Lie--Hamilton algebra $\mathfrak{W}_c=\langle {\bm h}_1,\ldots,{\bm h}_5\rangle$. 

It is worth emphasising that whether a given $k$-contact Lie system is projectable relative to a compatible $k$-contact form. To illustrate this, let us show that $X_c$ is not projectable relative to another compatible three-contact form constructed via the methods given in Section \ref{Sec:NewMethods}.  Note that the distribution spanned by $\langle Y_{1},Y_{2}\rangle$ is maximally non-integrable and invariant relative to $X_1,\ldots,X_5$. Moreover, it admits commuting Lie symmetries, namely $X_{3},X_{4},X_{5}$, which turns $\langle Y_{1},Y_{2}\rangle$ into a three-contact distribution.

Let us consider the dual forms to $Y_{1},Y_{2},X_{3},X_{4},X_{5}$, namely 
$$
\left(\eta'\right)^1=\d x_{1}\,, \qquad \left(\eta'\right)^2=\d x_{2}\,, \qquad \left(\eta'\right)^3=-x_{2}\d x_{1}+\d x_{3}\,,
$$
$$
\left(\eta'\right)^{4} =2\left(x_{1}x_{2}- x_{3}\right) \d x_{1}-2x_{1} \d x_{3}+\d x_{4}\,, \qquad \left(\eta'\right)^5=x_{2}^{2}\d x_{1}-2x_{3}\d x_{2}-2x_{2}\d x_{3} + \d x_{5}\,,
$$
whose non-zero differentials read
$$
\d \left(\eta'\right)^{3}=\d x_{1}\wedge \d x_{2}\,, \qquad \d \left(\eta'\right)^{4}=-2x_{1}\d x_{1}\wedge \d x_{2}\,,\qquad 
\d\left(\eta'\right)^{5}=-2x_{2}\d x_{1}\wedge \d x_{2}\,.
$$
One can construct the three-contact form
$$
\bm \eta_c' = \sum_{\alpha=3}^{5}\left(\eta'\right)^{\alpha} \otimes e_{\alpha}\,,
$$
satisfying that $\ker \bm \eta_c'= \langle Y_{1}, Y_{2}\rangle$ and $\ker \d \bm \eta_c'=\langle X_{3},X_{4}, X_{5}\rangle$.
As $X_{3},X_{4},X_{5}$ span $\ker \d \bm \eta_c'$ and are dual to $\bm\eta_c'$, they are Reeb vector fields. 

Note that $X_{1}\ldots,X_{5}$ are $\bm\eta_c'$-Hamiltonian vector fields with $\bm\eta_c'$-Hamiltonian three-functions given by
$$
\bm h_{1}'=x_{2}  e_{3}-2\left(x_{1}x_{2}-x_{3}\right)   e_{4} - x_{2}^{2}  e_{5}\,, \qquad \bm h_{2}'=-x_{1}  e_{3} +x_{1}^{2}  e_{4} +2x_{3}  e_{5}.
$$
$$
\bm h_{\mu}'= -   e_{\mu}\,, \qquad \mu =3,4,5\,.
$$
Clearly not all $\bm\eta_c'$-Hamiltonian three-functions are first integrals of all Reeb vector fields, e.g. $X_{3}{\bm h}_{2}'=-2e_{5}$.
One can verify that $\bm\eta_c$-Hamiltonian three-functions close a Lie algebra under the Lie bracket given by \eqref{eq:Lie_bracket_Ham} with opposite constant as their respective $\bm\eta_c$-Hamiltonian vector fields. Indeed, the non-vanishing commutation relations read
$$
\{\bm h_{1}',\bm h_{2}'\}=-\bm h_{3}'\,, \qquad \{ \bm h_{1}',\bm h_{3}'\}=-2\bm h_{4}'\,, \qquad \{\bm h_2',\bm h_{3}'\}=-2\bm h_{5}'\,.
$$
Note that ${\bm h}'_4$ and ${\bm h}'_5$ are dissipated functions for the three-contact Lie system  $(\mathbb{R}^5,\bm\eta'_c,X_c)$. In fact, their associated ${\bm \eta}_c'$-Hamiltonian vector fields are Lie symmetries of the corresponding three-contact Lie system.

Note that $[V_c,[V_c,[V_c,V_c]]]=0$ and \eqref{eq:controlnewold} is projectable relative to the Reeb vector fields of $\bm \eta_c$. Using Remark \ref{Rem:SpecialCases}, one obtains that, if $b_1(t),b_2(t)$ are assumed to be constants, then
$$
\frac{\d^3{\bm I}}{\d t^3}=0,\qquad \forall {\bm I}\in \mathfrak{W}_c,
$$
which gives lots of $t$-dependent constants of motion of such systems. Hence, every ${\bm I}$ is related to a generator of constants of motion of order, at most, three. In particular, the form of the ${\bm I}\in \mathfrak{M}_c$, implies that solutions to \eqref{eq:controlnewold} are such that their coordinates $x_1(t),x_2(t),x_3(t)$ are polynomial functions on $t$ up to order two. Moreover, one also gets that
$$
\frac{\d^3(x_1x_2)}{\d t^3}=0,
$$
which gives additional restrictions on the form of $x_1(t),x_2(t)$ for particular solutions. 

\subsection{The complex Schwarz equation}

The Schwarz derivative plays a significant role in studying the linearisation in $t$-dependent systems, projective systems, the study of special functions, etcetera \cite{GR_07, Hil_97, Leh_79}. It is particularly related to the $t$-dependent complex differential equation given by
\begin{equation}\label{Eq:ComplSchwa}
    \frac{\mathrm{d} z}{\mathrm{d} t} = v\,,\qquad \frac{\mathrm{d} v}{\mathrm{d} t} = a\,,\qquad\frac{\mathrm{d} a}{\mathrm{d} t} = \frac 32\frac{a^2}{v} + 2b(t)v\,,\qquad z,v,a\in \mathbb{C}\,,\qquad t\in \mathbb{R},
\end{equation}
for a certain complex $t$-dependent function $b(t)$. System \eqref{Eq:ComplSchwa} can be understood as a complex analogue of the Lie system on $\mathcal{O}=\{(z,v, a)
\in \T^2\mathbb{R}:v
\neq 0\}$ studied in \cite{LV_15}. It represents the complex equation $\{z,t\}_{sc}=2b(t)$, where $\{\cdot,\cdot\}_{sc}$ denotes the so-called {\it Schwarz derivative}.
$$
\{z,t\}_{sc}=\frac{\d^3 z}{\d t^3}\left(\frac{\d z}{\d t}\right)^{-1}-\frac 32 \left(\frac{\d^2z}{\d t^2}\right)^2\left(\frac{\d z}{\d t}\right)^{-2}=2b(t)\,.
$$
In fact, \eqref{Eq:ComplSchwa} more accurately describes the standard Schwarz derivative compared to the real version in 
\cite{LV_15}. Note that \eqref{Eq:ComplSchwa} is a differential equation on $\mathcal{O}_{sc}=\{(z,v,a)
\in \T^2\mathbb{C}:v
\neq 0\}$.

Although \eqref{Eq:ComplSchwa} can be considered as a complex Lie system, we will not develop the due theory here (which requires defining complexifications of tangent bundles and other related notions). Moreover, it can be proved that the complex approach simplifies only a part of the theory, making other computations more complicated. 

In real coordinates,
\[
v_1 = \Rp (z)\,,\quad v_2 = \Ip(z)\,,\quad v_3 = \Rp(v)\,,\quad v_4 = \Ip(v)\,,\quad v_5 = \Rp(a)\,,\quad v_6 = \Ip(a)\,,
\]
system \eqref{Eq:ComplSchwa} is associated with the $t$-dependent vector field
$$
X_{sc}=X_1+2b_R(t)X_2+2b_I(t)X_3\,,
$$
where $b_R(t) = \Rp(b(t))$, $b_I(t) = \Ip(b(t))$, and
{\footnotesize\begin{gather*}
   X_1 = v_3\frac{\partial}{\partial v_{1}} + v_4\frac{\partial}{\partial v_{2}} + v_5\frac{\partial}{\partial v_{3}} + v_6\frac{\partial}{\partial v_{4}} + \frac{3}{2}\frac{2v_4v_5v_6 + (v_5^2 - v_6^2)v_3}{v_3^2 + v_4^2}\frac{\partial}{\partial v_{5}} + \frac{3}{2}\frac{2v_3v_5v_6 - v_4(v_5^2 - v_6^2)}{v_3^2 + v_4^2}\frac{\partial}{\partial v_{6}}\,,\\ 
        X_2 = v_3\frac{\partial}{\partial v_{5}} + v_4\frac{\partial}{\partial v_{6}}\,,\qquad           
    X_3 = -v_4\frac{\partial}{\partial v_{5}} + v_3\frac{\partial}{\partial v_{6}}\,,\\
        X_4 = -v_3\frac{\partial}{\partial v_{3}} - v_4\frac{\partial}{\partial v_{4}} - 2v_5\frac{\partial}{\partial v_{5}} - 2v_6\frac{\partial}{\partial v_{6}}\,,\qquad                                       
    X_5 = v_4\frac{\partial}{\partial v_{3}} - v_3\frac{\partial}{\partial v_{4}} + 2v_6\frac{\partial}{\partial v_{5}} - 2v_5\frac{\partial}{\partial v_{6}}\,,\\
        X_6 = -v_4\frac{\partial}{\partial v_{1}} + v_3\frac{\partial}{\partial v_{2}} - v_6\frac{\partial}{\partial v_{3}} + v_5\frac{\partial}{\partial v_{4}} - \frac{3}{2}\frac{2v_3v_5v_6 - v_4(v_5^2 - v_6^2)}{(v_3^2 + v_4^2)}\frac{\partial}{\partial v_{5}} + \frac{3}{2}\frac{2v_4v_5v_6 + v_3(v_5^2 - v_6^2)}{(v_3^2 + v_4^2)}\frac{\partial}{\partial v_{6}}\,.
\end{gather*}}
These vector fields, defined for $v_3^2+v_4^2\neq 0$ satisfy the following commutation relations
\begin{align}
    [X_1,X_2] &= X_4\,, & [X_1,X_3] &= X_5\,, & [X_1,X_4] &= X_1\,, & [X_1,X_5] &= X_6\,, & [X_1,X_6] &= 0\,,\\
    &&[X_2,X_3] &= 0\,, & [X_2,X_4] &= -X_2\,, & [X_2,X_5] &= -X_3\,, & [X_2,X_6] &= -X_5\,,\\
    &&&&[X_3,X_4] &= -X_3\,, & [X_3,X_5] &= X_2\,, & [X_3,X_6] &= X_4\,,\\
    &&&&&&[X_4,X_5] &= 0\,, & [X_4,X_6] &= -X_6\,,\\
    &&&&&&&&[X_5,X_6] &= X_1\,,
\end{align}
Then, $X_1,\ldots,X_k$ span a Lie algebra $V_{sc}$ that is
isomorphic to $\mathfrak{sl}_2(\mathbb{C})=\mathbb{C}\otimes \mathfrak{sl}_2$ as a real vector space. Indeed, $\langle X_1,X_2,X_4\rangle \simeq \mathfrak{sl}_2(\R)\simeq\langle X_3,X_4,X_6\rangle $. Additionally, $\mathbb{C}\otimes \mathfrak{sl}_2(\R)$ decomposes as $\langle X_1,X_4,X_2\rangle\oplus \langle X_6,X_5,X_3\rangle$. Then,  $V_{sc}=E_{-1}\oplus E_0\oplus E_1$, where $E_{-1}=\langle X_6,X_1\rangle$, $E_0=\langle X_4,X_5\rangle$, and $E_1=\langle X_3,X_2\rangle$, with $[E_i,E_j]=E_{i+j}$, where the sum $i+j$ is taken relative to the additive group $
\{-1,0,1\}$ and the direct sum is relative to subspaces of $V_{sc}$. Relevantly, $X_1\wedge\dotsb\wedge X_6\neq 0$ on every point of $\mathcal{O}_{sc}$. This allows us to map, at least locally, these vector fields diffeomorphically into the right-invariant vector fields of a basis of a Lie group with Lie algebra isomorphic to $V_{sc}$. This relation can be used to obtain Lie symmetries of these vector fields or differential forms that are invariant relative to $X_1,\ldots,X_6$.

Meanwhile, the Lie algebra of symmetries of the system \eqref{Eq:ComplSchwa} reads 
\begin{equation}\label{eq:Liealgebra}
\begin{gathered}
2Y_1 =(v_1^2-v_2^2)\frac{\partial}{\partial v_1}+2v_1v_2\frac{\partial}{\partial v_2}+(2v_1v_3-2v_2v_4)\frac{\partial}{\partial v_3} + 2(v_3v_2+v_1v_4)\frac{\partial}{\partial v_4}\\ \quad + 2(v_3^2+v_1v_5-v_4^2-v_2v_6)\frac{\partial}{\partial v_5}+2(v_3^2+v_1v_5-v_4^2-v_2v_6)\frac{\partial}{\partial v_6}\,,
\\Y_2=\frac{\partial}{\partial v_1}\,,\qquad
Y_3 = \frac{\partial}{\partial v_2}\,,\\
Y_4 = -v_1\frac{\partial}{\partial v_1}-v_2\frac{\partial}{\partial v_2}-v_3\frac{\partial}{\partial v_3}-v_4\frac{\partial}{\partial v_4}-v_5\frac{\partial}{\partial v_5}-v_6\frac{\partial}{\partial v_6}\,,\\
Y_5 = v_2\frac{\partial}{\partial v_1}-v_1\frac{\partial}{\partial v_2}+v_4\frac{\partial}{\partial v_3}-v_3\frac{\partial}{\partial v_4}+v_6\frac{\partial}{\partial v_5}-v_5\frac{\partial}{\partial v_6}\,,\\
2Y_6 = -2v_1v_2\frac{\partial}{\partial v_1}+(v_1^2-v_2^2)\frac{\partial}{\partial v_2}-2(v_2v_3+v_1v_4)\frac{\partial}{\partial v_3}+(2v_1v_3-2v_2v_4)\frac{\partial}{\partial v_4} \\
\quad -2(2v_3v_4+v_2v_5+v_1v_6)\frac{\partial}{\partial v_5}+2(v_3^2-v_4^2+v_1v_5-v_2v_6)\frac{\partial}{\partial v_6}\,.
\end{gathered}
\end{equation}
The commutation relations are
\begin{align}
    [Y_1,Y_2] &= Y_4\,, & [Y_1,Y_3] &= Y_5\,, & [Y_1,Y_4] &= Y_1\,, & [Y_1,Y_5] &= Y_6\,, & [Y_1,Y_6] &= 0\,,\\
    &&[Y_2,Y_3] &= 0\,, & [Y_2,Y_4] &= -Y_2\,, & [Y_2,Y_5] &= -Y_3\,, & [Y_2,Y_6] &= -Y_5\,,\\
    &&&&[Y_3,Y_4] &=-Y_3\,, & [Y_3,Y_5] &=Y_2\,, & [Y_3,Y_6] &= Y_4\,,\\
    &&&&&&[Y_4,Y_5] &= 0\,, & [Y_4,Y_6] &= -Y_6\,,\\
    &&&&&&&&[Y_5,Y_6] &= Y_1\,.
\end{align}
Note that $Y_1,\ldots,Y_6$ admit identical structure constants as $X_1,\ldots,X_6$. One can choose one-forms $\Upsilon^1,\ldots,\Upsilon^6$ to be the dual to $Y_1,\ldots,Y_6$. These differential forms are locally diffeomorphic to a basis of right-invariant one-forms on a Lie group. The existence of these dual forms is ensured by the condition $Y_1\wedge \dotsb \wedge Y_6\neq 0$ and the dimension of the manifold $\mathbb{R}^6$. These dual forms remain invariant concerning the Lie derivatives of the vector fields $X_1\ldots, X_6$, i.e. $\Lie_{X_i}\Upsilon^j=0$ for $i,j=1,\ldots,6$.

Moreover, the differential forms $\mathrm{d}\Upsilon^1,\ldots, \mathrm{d}\Upsilon^6$, or their linear combinations, are closed differential forms that are invariant relative to the Lie derivatives along $X_1,\ldots, X_6$. Moreover (see Lemma \ref{Lem:MaurerCartan} in the Appendix), one has
$$
\d\Upsilon^i=-\frac 12\sum_{j,k=1}^{6}c_{jk}\,\!^i\Upsilon^j\wedge \Upsilon^k,\qquad i=1,\ldots,6\,.
$$
In particular, 
\begin{align}
    \mathrm{d}\Upsilon^1 &= -\Upsilon^5\wedge\Upsilon^6 - \Upsilon^1\wedge \Upsilon^4\,, & \mathrm{d}\Upsilon^2 &= -\Upsilon^3\wedge \Upsilon^5 - \Upsilon^4\wedge\Upsilon^2\,,\\
    \mathrm{d}\Upsilon^3 &= -\Upsilon^4\wedge\Upsilon^3-\Upsilon^5\wedge\Upsilon^2\,, & \mathrm{d}\Upsilon^4 &= -\Upsilon^1\wedge\Upsilon^2-\Upsilon^3\wedge\Upsilon^6\,,\\
     \mathrm{d}\Upsilon^5 &= -\Upsilon^1\wedge\Upsilon^3-\Upsilon^6\wedge\Upsilon^2\,, & \mathrm{d}\Upsilon^6 &= -\Upsilon^1\wedge\Upsilon^5-\Upsilon^6\wedge\Upsilon^4\,.
\end{align}
Hence,
$$
\sum_{\alpha=1}^6\d \Upsilon^\alpha\otimes e_\alpha=-\frac12\sum_{ \alpha,\beta, \gamma =1}^{6} c_{\beta\gamma}\,^\alpha\Upsilon^\beta\wedge \Upsilon^\gamma\otimes e_\alpha=:-\frac12\sum_{\mu,
\nu=1}^{6} [\Upsilon^\mu\otimes e_\mu,\Upsilon^\nu\otimes e_\nu]
$$
Then, for $\bm\Upsilon=\sum_{\alpha=1
}^6\Upsilon^\alpha\otimes e_\alpha$, one has 
$$\d\bm\Upsilon=-\frac 12[\bm\Upsilon,\bm\Upsilon]\Rightarrow \d\bm\Upsilon +\frac 12[\bm\Upsilon,\bm\Upsilon]=0\,,
$$
where $\{e_1,\ldots,e_6\}$ is a basis of the Lie algebra $\T_e\mathrm{SL}_2(\mathbb{C})$ with the commutation relations of the vector fields $X_1,\ldots,X_6$.  

Let us prove that $\Upsilon^4,\Upsilon^5, \d\Upsilon^4, \d\Upsilon^5$ give rise to a two-contact form $\bm \eta_{sc}=\Upsilon^4\otimes e_4+\Upsilon^5\otimes e_5$. Indeed, note that $Y^4,Y^5$ commute between themselves,  take values in $\ker \d\Upsilon^4\cap \ker \d\Upsilon^5$,  and span the Reeb distribution. Moreover, the space $\langle Y^1,Y^2,Y^3,Y^6\rangle$ span the associated two-contact distribution. It follows that $Y^4,Y^5$ are the Reeb vector fields of our two-contact distribution.
By verifying conditions \eqref{eq:conditions} for $X_{1},\ldots,X_{6}$, we obtain their $\bm\eta_{sc}$-Hamiltonian two-functions, which components are given by Table \ref{tab:eta-hami-fun-ele}. Hence, $(\mathcal{O}_{sc}, \bm \eta_{sc},X_{sc})$ becomes a two-contact Lie system with a two-contact Lie--Hamiltonian system $(\mathcal{O}_{sc},\bm \eta_{sc},{\bm h}_{sc}=
{\bm h}_1+2b_R(t){\bm h}_2+2b_I(t){\bm h}_3$).

\begin{table}[h!]
    
    \begin{center}
    {\footnotesize \begin{tabular}{|c|c|}
        \hline
          $h_{1}^{4}=\frac{\gamma^3}{2}\bigg(\Big[v_2 v_4 \left(v_4^2-3 v_3^2\right)-v_1 \left(v_3^3-3 v_3 v_4^2\right)\Big] v_6^2$ & $h_{1}^{5}=\frac{\gamma^3}{2}\bigg(\Big[v_1 v_4 \left(v_4^2-3 v_3^2\right)+v_2 \left(v_3^3-3 v_3 v_4^2\right)\Big] v_6^2-$ \\      
          $- 2v_6 \Big[v_5\Big(v_1 v_4\left( v_4^2-3 v_3^2\right)+v_2 \left(v_3^3-3 v_3 v_4^2\right)\Big)$&$ 2v_6\Big[v_5\Big(v_2 v_4 \left(v_4^2-3 v_3^2\right)-v_1 \left(v_3^3-3 v_3 v_4^2\right)\Big)$\\
          $+v_4 \left(v_3^2+v_4^2\right)^2\Big]+v_5 \Big[v_{5}\Big(v_1 \left(v_3^3-3 v_3 v_4^2\right)$&$+v_3 \left(v_3^2+v_4^2\right)^2\Big]+v_5 \Big[2 v_4 \left(v_3^2+v_4^2\right)^2$ \\
          $-v_2 v_4 \left(v_4^2-3 v_3^2\right)\Big)-2 v_3 \left(v_3^2+v_4^2\right)^2\Big]\bigg)$&$+\Big(v_1 v_4 \left(v_4^2-3 v_3^2\right)+v_2 \left(v_3^3-3 v_3 v_4^2\right)\Big) v_5\Big]\bigg)$ \\
        \hline
        $h_{2}^{4}=\gamma\left(v_1 v_3+v_2 v_4\right)$ & $h_{2}^{5}=\gamma\left(v_2 v_3-v_1 v_4\right)$\\
        \hline
        $h_{3}^{4}=\gamma\left(v_1 v_4-v_2 v_3\right)$ &$h_{3}^{5}=\gamma\left(v_1 v_3+v_2 v_4 \right)$ \\
        \hline
         $h_{4}^{4}=-\gamma^2\bigg(v_3^4+\left(2 v_4^2-v_1 v_5+v_2 v_6\right) v_3^2$ & $h_{4}^{5}=-\gamma^2\bigg(v_1 \left(-v_6 v_3^2+2 v_4 v_5 v_3+v_4^2 v_6\right)$ \\
        $-2 v_4 (v_2 v_5+v_1 v_6) v_3+v_4^2\left(v_4^2+v_1 v_5-v_2 v_6\right)\bigg)$ &$+v_2 \left(-2 v_4 v_6 v_3+v_4^2 v_5-v_3^2 v_5\right)\bigg)$ \\
        \hline
         $h_{5}^{4}=-\gamma^2 \bigg(v_2 \left(\left(v_3^2-v_4^2\right) v_5+2 v_3 v_4 v_6\right)$&$h_{5}^{5}=-\gamma^2\bigg(v_3^4+\left(2 v_4^2-v_1 v_5+v_2 v_6\right) v_3^2$ \\
         $+v_1 \left(\left(v_3^2-v_4^2\right) v_6-2 v_3 v_4 v_5\right)\bigg)$ &$-2 v_4 (v_2 v_5+v_1 v_6) v_3+v_4^2 \left(v_4^2+v_1 v_5-v_2 v_6\right)\bigg)$\\
         \hline
         $h_{6}^{4}=\frac{\gamma^3}{2}\bigg(v_6^2\Big(v_1 v_4 \left(v_4^2-3 v_3^2\right)+v_2 \left(v_3^3-3 v_3 v_4^2\right)\Big) $&$h_{6}^{5}=\frac{\gamma^3}{2}\bigg(\left(v_2 v_4 \left(v_4^2-3 v_3^2\right)-v_1 \left(v_3^3-3 v_3 v_4^2\right)\right) v_6^2$ \\
         $+2v_{6} \Big[v_5\Big(v_2 v_4 \left(v_4^2-3 v_3^2\right)-v_1 \left(v_3^3-3 v_3 v_4^2\right)\Big) $ &$-2v_6 \Big[v_5\Big(v_1 v_4 \left(v_4^2-3 v_3^2\right)+v_2 \left(v_3^3-3 v_3 v_4^2\right)\Big)$\\
         $+v_3 \left(v_3^2+v_4^2\right)^2\Big]+v_5 \Big[-v_5\Big(v_1 v_4 \left(v_4^2-3 v_3^2\right)$&$+v_4 \left(v_3^2+v_4^2\right)^2\Big]+v_5 \Big[v_5\Big(v_1 \left(v_3^3-3 v_3 v_4^2\right)$\\
         $+v_2 \left(v_3^3-3 v_3 v_4^2\right)\Big)-2 v_4 \left(v_3^2+v_4^2\right)^2\Big]\bigg)$&$-v_2 v_4 \left(v_4^2-3 v_3^2\right)\Big)-2 v_3 \left(v_3^2+v_4^2\right)^2\Big]\bigg)$ \\
         \hline
         
    \end{tabular}}
    \end{center}
    \caption{Elements of $\bm\eta_{sc}$-Hamiltonian functions $\bm h_{1},\ldots,\bm h_{6}$ of the vector fields $X_{1},\ldots,X_{6}$, with $\gamma=\frac{-1}{v_{3}^2+v_{4}^{2}}$, for the complex Schwarz equation.}
    \label{tab:eta-hami-fun-ele}
\end{table}

\subsection{Generalisation to degree three of the Brockett system}\label{Sec:Brockett}

Let us study now two examples of  Lie systems, whose motivation is due to its interest in control theory, Wei--Norman equations, and the theory of Lie systems \cite{Ram_02,Mur_93,MS_93}.  
Both of them are control systems on $\R^8$. 

Let $\{x_1,\ldots,x_8\}$ be linear coordinates in $\R^8$. Consider the system 
\begin{align}
&\tder{x_1}=b_1(t)\,, &&\tder{x_2}=b_2(t)\,,
&&\tder{x_3}=b_2(t) x_1\,,&&\tder{x_4}=b_1(t) x_3\,,		\nonumber\\
&\tder{x_5}=b_2(t) x_3\,, &&\tder{x_6}=b_1(t) x_4\,,
&&\tder{x_7}=b_2(t) x_4\,,&&\tder{x_8}=b_2(t) x_5\,,		
\label{Mur_not_ster_1}
\end{align}
where $b_1(t)$ and $b_2(t)$ are the so-called control functions. From our point of view, they are just arbitrary $t$-dependent functions whose behaviour can lead to interesting features.

The solutions of \eqref{Mur_not_ster_1} are the integral curves of $t$-dependent
vector field $X_{3b}=b_1(t)\, X_1+b_2(t)\, X_2$, with
\begin{equation}X_1=\pd{x_1}+x_3\pd{x_4}+x_4 \pd{x_6}\,,			\qquad X_2=\pd{x_2}+x_1\pd{x_3}+x_3 \pd{x_5}
+x_4 \pd{x_7}+x_5 \pd{x_8}\,.				
\end{equation}
Taking the successive Lie brackets of $X_1,X_2$, we obtain the smallest Lie algebra of vector fields containing $X_1,X_2$. In fact,
$$
\begin{gathered}
X_3=[X_1,\,X_2]=\pd{x_3}-x_1 \pd{x_4}+x_3 \pd{x_7}\,,			 \qquad
X_4=[X_1,\,X_3]=-2 \pd{x_4}+x_1 \pd{x_6}\,,	\\ X_5=[X_2,\,X_3]=[X_1,X_6]=-\pd{x_5}+2 x_1 \pd{x_7}\,,\qquad 
X_6=[X_1,\,X_4]=3\pd{x_6}\,,\\X_7=[X_1,\,X_5]=2\pd{x_7}\,,\qquad
\quad X_8=[X_2,\,X_5]=\pd{x_8}\,,				
\end{gathered}
$$
along with $X_1,X_2$ define an eight-dimensional Lie algebra of vector fields $V_{3b}=\langle X_1,\,\dots,\,X_8\rangle$. It is worth noting that $X_1\wedge\dotsb\wedge X_8\neq 0$ and these vector fields span $\T \mathbb{R}^8$. Hence, $X_{3b}$ gives rise to a locally automorphic Lie system $(\mathbb{R}^8,X_{3b},V_{3b})$.  Moreover, the fact that the vector fields $X_1,\ldots,X_8$ span the tangent space mean that the system is {\it controllable}, which has special relevance in control theory \cite{Re_02,Ja_02}. 

The non-zero commutation relations between the elements of the given basis of $V$ read
$$
\begin{gathered}\,[X_1,\,X_2]=X_3\,,\qquad [X_1,\,X_3]=X_4\,,
\qquad[X_1,\,X_4]=X_6\,,\qquad[X_1,\,X_5]=X_7\,,\\\,[X_2,\,X_3]=X_5\,,\qquad [X_2,\,X_4]=X_7\,,\qquad [X_2,\,X_5]=X_8\,.	
\label{Mur_not_ster_1_alg_camp_vec}
\end{gathered}
$$
This implies that $V_{3b}$ is a nilpotent Lie algebra. 

As the Brockett system \eqref{Mur_not_ster_1} gives rise to a locally automorphic Lie system $(\mathbb{R}^8,X_{3b},V_{3b})$, it can be considered to be locally diffeomorphic to a so-called automorphic Lie system on a Lie group \cite{GLMV_19}. This means that one can  consider the Lie algebra of Lie symmetries of such vector fields, which will span a Lie algebra isomorphic to the previous one \cite{GLMV_19}. In fact, one can consider a basis $Y_1,\ldots,Y_8$ of Lie symmetries of $X_1,\ldots,X_8$ on $\mathbb{R}^8$ taking the same values as $X_1,\ldots,X_8$ at $0\in \mathbb{R}^8$. Hence, the structure constants for $Y_1,\ldots,Y_8$ are opposite to the ones for $X_1,\ldots,X_8$. To illustrate our theory,  the Lie symmetries of the vector fields $X_1,\ldots,X_8$ are spanned by the basis  given by
\begin{eqnarray}
&& Y_1=\frac{\partial}{\partial x_1}+x_2\frac{\partial}{\partial x_3}+\left(x_1x_2-x_3\right)\frac{\partial}{\partial x_4}+\frac{x_2^2}{2}\frac{\partial}{\partial x_5}+\frac {x_1(x_1x_2-x_3)-x_4}2\frac{\partial}{\partial x_6}\\
&& \qquad+\left(x_3x_2-2x_5\right)\frac{\partial}{\partial x_7}+\frac{x_2^3}{6}\frac{\partial}{\partial x_8}\,,			 \\
&& Y_2=\frac{\partial}{\partial x_2}\,, \qquad Y_3=\frac{\partial}{\partial x_3}+x_1\frac{\partial}{\partial x_4}+x_2\frac{\partial}{\partial x_5}+\frac{x_1^2}{2}\frac{\partial}{\partial x_6}+x_3\frac{\partial}{\partial x_7}+\frac{x_2^2}{2}
\frac{\partial}{\partial x_8}\,,	\\
&&Y_4=-2\frac{\partial}{\partial x_4}-2x_1\frac{\partial}{\partial x_6}-2x_2\frac{\partial}{\partial x_7}\,,\qquad Y_5=-\frac{\partial}{\partial x_5}-x_2\frac{\partial}{\partial x_8}\,, \qquad Y_6=3\frac{\partial}{\partial x_6}\,,\\
&&Y_7=2\frac{\partial}{\partial x_7}\,,\qquad Y_8=\frac{\partial}{\partial x_8}.
\end{eqnarray}
Previous vector fields take the same values as $X_1,\ldots,X_8$ at $0\in \mathbb{R}^8$ , namely $X_\alpha(0)=Y_\alpha(0)$ for $\alpha=1,\ldots,8$. Moreover, the non-vanishing commutation relations for $Y_1,\ldots,Y_8$ read
$$
\begin{gathered}
\,[Y_1,\,Y_2]=-Y_3\,,\qquad [Y_1,\,Y_3]=-Y_4\,,
\qquad[Y_1,\,Y_4]=-Y_6\,,\qquad[Y_1,\,Y_5]=-Y_7\,,\\
\,[Y_2,\,Y_3]=-Y_5\,,\qquad [Y_2,\,Y_4]=-Y_7\,,\qquad [Y_2,\,Y_5]=-Y_8\,.	
\end{gathered}\label{Mur_not_ster_1_alg_camp_vec_symmetires}
$$
Let us construct a six-contact distribution that is invariant relative to the action by Lie brackets of the vector fields $X_1,\ldots,X_8$. Note that any distribution spanned by $Y_1,\ldots,Y_8$ is invariant relative to the action of the vector fields $X_1,\ldots,X_8$. In particular, consider the distribution spanned by the vector fields $ Y_1,Y_2$. Since $[Y_1,Y_2]=-Y_3$ and $Y_1\wedge Y_2\wedge Y_3$ is never vanishing, the distribution is maximally non-integrable. 
Consider the basis of vector fields on $\mathbb{R}^8$ given by
$$
Y_1,Y_2,X_3,\ldots,X_8.
$$
On the one hand, $Y_1,Y_2$ span 
a maximally non-integrable distribution, while $
X_3,\ldots,X_8$ are Lie symmetries of $\langle Y_1, Y_2\rangle$ spanning an abelian Lie algebra and a distribution that is supplementary to the one spanned by $Y_1,Y_2$. This gives rise to a six-contact distribution on $\mathbb{R}^8$. Moreover, $X_1,\ldots,X_8$ leave invariant the distribution spanned by $Y_1,Y_2$, which turns them into six-contact Hamiltonian vector fields. 

The vector fields $X_3,\ldots,X_8$ are the Reeb vector fields of the six-contact form $\bm \eta_{3b}$ to be described next . Nevertheless, the $\bm\eta_{3b}$-Hamiltonian $k$-functions of $X_1,X_2$ will not be in general first integrals of the Reeb vector fields of $\bm\eta_{3b}$. In particular, the vector fields $X_3,\ldots,X_8$ do not commute with $X_1,X_2$, which implies that the $\bm\eta_{3b}$-Hamiltonian six-functions of $X_1,X_2$ are not first integrals of the Reeb distribution, which is spanned by $X_3,\ldots,X_8$.

It is worth noting that the potential applications of this method are quite large as the scheme used is quite general: it only depends on a subspace, which is not a Lie algebra, but is invariant relative to a supplementary commutative  Lie subalgebra. 

Then, the dual forms to $Y_1,Y_2,X_3,\ldots,X_8$ take the form
$$
\begin{gathered}
\eta^1=\d x_1\,,\qquad
\eta^2=\d x_2\,,\qquad	
 \eta^3=\d x_3-x_2\d x_1\,,\qquad\eta^4=\left(x_{1}x_{2}-\frac{1}{2}x_{3}\right)\d x_1 - \frac{1}{2}x_1 \d x_3- \frac{1}{2}\d x_{4} \,,\\
\eta^5=\frac {1}{2}x_2^2\d x_1 -\d x_5\,,\qquad
\eta^6=\frac{1}{6} \left(2x_{3}x_{1}-3 x_{2} x_{1}^2+x_{4}\right)\d x_1+\frac{1}6 x_1^2\d x_3+\frac{1}{6}x_1\d x_4+\frac{1}{3}\d x_6\,,\\
\eta^7=\left(x_5-\frac{1}{2}x_1x_2^{2}\right)\d x_1-\frac{1}{2}x_{3} \d x_{3}+x_{1}\d x_{5}+\frac{1}{2}\d x_7\,,\qquad
\eta^8=-\frac1{6}x_2^3\d x_{1} +\d x_8\,.
\end{gathered}
$$
It is immediate that $\eta^3,\ldots,\eta^8$  give rise to a six-contact form
$$
\bm \eta_{3b}=\sum_{\alpha=3}^8\eta^\alpha\otimes e_\alpha.
$$
In fact, $
\ker\bm\eta_{3b}=\langle Y_1,Y_2\rangle
$ and the $\bm \eta_{3b}$-Hamiltonian six-functions for $X_3,\ldots,X_8$ are easily given by 
$$
{\bm h}_{\alpha}=-e_\alpha\,,\qquad \alpha=3,\ldots,8\,,
$$
because $X_3,\ldots,X_8$ are the Reeb vector fields for $\bm \eta_{3b}$. Meanwhile, 
$$
{\bm h}_1=x_2e_3-(x_1x_2-x_3)e_4-\frac{1}{2}x_2^2e_5-\frac 12(x_1(x_3-x_1x_2)+x_4)e_6+(x_1x_2^2/2-x_5)e_7+\frac 16 x_{2}^{3}e_8,\qquad $$
$${\bm h}_2=-x_1e_3+
\frac 12 x_1^2e_4+x_3e_5-\frac{x_1^3}6e_6-
\frac 12(x_1x_3+x_4)e_7-x_5e_8.
$$
Note that these $\bm \eta_{3b}$-Hamiltonian six-functions are not first integrals of all the Reeb vector fields, e.g. $X_5\bm h_2=e_8$.  

To verify that all our results are correct in a new way, let us write $\d\bm\eta_{3b}$, which reads
$$
\d\bm\eta_{3b}=\d x_1\wedge \d x_2 \otimes \left(e_3-x_1e_4-x_2 e_5+\frac 12 x_1^2e_6+x_2x_1e_7+\frac 12x_2^2e_8\right),
$$
which means that $X_3,\ldots,X_8$ take values in $\ker \d\bm\eta_{3b}$ and are ``dual'' to $\bm \eta_{3b}$. Finally, $(\mathbb{R}^8,\bm \eta_{3b},{\bm h}_{3b}=b_1(t){\bm h}_1+b_2(t){\bm h}_2)$ is a six-contact Lie--Hamiltonian system for $X_{3b}$. Moreover,  $(\mathbb{R}^8,\bm \eta_{3b},X_{3b})$ is a six-contact Lie system.

\subsection{A new control system}
Another example of a Lie system, given in \cite[Section 7.2.3.2]{Ram_02}, is a control system in $\mathbb{R}^{8}$ given by
\begin{equation}
\begin{aligned}
\frac{\d x_1}{\d t} &= b_{1}(t)\,, \qquad \frac{\d x_2}{\d t} = b_{2}(t)\,, \qquad \frac{\d x_3}{\d t} = b_2(t)x_1-b_{1}(t)x_{2}\,, \\
\frac{\d x_4}{\d t} &= b_{2}(t)x_{1}^{2}\,, \qquad
\frac{\d x_{5}}{\d t} = b_{1}(t)x_{2}^{2}\,, \qquad \frac{\d x_{6}}{\d t} = b_2(t)x_1^3\,, \\
\frac{\d x_7}{\d t} &= b_{1}(t)x_2^{3}\,, \qquad \frac{\d x_8}{\d t} = b_1(t)x_{1}^{2}x_{2}+b_{2}(t)x_{1}x_{2}^{2}\,.
\end{aligned}
\label{eq:Brockett_2}
\end{equation}
where $b_{1}(t), b_{2}(t)$ are arbitrary $t$-dependent functions. The above system is related to a generalisation to third degree of a Brockett system \cite{BD_93}.

System \eqref{eq:Brockett_2} is associated with the $t$-dependent vector field $X_{g3b}=b_{1}(t)X_{1}+b_{2}(t)X_{2}$ on $\mathbb{R}^8$, with
\begin{align*}
    &X_{1}=\pd{x_{1}}-x_{2}\pd{x_{3}}+x_{2}^{2}\pd{x_{5}}+x_{2}^{3}\pd{x_{7}}+x_{1}^{2}x_{2}\pd{x_{8}}, \\
    &X_{2}=\pd{x_{2}}+x_{1}\pd{x_{3}}+x_{1}^{2}\pd{x_{4}}+x_{1}^{3}\pd{x_{6}}+x_{1}x_{2}^{2}\pd{x_{8}}.
\end{align*}
By considering the vector fields\footnote{Note that there exists a little mistake in the form of the vector field $X_5$ in \cite[pg. 185]{Ram_02}.}
$$\begin{gathered}
    X_{3}=[X_{1},X_{2}]=2\pd{x_{3}}+2x_{1}\pd{x_{4}}-2x_{2}\pd{x_{5}}+3x_{1}^{2}\pd{x_{6}}+\left(x_2^{2}-x_{1}^{2}\right)\pd{x_{8}},\\
    X_{4}=[X_{1},X_{3}]=2\pd{x_{4}}+6x_{1}\pd{x_{6}}-2x_{1}\pd{x_{8}},\qquad
    X_{5}=[X_{2}.X_{3}]=-2\pd{x_{5}}+2x_{2}\pd{x_{8}},\\
    X_{6}=[X_{1},X_{4}]=6\pd{x_{6}}-2\pd{x_{8}}\,, \qquad X_{7}=[X_2,X_{5}]=2\pd{x_{8}},
\end{gathered}
$$
we obtain a nilpotent Lie algebra $V_{g3b}=\langle X_1,\ldots,X_7\rangle$ whose structure defined by the non-vanishing commutation relations
\begin{gather} [X_{1},X_{2}]=X_{3}\,,\qquad [X_{1},X_{3}]=X_{4}\,, \qquad [X_{1},X_{4}]=X_{6}\,,\\ \left[X_{2},X_{3}\right]=X_{5}\,,\qquad [X_{2},X_{5}]=X_{7}.
\end{gather}
It should be noted that there exists another vector field, $X_{8}=\pd{x_{7}}$, which commutes with $X_{1},\ldots, X_{7}$ so that $\{X_{1},\ldots, X_{8}\}$ span the tangent space at each point of $\mathbb{R}^{8}$. This gives rise to a locally automorphic Lie system $(\mathbb{R}^8,X_{g3b},V_{g3b})$. Since every automorphic Lie systems is locally diffeomorphic to an automorphic Lie system and $V_{g3b}$ is mapped into the Lie algebra of right-invariant vector fields on a Lie group with Lie algebra isomorphic to $V_{g3b}$, there exists a Lie algebra of Lie symmetries of the vector fields of $V_{g3b}$ on $\mathbb{R}^8$  isomorphic to $V_{g3b}$, which are locally diffeomorphic to left-invariant vector fields on a Lie group $G$ with Lie algebra isomorphic to $V_{g3b}$ (see \cite{GLMV_19} for details). Indeed, the Lie symmetries of $X_{1},\ldots ,X_{8}$ are spanned by 
\begin{align*}
    &Y_{1}=\pd{x_{1}}+x_{2}\pd{x_{3}}+\left(x_{1}x_{2}+x_{3}\right)\pd{x_{4}}+3x_{4}\pd{x_{6}}+\left(x_{1}^{2}x_{2}+\frac{1}{3}x_{2}^{3}-x_{4}\right)\pd{x_{8}},\\
    &Y_{2}=\pd{x_{2}}-x_{1}\pd{x_{3}}+\left(x_{1}x_{2}-x_{3}\right)\pd{x_5}+\left(\frac{1}{3}x_{1}^{3}+x_{1}x_{2}^{2}-x_{5}\right)\pd{x_{8}}\,, \qquad Y_{3}=2\pd{x_{3}}\,,\\
    &Y_{4}=2\pd{x_{4}}\,, \qquad Y_{5}=-2\pd{x_{5}}\,,\qquad 
    Y_{6}=6\pd{x_{6}}-2\pd{x_{8}}\,, \qquad Y_{7}=2\pd{x_{8}}\,, \qquad Y_{8}=\pd{x_{7}}.
\end{align*}
Since $Y_{1},\ldots, Y_{8}$ are linearly independent,  there exist dual forms $\theta^{1},\ldots, \theta^{8}$. Therefore, we can consider the distribution $\mathcal{D}=\langle Y_{1}, Y_{2},Y_{3}\rangle$  as the kernel of five-contact form $\bm \theta=\sum_{i=4}^{8}\theta^{i} \otimes e_{i}$ and a vector bundle mapping in Definition \ref{def:max_non_integrable_mapping} is given by $\rho(v_{\alpha},v_{\beta})=\bm \theta([Z_{\alpha},Z_{\beta}])$ for $Z_\alpha, Z_\beta \in \mathcal{D}.$ As the degeneracy of $\rho$ is locally equivalent to existence of a non-zero vector field $Z \in \mathcal{D}$ such that $[Z,W] \in \mathcal{D}$ for all $W\in \D$, one can check that if $Z=c_{1}Y_{1}+c_{2}Y_{2}+c_{3}Y_{3}$, then 
$$
\D\ni [Z,Y_{1}]=[c_{1}Y_{1}+c_{2}Y_{2}+c_{3}Y_{3},Y_{1}]=c_{3}[Y_{3},Y_{1}]+c_2[Y_2,Y_1] \Rightarrow c_{3}=0, 
$$
$$
\D\ni [Z,Y_{3}]=[c_{1}Y_{1}+c_{2}Y_{2},Y_{3}]=c_{1}[Y_{1},Y_{3}]+c_{2}[Y_{2},Y_{3}] \Rightarrow c_{1}=c_{2}=0, 
$$ thus $\rho$ is non-degenerate and $\mathcal{D}$ is maximally non-integrable. Additionally, it admits five commuting Lie symmetries $X_{4},\ldots,X_{8}$, hence it is a five-contact distribution invariant with respect to $X_{1},\ldots,X_{8}$, which are $\bm\eta_{gb3}$-Hamiltonian with regard to five-contact form $\bm\eta_{gb3}$ to be described next.

Consider the dual forms to $Y_{1},\ldots,Y_{3},X_{4},\ldots,X_{8}$, namely
\begin{gather}
    \eta^{1}=\d x_{1}\,,\qquad
    \eta^{2}=\d x_{2}\,,\qquad
    \eta^{3}= -\frac{1}{2}x_{2}\d x_{1}+ \frac{1}{2}x_{1} \d x_{2}+\frac{1}{2}\d x_{3},\\
    \eta^{4}=-\frac{1}{2}(x_{1}x_{2}+x_{3})\d x_{1}+\frac{1}{2}\d x_{4},\qquad
    \eta^{5}=\frac{1}{2}\left(x_{1}x_{2}-x_{3}\right)\d x_{2}-\frac{1}{2}\d x_{5}, \\
    \eta^{6}= \frac{1}{2}\left(x_{1}\left(x_{1}x_{2}+x_{3}\right)-x_{4}\right)\d x_{1}-\frac{1}{2}x_{1}\d x_{4}+\frac{1}{6}\d x_{6}, \\
    \eta^{7}=-\frac{1}{6}x_{2}\left(3x_{1}^{2}+x_{2}^{2}\right)\d x_{1}+\frac{1}{6}\left(-x_{1}^{3}-6x_{1}x_{2}^{2}+3\left(x_{2}x_{3}+x_{5} \right)\right)\d x_{2}+\frac{1}{2}x_{2}\d x_{5}+ \frac{1}{6}\d x_{6}+\frac{1}{2}\d x_{8}, \\\eta^{8}=\frac 32(x_1x_2-x_3)\d x_2-\frac 32\d x_5+\d x_{7}.
\end{gather}
Dual forms associated with $X_{4},\ldots ,X_{8}$ give rise to a five-contact form $
\bm \eta_{g3b}=\sum_{\alpha=4}^{8}\eta^{\alpha} \otimes e_{\alpha},
$ 
and the $\bm\eta_{g3b}$-Hamiltonian functions of $X_{1},\ldots,X_{8}$ read
\begin{gather*}
\bm h_{1}=\frac{1}{2}(x_{1}x_{2}+x_{3})e_{4}+\frac{1}{2}x_{2}^{2}e_{5}+\frac{1}{2}(-x_{1}(x_{1}x_{2}+x_{3})+x_{4})e_6-\frac{1}{3}x_{2}^{3}e_7+\frac 12(3-2x_2)x_2^2e_{8}, \\
\bm h_{2}=-\frac{1}{2}x_{1}^{2}e_{4}+\frac{1}{2}(-x_{1}x_{2}+x_{3})e_5+\frac{1}{3}x_{1}^{3}e_{6}+\frac{1}{2}\left(x_{1}x_{2}^{2}-x_{2}x_{3}-x_{5}\right)e_{7}+\frac 32(-x_1x_2+x_3)e_8, \\
\bm h_{3}= -x_{1}e_{4}-x_{2}e_{5}+\frac{1}{2}x_{1}^{2}e_{6}+\frac{1}{2}x_{2}^{2}e_{7}\,, \qquad \bm h_{\alpha}=-e_{\alpha}, \qquad \alpha=4,\ldots,8.
\end{gather*}
By taking Lie brackets of $\bm h_{1}, \ldots, \bm h_{8}$, one obtain that their non-vanishing relations read
\begin{gather} \{\bm h_{1},\bm h_{2}\}=-\bm h_{3}\,,\qquad \{\bm h_{1},\bm h_{3}\}=-\bm h_{4}\,, \qquad \{\bm h_{1},\bm h_{4}\}=-\bm h_{6}\,,\\ \left\{\bm h_{2},\bm h_{3}\right\}=-\bm h_{5}\,,\qquad \{\bm h_{2},\bm h_{5}\}=-\bm h_{7}.
\end{gather}
Then, $X_{g3b}$ admits a five-contact Lie--Hamiltonian system $(\mathbb{R}^8,{\bm \eta}_{g3b},{\bm h}_{g3b}=b_1(t){\bm h}_1+b_2(t){\bm h}_2)$. In fact, $(\mathbb{R}^8,\bm \eta_{g3b},X_{g3b})$ is a five-contact Lie system.

\subsection{Front-wheel driven kinematic car}

Let us consider another example of control system, namely
\begin{equation}
\label{eq:wheel-car}
\tder{x}=c_{1}(t)\,, \qquad \tder{y}=c_{1}(t)\tan{\theta}\,, \qquad \tder{\phi}=c_{2}(t)\,, \qquad \tder{\theta}=c_{1}(t)\frac{\tan{\phi}}{\cos{\theta}}\,,
\end{equation}
defined on the configuration manifold $M=\R^{2} \times I^{2}$, where $I=\left(-\frac{\pi}{2}, \frac{\pi}{2}\right)$, with coordinates $\{x,y,\theta, \phi\}$, respectively, and   arbitrary $t$-dependent functions $c_1(t)$ and $c_2(t)$. System \eqref{eq:wheel-car} is of interest as describes a simple model of a car with front and rear wheels \cite{Ram_02}. Its associated $t$-dependent vector field is $X=c_{1}(t)Z_{1}+c_{2}(t)Z_{2}$, where
$$
Z_{1}=\pd{x}+ \tan{\theta}\pd{y}+\frac{\tan{\phi}}{\cos{\theta}}\pd{\phi}\,, \qquad Z_{2}=\pd{\phi}\,.
$$
By taking their commutators
$$
Z_{3}=[Z_{1}, Z_{2}]=-\frac{1}{\cos{\theta}\cos^{2}{\phi}}\pd{\theta}\,, \qquad Z_{4}=[Z_{1},Z_{3}]=\frac{1}{\cos^{3}{\theta}\cos^{2}{\phi}}\pd{y}\,,
$$
one sees that $Z_{1}, Z_{2}, Z_{3}, Z_{4}$ span $\T M$. These vector fields do not close any finite-dimensional Lie algebra, indeed, the successive Lie brackets of $Y_2$ with $Y_4$, namely $[Y_2,Y_4],[Y_2,[Y_2,Y_4]],\dots$, span an infinite-dimensional Lie algebra of vector fields. When $c_1(t)$ and $c_2(t)$ are not linearly dependent, one then has that \eqref{eq:wheel-car}
 is not a Lie system. 
 
 Nonetheless, \eqref{eq:wheel-car} can be transformed into a Lie system  related to a nilpotent Vessiot--Guldberg Lie algebra by using the following transformation (see \cite {CR_02} for details)
$$
c_{1}(t)=b_{1}(t)\,, \qquad c_{2}(t)=-3\sin^{2}{\phi}\frac{\tan{\theta}}{\cos{\theta}}b_{1}(t)+\cos^{3}{\theta}\cos^{2}{\phi}\,b_{2}(t)\,.
$$
Indeed, the successive change of coordinates
$$
x_{1}=x\,, \qquad x_{2}=\frac{\tan{\phi}}{\cos^{3}{\theta}}\,, \qquad x_{3}=\tan{\theta}\,, \qquad x_{4}=y\,,
$$
transforms \eqref{eq:wheel-car} into the control system on $\mathbb{R}^{4}$ given by
\begin{equation}
\label{eq:wheel_car_transformed}
\tder{x_{1}}=b_{1}(t)\,, \qquad \tder{x_{2}}=b_{2}(t)\,, \qquad \tder{x_{3}}=b_{1}(t)x_{2}\,, \qquad \tder{x_{4}}=b_{1}(t)x_{3}\,.
\end{equation}
 System \eqref{eq:wheel_car_transformed} describes the integral curves of the $t$-dependent vector field $X_{fw}=b_{1}(t)X_{1}+b_{2}(t)X_{2}$, with
$$
X_{1}=\pd{x_{1}}+x_{2}\pd{x_{3}}+x_{3}\pd{x_{4}}\,, \qquad X_{2}=\pd{x_{2}}\,.
$$
Along with the following vector fields
$$
X_{3}=[X_{1}, X_{2}]=-\pd{x_{3}}\,, \qquad X_{4}=[X_{1},X_{3}]=\pd{x_{4}}\,,
$$
the set $\{X_{1},X_{2},X_{3},X_{4}\}$ generates a nilpotent Lie algebra $V_{fw}$ with non-vanishing commuting relations
$$
[X_{1}, X_{2}]=X_{3}\,, \qquad [X_{1},X_{3}]=X_{4}\,.
$$
Because $\langle X_{1},X_{2},X_{3},X_{4}\rangle = \T\R^{4}$, one has that $(\mathbb{R}^4,X_{fw},V_{fw})$  is a locally automorphic Lie system. Therefore, there exist, at least locally, Lie symmetries of $X_{1},\ldots, X_{4}$ spanning Lie algebra isomorphic to $V_{fw}$ \cite{GLMV_19}, namely
$$
Y_{1}=\pd{x_{1}}\,, \qquad Y_{2}=\pd{x_{2}}+x_{1}\pd{x_{3}}+\frac{1}{2}x_{1}^{2}\pd{x_{4}}\,, \qquad Y_{3}=\pd{x_{3}}+x_{1}\pd{x_{4}}\,, \qquad Y_{4}=\pd{x_{4}}\,.
$$
Since $[Y_{1},Y_{2}]=Y_{3}$ and $Y_1\wedge Y_2\wedge Y_3$ is non-vanishing, the distribution $\D=\langle Y_{1},Y_{2}\rangle$ is maximally non-integrable. Moreover, $\D$ is a two-contact distribution, since it admits two commuting Lie symmetries $X_{3}$ and $X_{4}$ such that $Y_1,Y_2,X_3,X_4$ span $\T\mathbb{R}^4$. One can construct its associated two-contact form $\bm \eta_{fw}$ by considering dual forms to $Y_{1},Y_{2},X_{3},X_{4}$, namely
$$
\eta^{1}=\d x_{1}\,, \qquad \eta^{2}=\d x_{2}\,, \qquad \eta^{3}=x_{1}\d x_{2} -\d x_{3}\,, \qquad \eta^{4}=-\frac{1}{2}x_{1}^{2}\d x_{2}+\d x_{4}\,,
$$
are duals to $Y_1,Y_2,X_3,X_4$. Then,
$$
\bm \eta_{fw} = \eta^{3}\otimes e_{3} + \eta^{4} \otimes e_{4}
$$
is a two-contact form with $\ker \bm \eta_{fw}=\D$. Indeed, both $Y_{1}$ and $Y_{2}$ spans $\ker \bm \eta_{fw}$, $\dker_{fw}= \langle X_{3}, X_{4}\rangle$ and $\ker \bm \eta_{fw}\, \cap\, \dker_{fw}=0$. Because $X_{1},\ldots, X_{4}$ leave invariant $\D$, they are $\bm \eta_{fw}$-Hamiltonian vector fields with $\bm\eta_{fw}$-Hamiltonian two-functions
$$
\bm h_{1}=x_{2}  e_{3}-x_{3}  e_{4}\,,\qquad\bm h_{2}=-x_{1} e_{3} +\frac{1}{2}x_{1}^{2}  e_{4}\,, \qquad {\bm h}_{\alpha}=-e_{\alpha}\,, \qquad\alpha=3,4\,,
$$
which non-vanishing Lie brackets are
$$
\{\bm h_{1},\bm h_{2}\}=-\bm h_{3}\,, \qquad\{\bm h_{1}, \bm h_{3}\}=-\bm h_{4}\,.
$$
Note that $(\mathbb{R}^4,\bm \eta_{fw},X_{fw})$ is a four-contact Lie system and $(\mathbb{R}^4,\bm\eta_{fw},\bm h_{fw}=b_1(t){\bm h}_1+b_2(t){\bm h}_2)$ is an associated $\bm\eta_{fw}$-Lie Hamiltonian system. 

Using Remark \ref{Rem:SpecialCases}, one has that 
$$
\frac{\d \langle {\bm h}_1,e^3\rangle}{\d t}=-X_{\bm h}\langle {\bm h}_1,e^3\rangle=-\langle\{b_1(t){\bm h}_1+b_2(t){\bm h}_2,{\bm h}_1\},e^3\rangle-R_{{\bm h}_1}\langle {\bm h}_{fw},e^3\rangle=b_2(t)
$$and $\int^tb_2(t')\d t'-\langle {\bm h}_1,e^3\rangle$ is a $t$-dependent constant of motion for \eqref{eq:wheel_car_transformed}. 

\section{PDE Lie systems with a compatible {\it k}-contact manifold}
\label{sec:PDE_Lie_systems}

Let us briefly introduce PDE Lie systems \cite{CL_11,Ram_02} and show that a Lie algebra of $\bm \eta$-Hamiltonian vector fields relative to a co-oriented $k$-contact manifold $(M,\bm\eta)$ allows us to construct PDE Lie systems with a compatible $k$-contact manifold, the so-called {\it $k$-contact PDE Lie systems}. These PDE Lie systems can be understood as a certain type of Hamilton--De Donder--Weyl equations in the $k$-contact realm 
\cite{Riv_21}. Moreover, methods developed for $k$-contact Lie systems can be easily generalised to $k$-contact PDE Lie systems. 

    A {\it $t$-dependent $k$-vector field} on  $M$ with $t\in \mathbb{R}^k$ is a mapping $\bm X\colon \R^k \times M \to \bigoplus^k \T M$ such that $\tau^k\colon \bigoplus^k \T M \to M$ satisfies $\tau^k \circ \bm X=\pi$, where $\pi\colon \R^k\times M \to M$ is the natural projection onto $M$. Every $t$-dependent $k$-vector fields on $M$ with $t\in \mathbb{R}^k$ amounts to a series of $t$-parametrised $k$-vector fields ${\bm X}_t:M\rightarrow \bigoplus^k\T M$, which in turn implies that $\bm X$ can be considered as a family $\bm X=(X_1,\ldots,X_k)$ of $t$-dependent vector fields on $M$ with $(X_\alpha)_t=\tau^\alpha\circ \bm X_t$ for $\alpha=1,\ldots,k$ and every $t\in \mathbb{R}^k$.  

Given a  $t$-dependent $k$-vector field ${\bf X}=(X_1,\ldots,X_k)$ on $M$, its \textit{associated system} is the system of first-order partial differential equations
\begin{equation}\label{Eq:SystemPDEs}
    \frac{\partial \gamma}{\partial t^\alpha } = X_\alpha (t,\gamma)\,, \qquad \alpha=1,\ldots, k\,,\qquad \gamma\in  M\,,\qquad t=(t^1,\ldots,t^k) \in \R^k\,.
\end{equation}
Each particular solution $\gamma : \mathbb{R}^k \rightarrow M$ with $\gamma(0)=x_0$ is called an  \textit{integral section} of ${\bf X}$ with initial condition $x_0$ at $t=0$. Every $t$-dependent $k$-vector field has an associated $t$-dependent system of partial differential equations in normal form, and every $t$-dependent system of first-order partial differential equations in normal form uniquely determines the integral sections of some $t$-dependent $k$-vector field. We can then use ${\bf X}$ to refer both to a $t$-dependent $k$-vector field and the $t$-dependent system of partial differential equations determining its integral sections. It is worth stressing that not every $t$-dependent $k$-vector field ${\bm X}$ is such that its associated system \eqref{Eq:SystemPDEs} is integrable.

As for ordinary differential equations, an integrable system of first-order partial differential equations on $M$ of the form \eqref{Eq:SystemPDEs} is said to admit a \textit{superposition rule} \cite{CGM_07} if there exists a $t$-independent map $\Phi \colon M^{\ell} \times M \rightarrow M$ of the form
\begin{equation}
    x=\Phi(x_{(1)},\ldots,x_{(\ell)};\lambda)\,,
\end{equation}
such that every generic solution of \eqref{Eq:SystemPDEs} can be written as
\begin{equation}\label{result}
     x(t)=\Phi(x_{(1)}(t),\ldots,x_{(\ell)}(t);\lambda)\,,
\end{equation}
where $x_{(1)}(t),\ldots,x_{(\ell)}(t)$ is any generic family of particular solutions of system  \eqref{eq:SysPDEs} and $\lambda$ is an element of $M$.

\begin{theorem}[PDE Lie Theorem]\label{Lie_theorem_PDEs}
A $t$-dependent system of PDEs ${\bm X}$, with $t\in \mathbb{R}^k$, admits a superposition rule if, and only if, ${\bf X}=(X_1,\ldots,X_k)$ is integrable and it can be cast into the form 
\begin{equation}
    X_\alpha (t,x)=\sum^{r}_{\mu=1}b^\mu_{\alpha}(t)Y_{\mu}(x)\,, \qquad \alpha=1,\ldots,k\,,\label{kvec2}
\end{equation}
where $Y_{1}, \ldots, Y_{r}$ is a family of vector fields on $M$ spanning an $r$-dimensional Lie algebra of vector fields $V$ on $M$ and $b^\mu_{\alpha}(t)$, with $\alpha=1,\ldots,k$ and $\mu=1,\ldots,r$,  are arbitrary $t$-dependent functions. 
\end{theorem}

As in the case of Lie systems, the Lie algebra spanned by $Y_1,\ldots,Y_r$ is called a Vessiot--Guldberg Lie algebra of the PDE Lie system.

Let us explain how a PDE Lie system on a Lie group $G$ allows us to solve a PDE Lie system on a manifold $M$ whose Vessiot--Guldberg Lie algebra is given by the fundamental vector fields of a Lie group action $\Phi\colon G\times M\rightarrow M$ \cite{Ram_02}.

\begin{theorem} Given a PDE Lie system on a Lie group $G$ given by

\begin{equation}\label{eq:Initial}
\frac{\partial g}{\partial t^\mu}=\sum_{\alpha=1}^rb_{\mu}^\alpha(t,g)X_\alpha^R(g)\,,\qquad g\in G,\qquad \mu=1,\ldots,k,
\end{equation}
and a Lie group action $\Phi\colon G\times M\rightarrow M$, the solution of the PDE Lie system on $M$ given by
$$
\frac{\partial y}{\partial t^\mu}=\sum_{\alpha=1}^rb_{\mu}^\alpha(t,y)X_\alpha(y)\,,\qquad y\in M,\qquad \mu=1,\ldots,k,
$$
where $X_\alpha(y)=\Phi_{y*}X_\alpha^R(e)$ for the right-invariant vector fields $X_{\alpha}^{R} \in \mathfrak{g} \cong \T_{e}G$ with $\alpha=1,\ldots,r$, is
$$
y(t)=\Phi(g(t),y(0))
$$
for any $y(0)\in M$ and assuming that $g(t)$ is the particular solution to \eqref{eq:Initial} with $g(0)=e$.

\end{theorem}
Let us provide several known and new examples of PDE Lie systems. 

\begin{example}  Let us consider an example coming from a reduction of Wess--Zumino--Witten--Novikov (WZWN) equations \cite{FGRSZ_99,CGL_19}. Let us define a system of PDEs induced by a Lie algebra $\mathfrak{g}$  of the form
\begin{equation}\label{Eq:WZNWSystem}
\frac{\partial \psi}{\partial t^i}=\sum_{\beta=1}^{r}f^\beta_i(t)L_\psi M_\beta\,,\qquad \frac{\partial \psi}{\partial t^{-i}}=\sum_{\beta=1}^{r}f_{-i}^\beta(t) R_\psi M_\beta\,,\qquad \psi\in G\,,\qquad i=1,\ldots,n\,,
\end{equation}
where $\{M_1,\ldots,M_r\}$ is a basis of $\mathfrak{g}$, while $i=1,\ldots,n$,  $t=(t^{-n},\ldots,t^{-1},t^1,\ldots,t^n)\in \mathbb{R}^{2n}$ and $f^\beta_{-n},\ldots,f^\beta_{-1},f^\beta_1,\ldots,f^\beta_n \in \Cinfty(\R^{2n})$ are arbitrary functions for every $\beta=1,\ldots,r$. Recall that $L_\psi$ and $R_\psi$ are the left- and right- multiplications in $G$ by $\psi\in G$. Let us choose bases $\{Y^L_{1},\ldots,Y_{r}^{L}\}$ and $\{Y^R_{1},\ldots, Y_{r}^R\}$ of left- and right-invariant vector fields on $G$, respectively, so that system \eqref{Eq:WZNWSystem} can be then rewritten as
$$
\parder{\psi}{t^i}=\sum_{\beta=1}^rf_{i}^{\beta}(t)Y^L_\beta(\psi)\,, \qquad \parder{\psi}{t^{-i}}=\sum_{\beta=1}^rf_{-i}^{\beta}(t)Y^{R}_\beta(\psi)\,, \qquad \psi\in G\,, \qquad i=1,\ldots, n\,,
$$
which is the associated system of the $t$-dependent $2n$-vector field ${\bf X}=(X_{-n},\ldots,X_{-1},X_{1},\ldots,X_{n})$ of the form
$$
X_{-i}=\sum_{\beta=1}^{r}f_{-i}^{\beta}(t)Y^L_{\beta}\,, \qquad X_{i}=\sum_{\beta=1}^{r}f_i^{\beta}(t)Y^R_\beta\,, \qquad i=1,\ldots,n\,.
$$
The vector fields $Y_{1}^L, \ldots, Y_r^L,Y_1^R,\ldots,Y_r^R$ span a finite-dimensional Lie algebra, as every left-invariant vector field commutes with every right-invariant vector field, and together span a Lie algebra isomorphic to $\mathfrak{g}\oplus \mathfrak{g}$, provided  there is no vector field that is left- and right-invariant  simultaneously.

The integrability condition for \eqref{Eq:WZNWSystem} read
$$
\left[\frac{\partial}{\partial t^i}+X_i\ ,\ \frac{\partial}{\partial t^j}+X_j\right]=0\,,\qquad i,j=-n,\ldots,-1,1,\ldots,n\,.
$$
One can see that if $f_j^\beta=f_j^\beta(t^1,\ldots,t^n)$, $f_{-j}^\beta=f_{-j}^\beta(t^{-1},\ldots,t^{-n})$, for $j=1,\ldots,n$ and $\beta=1,\ldots,r$, and systems of PDEs given by the left- and right-hand sides in \eqref{Eq:WZNWSystem} are integrable, then our WZNW system is a PDE Lie system. 

\end{example}

\begin{example}
Let us analyse a system of PDEs associated with Floquet theory and geometric phases for periodic Lie systems \cite{FLV10}, which is here shown to be a PDE Lie system for the first time. This allows one to extend Floquet theory, which is applied to $t$-dependent periodic linear systems, to general nonlinear $t$-dependent periodic Lie  systems. More in detail, the co-adjoint action of the Lie group $G$ on the dual $\mathfrak{g}^*$ to its Lie algebra $\mathfrak{g}$ induces a system of PDEs  of the form
\begin{equation}
\label{eq:PDE-coadjoin-action}
\parder{\theta}{t}=\mathrm{ad}^*_{p_1(t,s)}\theta\,, \qquad \parder{\theta}{s}=\mathrm{ad}^*_{p_2(t,s)}\theta\,, \qquad  \theta \in \mathfrak{g}^*\,,\qquad (t,s)\in \mathbb{R}^2\,,
\end{equation}
for certain functions $p_1,p_2\colon \mathbb{R}^2\rightarrow \mathfrak{g}$. In particular, additional assumptions on the periodicity of $p_1, p_2$ can be assumed in Floquet theory. Anyway, we will skip hereafter these conditions, which could be implemented if desired. 
Let us choose a basis $\{\xi_{1},\ldots,\xi_{r}\}$ of $\mathfrak{g}$, which induces a family of fundamental vector fields of the coadjoint action $X_{\xi_{\alpha}}^{\mathfrak{g}^*}(\theta)=-\mathrm{ad}^*_{\xi_\alpha}\theta$ for $\alpha=1,\ldots,r$ and $\theta \in \mathfrak{g}^*$. Then, \eqref{eq:PDE-coadjoin-action} reads
$$
\parder{\theta}{t}=-\sum_{\alpha=1}^r b_1^{\alpha}(t,s)X^{\mathfrak{g}^*}_{\xi_\alpha}(\theta)\,, \qquad \parder{\theta}{s}=-\sum_{\alpha=1}^r b_2^{\alpha}(t,s)X^{\mathfrak{g}^*}_{\xi_\alpha}(\theta)\,, \qquad \theta \in \mathfrak{g}^*\,.
$$
The solutions of \eqref{eq:PDE-coadjoin-action} are the integral sections of the $(t,s)$-dependent two-vector field on $\mathfrak{g}^*$ given by
$$\bm X=\left(-\sum_{\alpha=1}^r b_1^{\alpha}(t,s)X^{\mathfrak{g}^*}_{\xi_\alpha}(\theta),  -\sum_{\alpha=1}^r b_2^{\alpha}(t,s)X^{\mathfrak{g}^*}_{\xi_\alpha}(\theta)\right),$$
which makes \eqref{eq:PDE-coadjoin-action} a PDE Lie system provided the integrability condition
$$
\left[\frac{\partial}{\partial t}-\sum_{\alpha=1}^r b_1^{\alpha}(t,s)X^{\mathfrak{g}^*}_{\xi_\alpha}(\theta),\frac{\partial}{\partial s}-\sum_{\alpha=1}^r b_2^{\alpha}(t,s)X^{\mathfrak{g}^*}_{\xi_\alpha}(\theta)\right]=0
$$
is satisfied.
\end{example}

\begin{example}
Let us consider a so-called {\it $\mathfrak{g}$-structure}, as defined in \cite{DK_16}, which are related to completely integrable distributions on  $M$ with simply transitive symmetry algebras and Cartan connections.
A $\mathfrak{g}$-structure is a differential one-form, $\bm\Upsilon$, on an $m$-dimensional manifold $M$ taking values in a Lie algebra $\mathfrak{g}$ so that
\begin{equation}\label{eq:gstructure}
\d \bm\Upsilon(X_1,X_2)=-\big[\bm\Upsilon(X_1),\bm\Upsilon(X_2)\big]=-\frac 12[\bm\Upsilon,\bm\Upsilon](X_1,X_2)
\end{equation}
for all vector fields $X_1,X_2$ on $M$. In particular, Maurer--Cartan forms are $\mathfrak{g}$-structures \cite{KNVol1_96}. Assume that $\mathfrak{g}$ is the Lie algebra of the Lie group $G$, and $G$ acts effectively on an $n$-dimensional manifold $N$. A {\it first integral} of $\bm\Upsilon$ is a function $f\colon M\rightarrow G$ such that $R\circ \T f=\bm\Upsilon$, where $R\colon v_g\in \T G\mapsto R_{g^{-1}*}v_g\in \mathfrak{g}$. By fixing a point $b \in N$ and introducing a Lie algebra homomorphism $\alpha\colon \mathfrak{g} \rightarrow \mathfrak{X}(N)$, we can find a local integral $f \colon M \rightarrow G$ of $\bm\Upsilon$ by means of the mapping $F \colon q\in M \mapsto f(q)b\in N$ , with $b\in N$ and $f(q)b$ the action of $f(q)$ on $b$ , satisfying 
\begin{equation}
\label{eq:g_structure}
(F_q)_*=\alpha_{F(q)} \circ  \Upsilon_q\,,
\end{equation}
where $\alpha_q \colon \mathfrak{g} \rightarrow \T_qN$ is given by $\alpha_q(X)=\alpha(X)_q$. In local coordinates, \eqref{eq:g_structure} reads
\begin{equation}\label{Eq:PDELieSystemGstructure}
\parder{F(q)}{q^i}=\sum_{\beta=1}^{\dim \mathfrak{g}}\alpha_\beta(F(q))\bm\Upsilon^\beta _i(q)\,, \qquad i=1,\ldots,m\,, 
\end{equation}
for $ M\ni q=(q^1,\ldots,q^m)$. It follows that \eqref{eq:g_structure} is a system of PDEs associated with the $M$-dependent $m$-vector field $\bm X =(X_{1}, \ldots, X_m)$ on $N$ of the form
$$
X_{i}=\sum_{j=1}^{\dim N}\sum_{\beta=1}^{\dim\mathfrak{g}}\alpha^j
_\beta(y)\Upsilon^\beta_i(q)\frac{\partial}{\partial y^j}\,, \qquad i=1,\ldots, m\,,\qquad y\in N.$$
Moreover, one can verify that $\bm X$ is integrable (see Proposition \ref{Lem:XInt} in the Appendix). 
This makes \eqref{eq:g_structure} into a PDE Lie system, which admits a superposition rule described in  \cite{DK_16}.

It is worth noting that all locally automorphic Lie systems admit a related $\mathfrak{g}$-structure.  Above example shows how the concept given in \cite{DK_16} is really important in the theory of Lie systems and their generalisations, which seems to passed unadvertised so far.
\end{example}

\begin{example}
The next example of PDE Lie system that we are going to study is related to Lax pairs for systems of PDEs \cite{FG_96}. In this case, a system of PDEs is written as the compatibility condition for the integrability of a system of PDEs of the form
\begin{equation}
\label{eq:Lax_pair}
\frac{\partial g}{\partial t}=A(t,s)g\,,\qquad \frac{\partial g}{\partial s}=B(t,s)g\,,\qquad g\in G\,,
\end{equation}
where $A(t,x)$ and $B(t,s)$ are functions taking values in the matrix algebra $\mathfrak{g}$ of a matrix group $G$. Moreover, the integrability condition for \eqref{eq:Lax_pair}, also called the {\it zero curvature condition}, can be written as
\begin{equation}\label{eq:IntPDE}
\frac{\partial A}{\partial s}-\frac{\partial B}{\partial t}+[A,B]=0\,.
\end{equation}
The coefficients of the matrix functions $A,B$ depend on a function $u\colon\mathbb{R}^2\rightarrow M$ that is a solution of the initial system of PDEs if, and only if, (\ref{eq:Lax_pair}) is integrable. In other words, the above system of PDEs \eqref{eq:IntPDE}, when considered as a system of PDEs on the variable $u$, retrieves the system of PDEs under study. Meanwhile, (\ref{eq:Lax_pair})  can be used to study the properties of \eqref{eq:IntPDE}, e.g. constants of motion, B\"acklund transformations, and other properties \cite{LG_18}. In particular, \eqref{eq:Lax_pair} can be understood as a linear spectral problem for studying immersed submanifolds related to integrable systems \cite{LG_18}. 

The adjoint Lie group action of $G$ on $\mathfrak{g}$ relate \eqref{eq:Lax_pair} with systems of PDEs of the form
$$
\frac{\partial X}{ \partial s}=[A,X]\,,\qquad \frac{\partial X}{ \partial t}=[B,X]\,, 
$$
where $X$ takes values in the matrix Lie algebra of $G$. In general, any system of the form
$$
\frac{\partial F}{\partial t}=\sum_{\alpha=1}^rA^\alpha_1(t,s)X_\alpha\,,\qquad \frac{\partial F}{\partial s}=\sum_{\alpha=1}^rA^\alpha_2(t,s)X_\alpha,
$$
where $X_1,\ldots,X_r$ are fundamental vector fields of an action of $G$ on a manifold $M$,
is a PDE Lie system related to \eqref{eq:Lax_pair}.
\end{example}
\begin{example} Let us again analyse a system of PDEs associated with Floquet theory \cite{FLV10}, which is here shown to be a PDE Lie system for the first time. Additionally, we will show how relate it to $k$-contact geometry. Let $G$ be an $r$-dimensional Lie group and let $\sigma$ be a smooth map $\R^{2} \ni (s,t) \mapsto \sigma(s,t) \in G$. The tangent vectors to the image of $\sigma$ are spanned by 
\begin{equation}
\label{eq:parametrized_surface}
\sigma_{s^*}(s,t)\equiv\parder{\sigma(s,t)}{s}=\sum_{\alpha=1}^rf_{s}^{\alpha}(s,t)R_{\sigma(s,t)^*}v_{\alpha}\,, \qquad \sigma_{t^*}(s,t)\equiv\parder{\sigma(s,t)}{t}=\sum_{\alpha=1}^rf_{t}^{\alpha}(s,t)R_{\sigma(s,t)^*}v_{\alpha}\,,
\end{equation}
where $\{v_{1}, \ldots, v_{r}\}$ is a basis of the Lie algebra $\mathfrak{g} \cong \T_{e}G$. By choosing  a base of right-invariant vector fields $X^{R}_1,\ldots, X_{r}^{R}$ on $G$, system \eqref{eq:parametrized_surface} can be written as

\begin{equation}
\parder{\sigma(s,t)}{s}=\sum_{\alpha=1}^rf_{s}^{\alpha}(s,t)X^{R}_{\alpha}(\sigma(s,t))\,, \qquad \parder{\sigma(s,t)}{t}=\sum_{\alpha=1}^rf_{t}^{\alpha}(s,t)X^{R}_{\alpha}(\sigma(s,t))\,.  
\end{equation}
The associated $\mathbb{R}^2$-dependent two-vector field $\bm X$ on $G$ is given by 
\begin{equation}
    \bm X=\left(\sum_{\alpha=1}^{r}f_{s}^{\alpha}(s,t)X^{R}_{\alpha}(\sigma), \sum_{\alpha=1}^{r}f_{t}^{\alpha}(s,t)X^{R}_{\alpha}(\sigma)\right)
\end{equation}
and the integrability condition reads
$$
\left[\parder{}{s}+ \sum_{\alpha=1}^{r}f_{s}^{\alpha}(s,t)X^{R}_{\alpha}(\sigma)\ ,\ \parder{}{t}+\sum_{\alpha=1}^{r}f_{t}^{\alpha}(s,t)X^{R}_{\alpha}(\sigma)\right]=0\,.
$$
In certain cases, one can transform the right-invariant vector fields $(X_1)_{(s,t)}=\sum_{\alpha=1}^{r}f_{s}^{\alpha}(s,t)X^{R}_{\alpha} $, $(X_2)_{(s,t)}=\sum_{\alpha=1}^{r}f_{t}^{\alpha}(s,t)X^{R}_{\alpha} $ into $\bm \eta$-Hamiltonian vector fields related to a $k$-contact form $\bm \eta$, making \eqref{eq:parametrized_surface} a $k$-contact PDE Lie system. It is enough to use the properties of the Lie algebra of a Lie group $G$ to construct a left-invariant $k$-contact form  as done in \cite{LRS_25} for Lie algebras of certain special unitary Lie groups. 
\end{example}
The above example suggests us the following definition.

\begin{definition}
    A {\it $k$-contact PDE Lie system} is a triple $(M,\bm\eta,{\bm X}:\mathbb{R}^k\times M\rightarrow \bigoplus^k\T M)$, where $\bm\eta$ is a $k$-contact form on $M$ and ${\bm X}$ is a PDE Lie system on $M$ with $t\in \mathbb{R}^k$, whose smallest 
    Lie algebra, $V^X$, consists of $\bm\eta$-Hamiltonian vector fields. 
\qeddiamond\end{definition}

Let us prove that a finite-dimensional Lie algebra of Hamiltonian vector fields relative to a $k$-contact form gives rise a PDE Lie system which can be understood as some Hamilton--De Donder--Weyl equations.

\begin{theorem} 
\label{thm:eta_Hamiltonian_k_vector_fields}
If $X_1,\ldots,X_k$ are $\bm\eta$-Hamiltonian vector fields relative to a co-oriented $\bm \eta$-contact manifold $(M,\bm\eta)$, then $(X_1,\ldots,X_k)$ is an ${\bm \eta}$-Hamiltonian $k$-vector field.
\end{theorem}

\begin{proof}
Since $X_1,\ldots,X_k$ are $\bm \eta$-Hamiltonian vector fields relative to $(M,\bm\eta)$, they posses a series of $\bm\eta$-contact $k$-functions ${\bm h}_1,\ldots,{\bm h}_k$ such that
$$
\iota_{X_\alpha}\d\bm\eta=\d\bm h_\alpha-\mathfrak{R}_{{\bm \eta}}\bm h_\alpha\,,\qquad \iota_{X_\alpha}\bm\eta=-\bm h_\alpha\,, \qquad \alpha=1,\ldots,k\,.
$$
If $\{e^1,\ldots,e^k\}$ is the dual basis  to $\{e_1,\ldots,e_k\}$ in $\mathbb{R}^k$, one has that
$$
\langle\iota_{X_\alpha}\d\bm\eta,e^\alpha\rangle=\langle\d\bm h_\alpha,e^\alpha\rangle-\langle\mathfrak{R}_{\bm\eta}\bm h_\alpha,e^\alpha\rangle\,,\qquad \langle\iota_{X_\alpha}\bm\eta,e^\alpha\rangle=-\langle \bm h_\alpha,e^\alpha\rangle\,, \qquad \alpha=1,\ldots,k\,.
$$
Summing over $\alpha=1,\ldots,k$, one gets
$$
\sum_{\alpha=1}^k\iota_{X_\alpha}\d\eta^\alpha=\d \sum_{\alpha=1}^k\langle \bm h_\alpha,e^\alpha\rangle -\mathfrak{R}_{\bm\eta}\sum_{\alpha=1}^k\langle \bm h_\alpha,e^\alpha\rangle\,,\qquad  \iota_{X_\alpha}\eta^\alpha=-\sum_{\alpha=1}^k\langle \bm h_\alpha,e^\alpha\rangle\,.
$$
In other words, $(X_1,\ldots,X_k)$ is an $\bm\eta$-Hamiltonian $k$-vector field. 
\end{proof}

\begin{example} Let us analyse Floquet theory for Lie systems related to a Lie algebra isomorphic to $\mathfrak{gl}_2$. Lie systems of this type appear in numerous applications \cite{BCHLS13,BHLS15,FLV10}. For instance, it appears in the Floquet theory and geometric phases for oscillators with $t$-dependent periodic frequency, mass and drift \cite{LS_20}. Consider the system of PDEs of the form
\begin{equation}
\label{eq:Floquet_PDE_Lie_system}
\frac{\partial A}{\partial t}=Af_1(t,s)\,,\qquad\frac{\partial A}{\partial s} = Af_2(t,s)\,,\qquad A\in \mathrm{GL}_2\,,
\end{equation}
where $f_1,f_2\colon\mathbb{R}^2\rightarrow \mathfrak{gl}_2$ are functions such that the above system of PDEs is integrable, while $\mathrm{GL}_2$ is the general linear group of invertible $2\times 2$ matrices. In other words, the above system of PDEs can be written as
$$
\frac{\partial A} {\partial t}=X_1^L(t,s,A)\,,\qquad \frac{\partial A} {\partial s}=X_2^L(t,s,A)\,, 
$$
where $X_i^L(t,s,A)$, with $i=1,2$, are $\mathbb{R}^2$-dependent vector fields on $\mathrm{GL}_2$. Then,  \eqref{eq:Floquet_PDE_Lie_system} is a PDE Lie system associated with the $(s,t)$-dependent two-vector field $\bm X=(X^L_1,X^L_2)$ on ${\rm GL}_2$. Moreover, 
$$
\left[\parder{}{t}+ X_{1}^L\ ,\ \parder{}{s}+ X_{2}^{L}\right]=0\,.
$$
Moreover, one can consider a basis of right-invariant one-forms $\eta^R_1,\ldots,\eta^R_4$ dual at the neutral element $e$ in $\mathrm{GL}_2$ to a basis of $\mathfrak{gl}_2$ such that
\[[v_1,v_2]=v_1\,,\qquad [v_1,v_3]=2v_2\,,\qquad [v_2,v_3]=v_3\,,
\]
and $[v_4,v_{i}]=0$ for $i=1,\ldots,4$. Then, $\bm \eta^{\mathrm{GL}_2}=\eta^R_2 \otimes e_{1} + \eta^R_4 \otimes e_{2}$ is a two-contact form since $\d\eta^R_2=2\eta^R_1\wedge\eta^R_3$ and $\d\eta^R_4=0$. By recalling Example \ref{Ex:TDepIsotropic}, one gets that $X^L_{1}(t,s,A), X^L_{2}(t,s,A)$ are $\mathbb{R}^2$-dependent $\bm \eta^{{\rm GL}_2}$-Hamiltonian vector fields with $\mathbb{R}^2$-dependent $\bm \eta^{{\rm GL}_2}$-Hamiltonian two-functions $\bm h_{1}, \bm h_{2}$. Therefore, Theorem \ref{thm:eta_Hamiltonian_k_vector_fields} ensures that the $\bm X$ associated with \eqref{eq:Floquet_PDE_Lie_system} is an  $\mathbb{R}^2$-dependent $\bm \eta^{{\rm GL}_2}$-Hamiltonian two-vector field with the  $\mathbb{R}^2$-dependent Hamiltonian function $h=\langle\bm h_{1}, e^{1}\rangle+\langle\bm h_{2}, e^{2}\rangle$. Moreover, $({\rm GL}_2,\bm \eta^{{\rm GL}_2},{\bm X}:\mathbb{R}^2\times {\rm GL}_2\rightarrow \bigoplus^2\T {\rm GL}_2)$ is a two-contact PDE Lie system. 


\end{example}

The whole theory of $k$-contact Lie systems can be quite straightforwardly generalised to $k$-contact PDE Lie systems. For instance, every $k$-contact PDE Lie system ${\bm X}$ admits a Vessiot--Guldberg Lie algebra $V$ of Hamiltonian vector fields relative to a $k$-contact form $\bm \eta$. In turn, $V$ is related to a Lie algebra $\mathfrak{W}$ of $\bm \eta$-Hamiltonian $k$-functions. Moreover, every $k$-contact PDE Lie system ${\bm X}$ has $k$-components $X_1,\ldots,X_k$ and every $X_\alpha$ of them admits an $\mathbb{R}^k$-dependent $\bm \eta$-Hamiltonian $k$-function $\bm g_\alpha$ taking values in $\mathfrak{W}$. Then,  the analogue of \eqref{eq:Derivation} for a $k$-contact PDE Lie systems ${\bm X}$ and a general $\bm \eta$-Hamiltonian $k$-function ${\bm I}$ is
\begin{equation}\label{eq:Derivation2}
\frac{\partial {\bm I}}{\partial t^\beta}(t)=\frac{\partial  {\bm I}}{\partial t^\beta}(t)+(X_\beta)_t {\bm I}_t=\frac{\partial  {\bm I}}{\partial t^\beta}(t)+
\{{\bm I}_t,{\bm g}_\beta\},\qquad \beta=1,\ldots,k.
\end{equation}
A whole theory of master symmetries and constants of motion for $k$-contact PDE Lie systems can be developed on the basis of the above formula and the methods for $k$-contact Lie systems.

\section{Conclusions and outlook}\label{Sec:ConclusionsOutlook}
By using some new ideas based on the notions in \cite{LRS_25},  this paper introduces and establishes the theoretical and practical significance of $k$-contact Lie systems. Many examples of $k$-contact Lie systems have been presented, and $k$-contact geometry has been employed to study some of their features like $t$-dependent constants of motion, generators of constants of motion of order $s$, and master symmetries.  Applications demonstrate the approach versatility of our ideas.  The work also inspects PDE Lie systems and their potential compatibility with $k$-contact forms. 

It is quite certain that a coalgebra method is available for $k$-contact Lie systems of projectable type. The method could be also extended to $k$-contact PDE Lie systems admitting a projectable condition. Our future aim will be to develop it analysing its properties and main applications in the future. Note also that the study of master symmetries and generators of constants of motion of order $s$ has been just started, and there are many further possibilities that can be analysed. In particular, we aim to apply these ideas more in detail to PDE Lie systems with a compatible $k$-contact form. Moreover, it is still convenient to characterise $k$-contact Lie groups on Lie groups of dimension larger than three to extend the work \cite{LR_23}, which characterises $k$-contact Lie systems on Lie groups of dimension three.  

\section*{Appendix}
This appendix provides some technical issues that can be also found in \cite{LRS_25} and are here given to make the work more self-contained. Nevertheless, they may be skipped in a first lecture and used just in case of necessity.

\begin{proposition}\label{Prop:WellDefined} The map \eqref{eq:CurvDis} is well defined.
\end{proposition}
\begin{proof}
     The map does not depend on the vector fields $X,X'$ chosen to extend $v,v'$. Moreover, any $\bm\zeta\in \Omega^1(U,\R^k)$ such that $\ker \bm\zeta=\mathcal{D}|_U$ gives rise a trivialisation $\T U/\mathcal{D}|_U\simeq U\times \mathbb{R}^k$ and $\rho$ can be described  as 

\begin{equation}\label{eq:rho}
    \rho(v,v') = \bm\zeta_x([X,X']_x) = X_x\inn{X' }\bm\zeta  - X'_x\inn{X }
    \bm\zeta -\d\bm\zeta_x(X_x,X'_x) = -\d\bm\zeta_x(v,v')\,, \quad \forall v,v'\in \mathcal{D}_x\,.
\end{equation}           
\end{proof}

\begin{proof}[Proof of Proposition \ref{prop:eta-Hamiltonian-vector-field}]
    If $X$ is an $\bm\eta$-Hamiltonian vector field, then $[X,\mathcal{D}]\subset \mathcal{D}$. This implies that $\Lie_X\eta^\alpha=\sum_{\beta=1}^kf^\alpha_\beta\eta^\beta$ for certain functions $f^\alpha_\beta\in \Cinfty(M)$ and $\alpha,\beta=1,\ldots,k$. Then, defining $h^\alpha=-\iota_X\eta^\alpha$ for $\alpha=1,\ldots,k$, one has
    $$
        (\d \iota_X+\iota_X\d)\eta^\alpha=\sum_{\beta=1}^kf^\alpha_\beta\eta^\beta, \quad \alpha=1,\ldots,k \quad \Rightarrow \quad \iota_X\d \eta^\alpha=\d h^\alpha +\sum_{\beta=1}^kf^\alpha_\beta\eta^\beta,\qquad \alpha=1,\ldots,k\,.
    $$
    Contracting with $R_\gamma$, one obtains
    \begin{equation}\label{eq:Contr}
        0=\inn{R_\gamma}\inn{X}\d \eta^\alpha=R_\gamma h^\alpha+f^\alpha_\gamma,\qquad \alpha,\gamma=1,\ldots,k\,,
    \end{equation}
    which implies that $f_\gamma^\alpha = -R_\gamma h^\alpha$, for $\alpha,\gamma=1,\ldots,k$, and thus
    $$
        \inn{X}\d\eta^\alpha=\d h^\alpha -\sum_{\beta=1}^k\eta^\beta(R_\beta h^\alpha) ,\qquad \alpha=1,\ldots,k\,.
    $$
    The converse is immediate.
    \end{proof}
  \begin{proof}[Proof of  Proposition \ref{Prop:MaxNonkCon}]
    Let us proceed by reducing to absurd. If $\d\bm \eta$ is degenerate when restricted to $\ker \bm \eta$ at a certain $x\in M$, then there exists a non-zero tangent vector $v\in \ker\bm\eta_x$ such that $\d\bm\eta_x(v,w) = 0$ for every tangent vector $w\in \ker \bm\eta_x$. One gets that $\d\bm\eta_x(v,R_x) = 0$ for each Reeb vector field $R$ at $x$. Since $\bm\eta$ has $k$ Reeb vector fields spanning $\ker\d\bm\eta$, these vector fields, along with a basis of $\ker\bm \eta$, give rise at $x$ to a basis of $\T_xM$. It follows that $0\neq v\in \ker \d\bm\eta_x\cap\ker\bm\eta_x$ and $\bm \eta$ is not a $k$-contact form. This is a contradiction and thus $\d\bm\eta$ must be non-degenerate when restricted to $\ker\bm\eta$. In view of \eqref{eq:rho}, it follows that $\ker \bm\eta$ is maximally non-integrable.
\end{proof}

\begin{proof}[Proof of Theorem \ref{Prop:LocalEqu}]
    The direct part is a consequence of Proposition \ref{Prop:MaxNonkCon} and the Reeb vector fields of $\bm \eta$.
    
    Let us prove the converse part. Given the vector fields $S_1,\ldots,S_k$ on $U$, consider the differential one-forms $\eta^1,\ldots,\eta^k$ vanishing on $\mathcal{D}$ and dual to $S_1,\ldots,S_k$ on $U$. Such differential one-forms are unique and exist due to decomposition \eqref{eq:Dec}. Then, they give rise to $\bm\eta = \sum_{\alpha=1}^k\eta^\alpha\otimes e_\alpha\in \Omega^1(U,\mathbb{R}^k)$ and $\eta^1\wedge\ldots\wedge \eta^k\neq 0$ on $U$. For a basis $X_{k+1},\ldots,X_m$ of $\ker \bm \eta$ around $x$, one has that the fact that $[S_i,\mathcal{D}]\subset \mathcal{D}$ and $[S_i,S_j]=0$ for $i,j=1,\ldots,k$ implies that
    $$
        \d\bm\eta(S_\beta,S_\gamma) = 0\,,\qquad \d\bm\eta(S_\beta,X_\alpha) = 0\,,\qquad \alpha=k+1,\ldots, m\,,\quad \beta,\gamma=1,\ldots,k\,.
    $$
    Therefore, the rank of $\ker\d\bm\eta$ is at least $k$, and $S_1,\ldots,S_k$ span a supplementary distribution to $\mathcal{D}$ around $x\in M$. Moreover, there exists no tangent vector $v'_{x'}\in \ker \d\bm \eta$ such that $v_{x'}\wedge S_1(x')\wedge\ldots\wedge S_k(x')\neq 0$ for $x'\in U$, as otherwise there would be an element $0\neq v_{x'}\in \ker \bm \eta_{x'}\cap \ker\d\bm\eta_{x'}$. Such an element would belong to the kernel of $\d\bm\eta_{x'}$ restricted to $\ker \bm\eta_{x'}$ and, since $\mathcal{D}$ is maximally non-integrable, one has that $v_{x'}=0$. This is a contradiction, $\ker \d \bm \eta$ has rank $k$,  while $
    \rk 
    \mathcal{D}>0$ and $\mathcal{D}$ has corank $k$, and $\ker \d \bm\eta \cap \ker \bm\eta=0$. Hence, $\bm \eta$ is a $k$-contact form and $\mathcal{D}$ is a $k$-contact distribution.
\end{proof}

\begin{lemma}\label{Lem:MaurerCartan} Given an $n$-dimensional Lie algebra of vector fields $\langle Y_1,\ldots,Y_n\rangle$ on an $n$-dimensional manifold $M$ whose elements span  $\T M$, and let $\Upsilon_1,\ldots,\Upsilon_n$ be a dual basis to $Y_1,\ldots,Y_n$, one has
\[
 \d\Upsilon^i=-\frac 12 \sum_{j,k=1}^nc_{jk}
 \,^i\Upsilon^j\wedge \Upsilon^k,\qquad i=1,\ldots,n,
\]
for $[Y_i,Y_j]=\sum_{k=1}^nc_{ij}\,^kY_k$ and $1\leq i<j\leq n$.
\end{lemma}
\begin{proof} By using the formula for the differential of a one-form, it follows that
$$
 \d\Upsilon^i(Y_j,Y_k)=Y_j \Upsilon^i(Y_k)-Y_k\Upsilon^i(Y_j)-\Upsilon^i([Y_j,Y_k])=-\Upsilon^i([Y_j,Y_k])=-\Upsilon^i\left(\sum_{l=1}^6c_{jk}\,^lY_l\right)=-c_{jk}\,\!^i\,
$$
for every $i,j,k=1,\ldots,n$. Since $Y_1,\ldots,Y_n$ form a basis of vector fields on $M$, this finishes the proof.
\end{proof}

\begin{lemma}\label{Lem:XInt} The system of PDEs \eqref{Eq:PDELieSystemGstructure} is integrable.
\end{lemma}
\begin{proof} The result follows from the fact that the vector fields $\partial/\partial q^i+X_i$, with $i=1,\ldots,m$, on $M\times N$ commute among themselves, as shown in the following calculation
\begin{multline}
\left[\parder{}{q^i}+\sum_{\beta=1}^{\dim \mathfrak{g}}\alpha
_\beta(y)\Upsilon^\beta_i(q), \parder{}{q^j}+\sum_{\gamma=1}^{\dim \mathfrak{g}}\alpha 
_\gamma(y)\Upsilon^\gamma_j(q)\right]=
\\=\sum_{\gamma,\beta=1}^{\dim \mathfrak{g}}\Upsilon_i^\beta(q)\Upsilon_j^\gamma(q)[\alpha_\beta,\alpha_\gamma](y)+\sum_{\beta=1}^{\dim \mathfrak{g}}\alpha_\beta(y)\left(\parder{\Upsilon^\beta_j(q)}{q^i}-\parder{\Upsilon^\beta_i(q)}{q^j}\right)=\\
=\sum_{\gamma=1}^{\dim \mathfrak{g}}\left(\parder{\Upsilon^\gamma_j}{q_i}-\parder{\Upsilon^\gamma_i}{q_j}+\sum_{\beta,\epsilon=1}^{\dim \mathfrak{g}}\Upsilon^\beta_i\Upsilon^\epsilon_jc_{\beta\epsilon}\,\!^\gamma\right)(q)\alpha_\gamma(y)=2\left(\d\bm\Upsilon+\frac 12[\bm\Upsilon,\bm\Upsilon]\right) = 0\,,
\end{multline}
where the last equality follows from \eqref{eq:gstructure}.
\end{proof}
\section{Acknowledgements}
X. Rivas acknowledges partial financial support from a mikrogrant IDUB 01-D111-20-2004310 funded by the IDUB program of the Faculty of Physics (University of Warsaw) for his research stay. X. Rivas also acknowledges partial financial support from the Spanish Ministry of Science and Innovation, grants  PID2021-125515NB-C21, and RED2022-134301-T of AEI, and Ministry of Research and Universities of the Catalan Government, project 2021 SGR 00603. T. Sobczak acknowledges financial support from a doctoral grant offered by the the Doctoral School of Natural and Exact Sciences of the University of Warsaw. We would also like to thank the anonymous referees for their suggestions, which allowed us to improve the quality and clarity of this work. 

X. Rivas and J. de Lucas would like to devote this work to the memory of our colleague and friend, Miguel C. Mu\~noz–Lecanda, who
passed away on the Christmas’s Eve of 2023. He was, and will always be, a source of inspiration
for us.

\bibliographystyle{abbrv}
\bibliography{references.bib}

@article {CFR_13,
    AUTHOR = {Cari\~{n}ena, Jos\'{e} F. and Falceto, Fernando and Ra\~{n}ada, Manuel F.},
     TITLE = {Canonoid transformations and master symmetries},
   JOURNAL = {J. Geom. Mech.},
  FJOURNAL = {Journal of Geometric Mechanics},
    VOLUME = {{\bf 5}},
      YEAR = {2013},
    NUMBER = {2},
     PAGES = {151--166},
      ISSN = {1941-4889},
   MRCLASS = {70H15 (37J05 37J15 53D22)},
  MRNUMBER = {3078679},
MRREVIEWER = {Patricia Yanguas},
       DOI = {10.3934/jgm.2013.5.151},
       note = {\href{https://doi.org/10.3934/jgm.2013.5.151}{10.3934/jgm.2013.5.151}},
}

@phdthesis{Lop_24,
      title={The geometry of dissipation}, 
      author={Asier López-Gordón},
      year={2024},
school ={ICMAT, Madrid},
      eprint={2409.11947},
      archivePrefix={arXiv},
      primaryClass={math-ph},
      url={https://arxiv.org/abs/2409.11947}, 
        note = {\href{https://www.icmat.es/Thesis/2024/Tesis_Asier_Lopez.pdf}{arXiv:2409.11947}}
}

@incollection {Re_02,
    AUTHOR = {Respondek, Witold},
     TITLE = {Introduction to geometric nonlinear control; linearization,
              observability, decoupling},
 BOOKTITLE = {Mathematical control theory, {P}art 1, 2 ({T}rieste, 2001)},
    SERIES = {ICTP Lect. Notes, VIII},
     PAGES = {169--222},
 PUBLISHER = {Abdus Salam Int. Cent. Theoret. Phys., Trieste},
      YEAR = {2002},
   MRCLASS = {93-02 (93B07 93B29)},
  MRNUMBER = {1972789},

}

@incollection {Ja_02,
    AUTHOR = {Jakubczyk, B.},
     TITLE = {Introduction to geometric nonlinear control; controllability
              and {L}ie bracket},
 BOOKTITLE = {Mathematical control theory, {P}art 1, 2 ({T}rieste, 2001)},
    SERIES = {ICTP Lect. Notes, VIII},
     PAGES = {107--168},
 PUBLISHER = {Abdus Salam Int. Cent. Theoret. Phys., Trieste},
      YEAR = {2002},
   MRCLASS = {93-02 (93B29 93C10)},
  MRNUMBER = {1972788},
}

@article {MS_93,
    AUTHOR = {Murray, Richard M. and Sastry, S. Shankar},
     TITLE = {Nonholonomic motion planning: steering using sinusoids},
   JOURNAL = {IEEE Trans. Automat. Control},
  FJOURNAL = {Institute of Electrical and Electronics Engineers.
              Transactions on Automatic Control},
    VOLUME = {{\bf 38}},
      YEAR = {1993},
    NUMBER = {5},
     PAGES = {700--716},
      ISSN = {0018-9286},
   MRCLASS = {},
  MRNUMBER = {},
       DOI = {10.1109/9.277235},
       URL = {https://doi.org/10.1109/9.277235},
   note = {\href{https://doi.org/10.1109/9.277235}{10.1109/9.277235}}
}

@incollection {Mur_93,
    AUTHOR = {Murray, Richard M.},
     TITLE = {Control of nonholonomic systems using chained form},
 BOOKTITLE = {Dynamics and control of mechanical systems ({W}aterloo, {ON},
              1992)},
    SERIES = {Fields Inst. Commun.},
    VOLUME = {1},
     PAGES = {219--245},
 PUBLISHER = {Amer. Math. Soc., Providence, RI},
      YEAR = {1993},
   MRCLASS = {},
  MRNUMBER = {},
   note = {\href{https://doi.org/10.1090/fic/001/10}{10.1090/fic/001/10}}

}

@book{Lie1,
	title = {Vorlesungen über continuierliche Gruppen mit geometrischen und anderen Anwendungen  },
	copyright = {NOT_IN_COPYRIGHT},
	url = {https://www.biodiversitylibrary.org/item/58166},
	note = {},
	publisher = {Leipzig, B.G. Teubner},
	author = {Sophus Lie and Georg Scheffers},
	year = {1893},
	pages = {838},
}

@book{Lie2,
	title = {Theorie der Transformationsgruppen Dritter Abschnitt, Abteilung I. Unter Mitwirkung von
 Dr. F. Engel},
	copyright = {NOT_IN_COPYRIGHT},
	url = {https://www.biodiversitylibrary.org/item/58166},
	note = {},
	publisher = {Leipzig, B.G. Teubner},
	author = {Sophus Lie},
	year = {1893},
	pages = {838},
}

@book{Lie3,
	title = {On differential equations possessing fundamental integrals},
	copyright = {NOT_IN_COPYRIGHT},
	url = {https://www.biodiversitylibrary.org/item/58166},
	note = {},
	publisher = {Leipziger Berichte},
	author = {Sophus Lie},
	year = {1893},
	pages = {838},
}

@article{CCM_18,
    author = {Florio M. Ciaglia and Hans Cruz and Giuseppe Marmo},
    title = {{Contact manifolds and dissipation, classical and quantum}},
    journal = {Ann. Phys.},
    volume = {{\bf 398}},
    pages = {159--179},
    year = {2018},
    note = {\href{https://doi.org/10.1016/j.aop.2018.09.012}{10.1016/j.aop.2018.09.012}}
}

@article {CLS_13,
    AUTHOR = {Cari\~{n}ena, J. F. and de Lucas, J. and Sard\'{o}n, C.},
     TITLE = {Lie-{H}amilton systems: theory and applications},
   JOURNAL = {Int. J. Geom. Methods Mod. Phys.},
  FJOURNAL = {Int. J. Geom. Methods Mod. Phys.},
    VOLUME = {{\bf{10}}},
      YEAR = {2013},
    NUMBER = {9},
     PAGES = {1350047, 25},
      ISSN = {0219-8878},
   MRCLASS = {},
  MRNUMBER = {},
MRREVIEWER = {},
       DOI = {},
       URL = {},
  note = {\href{https://www.worldscientific.com/doi/abs/10.1142/S0219887813500473}{10.1142/S0219887813500473}}
}

@book {KNVol1_96,
    AUTHOR = {Kobayashi, Shoshichi and Nomizu, Katsumi},
     TITLE = {Foundations of differential geometry. {V}ol. {I}},
    SERIES = {Wiley Classics Library},
      NOTE = {},
 PUBLISHER = {John Wiley \& Sons, Inc., New York},
      YEAR = {1996},
     PAGES = {xii+329},
      ISBN = {0-471-15733-3},
   MRCLASS = {53-01},
  MRNUMBER = {1393940},
}

@article{GGMRR_21,
    author = {Jordi Gaset and Xavier Gràcia and Miguel C. Muñoz-Lecanda and Xavier Rivas and Narciso Román-Roy},
    title = {{A $k$-contact Lagrangian formulation for nonconservative field theories}},
    journal = {Rep. Math. Phys.},
    volume = {{\bf 87}},
    number = {3},
    pages = {347--368},
    year = {2021},
     note = {\href{https://doi.org/10.1016/S0034-4877(21)00041-0}{10.1016/S0034-4877(21)00041-0}}

}

@article{GLMV_19,
    author = {Xavier Gràcia and Javier de Lucas and Miguel Carlos Muñoz-Lecanda and Silvia Vilariño},
    title = {{Multisymplectic structures and invariant tensors for Lie systems}},
    journal = {J. Phys. A: Math. Theor.},
    volume = {{\bf 52}},
    number = {21},
    pages = {215201},
    year = {2019},
        note = {\href{https://iopscience.iop.org/article/10.1088/1751-8121/ab15f2}{10.1088/1751-8121/ab15f2}}
}

@article{GRR_22,
    title = {{{Skinner--Rusk} formalism for $k$-contact systems}},
    author = {Xavier Gràcia and Xavier Rivas and Narciso Román-Roy},
    journal = {J. Geom. Phys.},
    year = {2022},
    volume = {{\bf 172}},
    pages = {104429},
    note = {\href{https://doi.org/10.1016/j.geomphys.2021.104429}{10.1016/j.geomphys.2021.104429}}
}

@book {Hil_97,
    AUTHOR = {Hille, Einar},
     TITLE = {Ordinary differential equations in the complex domain},
    SERIES = {Pure and Applied Mathematics},
 PUBLISHER = {Wiley-Interscience, New
              York-London-Sydney},
      YEAR = {1976},
     PAGES = {xi+484},
   MRCLASS = {34A20},
  MRNUMBER = {499382},
MRREVIEWER = {Shlomo Strelitz},
note = {\href{https://doi.org/10.1090/S0002-9904-1977-14328-0}{10.1090/S0002-9904-1977-14328-0}}
}

@article{HLT_17,
    author = {Francisco J. Herranz and Javier de Lucas and Mariusz Tobolski},
    title = {{Lie--Hamilton systems on curved spaces: a geometrical approach}},
    journal = {J. Phys. A: Math. Theor.},
    volume = {{\bf 50}},
    year = {2017},
    number = {49},
    pages = {495201},
  note = {\href{https://doi.org/10.1088/1751-8121/aa918f}{10.1088/1751-8121/aa918f}}
}

@article {GL19,
    AUTHOR = {Grundland, A. M. and de Lucas, J.},
     TITLE = {On the geometry of the {C}lairin theory of conditional
              symmetries for higher-order systems of {PDE}s with
              applications},
   JOURNAL = {Differential Geom. Appl.},
  FJOURNAL = {Differential Geometry and its Applications},
    VOLUME = {{\bf 67}},
      YEAR = {2019},
     PAGES = {101557, 29},
      ISSN = {0926-2245},
   MRCLASS = {},
  MRNUMBER = {},
MRREVIEWER = {},
       DOI = {},
       URL = {},
note={\href{https://doi.org/10.1016/j.difgeo.2019.101557}{ 10.1016/j.difgeo.2019.101557}},
}

@article {FGRSZ_99,
    AUTHOR = {Ferreira, L. A. and Gomes, J. F. and Razumov, A. V. and
              Saveliev, M. V. and Zimerman, A. H.},
     TITLE = {Riccati-type equations, generalised {WZNW} equations, and
              multidimensional {T}oda systems},
   JOURNAL = {Comm. Math. Phys.},
  FJOURNAL = {Communications in Mathematical Physics},
    VOLUME = {{\bf 203}},
      YEAR = {1999},
    NUMBER = {3},
     PAGES = {649--666},
      ISSN = {0010-3616},
   MRCLASS = {37K10 (17B81 34A34 35Q55 37K30 81R12)},
  MRNUMBER = {1700146},
MRREVIEWER = {Gloria Mar\'{\i} Beffa},
       DOI = {10.1007/s002200050630},
       URL = {https://doi.org/10.1007/s002200050630},
        note = {\href{https://doi.org/10.1007/s002200050630}{10.1007/s002200050630}}

}

@article {BCHLS13,
    AUTHOR = {Ballesteros, A. and Cari\~{n}ena, J. F. and Herranz, F. J. and de
              Lucas, J. and Sard\'{o}n, C.},
     TITLE = {From constants of motion to superposition rules for
              {L}ie-{H}amilton systems},
   JOURNAL = {J. Phys. A},
  FJOURNAL = {Journal of Physics. A. Mathematical and Theoretical},
    VOLUME = {46},
      YEAR = {2013},
    NUMBER = {28},
     PAGES = {285203, 25},
      ISSN = {1751-8113},
   MRCLASS = {},
  MRNUMBER = {},
MRREVIEWER = {},
       DOI = {},
       URL = {},
    note = {\href{https://iopscience.iop.org/article/10.1088/1751-8113/46/28/285203}{10.1088/1751-8113/46/28/285203}}

}

@article {BHLS15,
    AUTHOR = {Blasco, Alfonso and Herranz, Francisco J. and de Lucas, Javier
              and Sard\'{o}n, Cristina},
     TITLE = {Lie-{H}amilton systems on the plane: applications and
              superposition rules},
   JOURNAL = {J. Phys. A},
  FJOURNAL = {J.  Physics. A},
    VOLUME = {{\bf 48}},
      YEAR = {2015},
    NUMBER = {34},
     PAGES = {345202, 35},
      ISSN = {1751-8113},
   MRCLASS = {},
  MRNUMBER = {},
MRREVIEWER = {},
       DOI = {10.1088/1751-8113/48/34/345202},
       URL = {https://doi.org/10.1088/1751-8113/48/34/345202},
 note = {\href{https://iopscience.iop.org/article/10.1088/1751-8113/48/34/345202}{10.1088/1751-8113/48/34/345202}}
}

@article {FLV10,
    AUTHOR = {Flores-Espinoza, R. and de Lucas, J. and Vorobiev, Yu. M.},
     TITLE = {Phase splitting for periodic {L}ie systems},
   JOURNAL = {J. Phys. A},
  FJOURNAL = {Journal of Physics. A. Mathematical and Theoretical},
    VOLUME = {{\bf 43}},
      YEAR = {2010},
    NUMBER = {20},
     PAGES = {205208, 11},
      ISSN = {1751-8113},
   MRCLASS = {},
  MRNUMBER = {},
MRREVIEWER = {},
       DOI = {},
       URL = {},
       NOTE = {{\href{https://iopscience.iop.org/article/10.1088/1751-8113/43/20/205208}{10.1088/1751-8113/43/20/205208}}}
}

@book{Lee_12,
    title = {{Introduction to Smooth Manifolds}},
    author = {John M. Lee},
    isbn = {},
    lccn = {},
    series = {Graduate Texts in Mathematics},
    volume = {218},
    year = {2012},
    edition = {},
    publisher = {Springer New York Heidelberg Dordrecht London},
         note = {\href{https://link.springer.com/book/10.1007/978-0-387-21752-9}{10.1007/978-0-387-21752-9}}
}

@article{Leh_79,
    author = {Olli Lehto},
    title = {{Remarks on Nehari's theorem about the Schwarzian derivative and schlicht functions}},
    journal = {J. Anal. Math.},
    volume = {{\bf 36}},
    pages = {184--190},
    year = {1979},
   note = {\href{https://doi.org/10.1007/BF02798778}{10.1007/BF02798778}}




}

@article{LG_18,
    title = {{A cohomological approach to immersed submanifolds via integrable systems}},
    author = {J. de Lucas and A.M. Grundland},
    year = {2018},
    journal = {Sel. Math. New Ser.},
    volume = {{\bf 24}},
    pages = {4749--4780},
    note = {\href{https://doi.org/10.1007/s00029-018-0434-y}{10.1007/s00029-018-0434-y}}
}

@article{LGRRV_22,
    author = {Javier de Lucas and Xavier Gràcia and Xavier Rivas and Narciso Román-Roy and Silvia Vilariño},
    title = {{Reduction and reconstruction of multisymplectic Lie systems}},
    journal = {J. Phys. A: Math. Theor.},
    volume = {{\bf 55}},
    number = {29},
    pages = {295204},
    year = {2022},
    note = {\href{https://doi.org/10.1088/1751-8121/ac78ab}{10.1088/1751-8121/ac78ab}}
}

@article{LL_19,
    author = {Manuel de León and Manuel Laínz-Valcázar},
    title = {{Contact Hamiltonian systems}},
    journal = {J. Math. Phys.},
    volume = {{\bf 60}},
    number = {10},
    pages = {102902},
    year = {2019},
note = {\href{https://doi.org/10.1063/1.5096475
}{10.1063/1.5096475}}
}

@article{LR_23,
    author = {Javier de Lucas and Xavier Rivas},
    title = {{Contact Lie systems: theory and applications}},
    journal = {J. Phys. A: Math. Theor.},
    year = {2023},
    volume = {{\bf 56}},
    number = {33},
    pages = {335203},
      note = {\href{https://iopscience.iop.org/article/10.1088/1751-8121/ace0e7}{10.1088/1751-8121/ace0e7}
}

}

@unpublished{LRS_25,
    author = {Javier de Lucas and Xavier Rivas and Tomasz Sobczak},
    title = {{Foundations on $k$-contact geometry}},
         year = {2025},
  doi = {10.48550/arXiv.2409.11001},
 URL = {https://arxiv.org/abs/2409.11001},
note = {\href{https://arxiv.org/abs/2409.11001}{	arXiv:2409.11001}},
    }

@book{LS_20,
    author = {Javier de Lucas and Cristina Sardón},
    title = {{A Guide to Lie Systems with Compatible Geometric Structures}},
    year = {2020},
    publisher = {World Scientific Publishing Co. Pte. Ltd., Singapore},
    note = {\href{https://doi.org/10.1142/q0208}{10.1142/q0208}}
}

@book{LSV_15,
    author = {Manuel de León and Modesto Salgado and Silvia Vilariño},
    title = {{Methods of Differential Geometry in Classical Field Theories}},
    publisher = {World Scientific},
    year = {2015},
address = {Hackensack},
  note = {\href{https://doi.org/10.1142/9693}{10.1142/9693}}
}

@article{LV_15,
    author = {Javier de Lucas and Silvia Vilariño},
    title = {{$k$-symplectic Lie systems: theory and applications}},
    journal = {J. Differ. Equ.},
    volume = {{\bf 258}},
    number = {6},
    pages = {2221--2255},
    year = {2015},
        note = {\href{https://doi.org/10.1016/j.jde.2014.12.005}{10.1016/j.jde.2014.12.005}}
}

@article{MS_88,
    author = {Jerrold E. Marsden and Juan Carlos Simo},
    title = {{The energy momentum method}},
    journal = {Act. Acad. Sci. Tau.},
    volume = {{\bf 1}},
    year = {1988},
    number = {124},
    pages = {245--268},
}

@article{Nik_00,
    author = {Sergey Nikitin},
    title = {{Control Synthesis for Čaplygin Polynomial Systems}},
    journal = {Acta Appl. Math.},
    volume = {{\bf 60}},
    pages = {199--212},
    year = {2000},
   note = {\href{https://doi.org/10.1023/A:1006474511627}{10.1023/A:1006474511627}}
}

@phdthesis{Riv_21,
    author = {Xavier Rivas},
    title = {{Geometrical aspects of contact mechanical systems and field theories}},
    school = {Universitat Politècnica de Catalunya (UPC)},
    year = {2021},
         note = {\href{https://arxiv.org/abs/2204.11537}{arXiv:2204.11537}}
}

@article{CR_02,
   author = {José F. Cariñena and Arturo Ramos},
   doi = {10.1023/A:1013913930134/METRICS},
   issn = {01678019},
   issue = {1-3},
   journal = {Acta Appl.  Math.},
   pages = {43-69},
   publisher = {Springer},
   title = {A new geometric approach to {L}ie systems and physical applications},
   volume = {{\bf 70}},
   year = {2002},
 note = {\href{https://doi.org/10.1023/A:1013913930134}{10.1023/A:1013913930134}}
}

@article{RRSV_11,
    author = {\'Angel M. Rey and Narciso Román-Roy and Modesto Salgado and Silvia Vilariño},
    title = {{On the $k$-symplectic, $k$-cosymplectic and multisymplectic formalisms of classical field theories}},
    journal = {J. Geom. Mech.},
    volume = {{\bf 3}},
    number = {1},
    year = {2011},
    pages = {113--137},
      note = {\href{https://www.aimsciences.org/article/doi/10.3934/jgm.2011.3.113}{10.3934/jgm.2011.3.113}}
}

@article{Sus_73,
    author = {Hector J. Sussmann},
    title = {Orbits of families of vector fields and integrability of systems with singularities},
    journal = {Trans. Amer. Math. Soc.},
    volume = {{\bf 180}},
    pages = {171--188},
    year = {1973},
           note = {\href{ https://www.ams.org/journals/bull/1973-79-01/S0002-9904-1973-13152-0/home.html}{10.1090/S0002-9904-1973-13152-0}}
}

@book{BH_16,
    author = {Augustin Banyaga and Djideme F. Houenou},
    title = {{A brief introduction to symplectic and contact manifolds}},
    address = {Singapore},
    series = {Nankai Tracts in Mathematics},
    publisher = {World Scientific Publishing Co. Pte. Ltd.},
    volume = {15},
    year = {2016},
        note = {\href{https://www.worldscientific.com/worldscibooks/10.1142/9667#t=aboutBook}{10.1142/9667 
}}
}

@book{CGM_00,
    author = {José F. Cariñena and Janusz Grabowski and Giuseppe Marmo},
    title = {{Lie--Scheffers systems: a geometric approach}},
    publisher = {Bibliopolis},
    year = {2000},
    address = {Naples},
    series = {Napoli Series on Physics and Astrophysics},

}

@article{CGM_07,
    author = {José F. Cariñena and Janusz Grabowski and Giuseppe Marmo},
    title = {Superposition rules, {Lie} theorem, and partial differential equations},
    journal = {Rep. Math. Phys.},
    volume = {{\bf 60}},
    year = {2007},
    number = {2},
    pages = {237--258},
  note = {\href{https://doi.org/10.1016/S0034-4877(07)80137-6}{10.1016/S0034-4877(07)80137-6}}

}

@article{CL_11,
    author = {José F. Cariñena and Javier de Lucas},
    title = {{Lie systems: theory, generalisations, and applications}},
    journal = {Dissertationes Math.},
    volume = {{\bf 479}},
    year = {2011},
    pages = {1--162},
  note = {\href{https://www.impan.pl/pl/wydawnictwa/czasopisma-i-serie-wydawnicze/dissertationes-mathematicae/all/479//88196/lie-systems-theory-generalisations-and-applications}{10.4064/dm479-0-1}}

}

@article{GGMRR_20,
    author = {Jordi Gaset and Xavier Gràcia and Miguel C. Muñoz-Lecanda and Xavier Rivas and Narciso Román-Roy},
    title = {{A contact geometry framework for field theories with dissipation}},
    journal = {Ann. Phys.},
    volume = {{\bf 414}},
    pages = {168092},
    year = {2020},
    note = {\href{https://doi.org/10.1016/j.aop.2020.168092}{10.1016/j.aop.2020.168092}}

}

@book{GR_07,
    title = {{L'algèbre et le Groupe de Virasoro. Aspects géométriques et algébriques, généralisations}},
    author = {Laurent Guieu and Claude Roger},
    publisher = {Les Publications CRM},
    address = {Montréal},
    year = {2007}
}

@article{Lav_18,
    author = {Sylvain Lavau},
    title = {{A short guide through integration theorems of generalized distributions}},
    journal = {Diff. Geom. Appl.},
    volume = {{\bf 61}},
    year = {2018},
    pages = {42--58},
   note = {\href{https://www.sciencedirect.com/science/article/pii/S0926224518302730?via%3Dihub}{ 10.1016/j.difgeo.2018.07.005}}
}

@article {FG_96,
    AUTHOR = {Fokas, A. S. and Gelfand, I. M.},
     TITLE = {Surfaces on {L}ie groups, on {L}ie algebras, and their
              integrability},
   JOURNAL = {Comm. Math. Phys.},
  FJOURNAL = {Communications in Mathematical Physics},
    VOLUME = {177},
      YEAR = {1996},
    NUMBER = {1},
     PAGES = {203--220},
      ISSN = {0010-3616},
   MRCLASS = {53A05 (22E60 35A30 53A10 58F07)},
  MRNUMBER = {1382226},
MRREVIEWER = {Jan L. Cie\'{s}li\'{n}ski},
       URL = {http://projecteuclid.org/euclid.cmp/1104286243},
        NOTE = {\href{http://projecteuclid.org/euclid.cmp/1104286243}{ 
10.1007/BF02102436}},
}

@book{Ol_93,
    author = {Olver, Peter J.},
    title = {{Applications of Lie groups to differential equations}},
    series = {Graduate Texts in Mathematics},
    volume = {107},
    edition = {},
    publisher = {Springer-Verlag, New York},
    year = {1993},
    note = {\href{https://link.springer.com/book/10.1007/978-1-4612-4350-2}{10.1007/978-1-4612-4350-2}}
}

@phdthesis{Ram_02,
  author = {Arturo Ramos},
  title = {Sistemas de Lie y sus aplicaciones en Física y Teoría de Control},
  school = {Universidad de Zaragoza},
  year = {2002},
  advisor = {José F. Cariñena},
   note = {\href{https://arxiv.org/abs/1106.3775}{arXiv:1106.3775}}
}

@article{Riv_23,
    author = {Xavier Rivas},
    title = {{Nonautonomous $k$-contact field theories}},
    journal = {J. Math. Phys.},
    volume = {{\bf 64}},
    number = {3},
    pages = {033507},
    year = {2023},
    note = {\href{https://doi.org/10.1063/5.0131110}{10.1063/5.0131110}}
}

@article{Ste_74,
    title = {{Accessible sets, orbits, and foliations with singularities}},
    author = {P. Stefan},
    journal = {Proc. Lond. Math. Soc.},
    year = {1974},
    volume = {{\bf s3-29}},
    number = {4},
    pages = {699--713},
             note = {\href{ https://doi.org/10.1112/plms/s3-29.4.699}{10.1112/plms/s3-29.4.699}}
}

@article{Vit_15,
    author = {Luca Vitagliano},
    title = {{$L_\infty$-algebras from multicontact geometry}},
    journal = {Diff. Geom. Appl.},
    year = {2015},
    volume = {{\bf 39}},
    pages = {147--165},
     note = {\href{https://doi.org/10.1016/j.difgeo.2015.01.006}{10.1016/j.difgeo.2015.01.006}}
}

@misc{ada_21,
      title={Existence and classification of maximally non-integrable distributions of derived length one}, 
      author={Jiro Adachi},
      year={2021},
      eprint={2111.01403},
      archivePrefix={arXiv},
      primaryClass={math.GT},
      url={}, 
    note = {\href{https://arxiv.org/abs/2111.01403}{	arXiv:2111.01403}}
}

@article {CH_13,
    AUTHOR = {Colin, Vincent and Honda, Ko},
     TITLE = {Reeb vector fields and open book decompositions},
   JOURNAL = {J. Eur. Math. Soc. (JEMS)},
  FJOURNAL = {Journal of the European Mathematical Society (JEMS)},
    VOLUME = {{\bf 15}},
      YEAR = {2013},
    NUMBER = {2},
     PAGES = {443--507},
MRREVIEWER = {Mehmetcik Pamuk},
note={\href{https://doi.org/10.4171/JEMS/365}{10.4171/JEMS/365}}
}

@article {Pas_12,
    AUTHOR = {Pasquotto, Federica},
     TITLE = {A short history of the {W}einstein conjecture},
   JOURNAL = {Jahresber. Dtsch. Math.-Ver.},
  FJOURNAL = {Jahresbericht der Deutschen Mathematiker-Vereinigung},
    VOLUME = {114},
      YEAR = {2012},
    NUMBER = {3},
     PAGES = {119--130},
      ISSN = {},
   MRCLASS = {},
  MRNUMBER = {},
MRREVIEWER = {},
       DOI = {},
       URL = {},
    note = {\href{https://doi.org/10.1365/s13291-012-0051-1}{10.1365/s13291-012-0051-1}}
}

@article{Fer_93,
   author = {R. L. Fernandes},
   note = {\href{https://doi.org/10.1088/0305-4470/26/15/028}{10.1088/0305-4470/26/15/028}},
   issn = {0305-4470},
   issue = {15},
   journal = {J. Phys. A},
   pages = {3797},
   publisher = {IOP Publishing},
   title = {On the master symmetries and bi-Hamiltonian structure of the Toda lattice},
   volume = {{\bf 26}},
   year = {1993}
}

@article {BM_10,
    AUTHOR = {Bl\'{a}zquez-Sanz, David and Morales-Ruiz, Juan J.},
     TITLE = {Local and global aspects of {L}ie superposition theorem},
   JOURNAL = {J. Lie Theory},
  FJOURNAL = {Journal of Lie Theory},
    VOLUME = {{\bf 20}},
      YEAR = {2010},
    NUMBER = {3},
     PAGES = {483--517},
      ISSN = {0949-5932},
   MRCLASS = {34M15},
  MRNUMBER = {2743101},
     note = {\href{ https://www.heldermann-verlag.de/jlt/jlt20/blala2e.pdf}{jlt20/blala2e.pdf}},

}

@misc{DK_16,
      title={The geometry of second-order ordinary differential equations}, 
      author={Boris Doubrov and Boris Komrakov},
      year={2016},
      eprint={1602.00913},
      archivePrefix={arXiv},
      primaryClass={math.DG},
      note = {\href{https://arxiv.org/abs/1602.00913}{arXiv:1602.00913}}, 
}

@inproceedings{Car_25,
author = {Carballal, Oscar},
    editor = { F. Nielsen and F. Barbaresco},
     title = {{New Lie systems from Goursat distributions: reductions and reconstructions}},
    booktitle = {{Geometric Science of Information. GSI 2025. LNCS}},
    year = {2026},
    publisher = {Springer, Cham},
    pages = { 311-319},
    volume = {{\bf 16035}},
    note = {\href{https://doi.org/10.1007/978-3-032-03924-8_36}{10.1007/978-3-032-03924-8\_36}}
}

@article{BD_93,
    title = {{Non-holonomic Kinematics and the Role of Elliptic Functions in Constructive Controllability}},
    year = {1993},
    journal = {Nonholonomic Motion Planning},
    author = {Brockett, R. W. and Dai, Liyi},
    pages = {1--21},
    publisher = {Springer, Boston, MA},
    url = {https://link.springer.com/chapter/10.1007/978-1-4615-3176-0_1},
    isbn = {978-1-4615-3176-0},
    doi = {10.1007/978-1-4615-3176-0{\_}1},
    note = { \href{https://doi.org/10.1007/978-1-4615-3176-0_1}{10.1007/978-1-4615-3176-0\_1}}

}

@unpublished{CCH_25,
    title = {{Contact Lie systems on Riemannian and Lorentzian spaces: from scaling symmetries to curvature-dependent reductions}},
    year = {2025},
    author = {Campoamor-Stursberg, Rutwig and Carballal, Oscar and Herranz, Francisco J.},
    url = {https://arxiv.org/pdf/2503.20558},
    arxivId = {2503.20558},
doi ={ 10.48550/arXiv.2503.20558
},
note = {\href{https://doi.org/10.48550/arXiv.2503.20558}{arXiv:2503.20558}}
}

@inproceedings{Vil_25,
 author = {Vilari{\~{n}}o, S.},
    editor = { F. Nielsen and  F. Barbaresco},
     title = {{A relation between {\it k}-symplectic and {\it k}-contact Hamiltonian systems}},
    booktitle = {{Geometric Science of Information. GSI 2025. LNCS}},
    year = {2026},
    publisher = {Springer, Cham},
    pages = {346–353  },
    volume = {{\bf 16035}},
   note = {\href{
https://doi.org/10.1007/978-3-032-03924-8_36}{10.1007/978-3-032-03924-8\_36}}
}

@article{GL_25,
    title = {{Quasi-rectifiable Lie algebras for partial differential equations}},
    year = {2025},
    journal = {Nonlinearity},
    author = {Grundland, A. M. and de Lucas, J.},
    number = {2},
    pages = {025006},
    volume = {{\bf 38}},
    publisher = {Institute of Physics},
    doi = {10.1088/1361-6544/ADA50E},
    issn = {13616544},
note={\href{https://doi.org/10.1088/1361-6544/ADA50E}{10.1088/1361-6544/ADA50E}},
}

@article{CGL_19,
    title = {{Quasi-Lie schemes for PDEs}},
    year = {2019},
    journal = {Int. J. Geom. Methods Mod. Phys.},
    author = {Carinena, J. F. and Grabowski, J. and de Lucas, J.},
    number = {7},
    volume = {{\bf 16}},
    note = {\href{https://doi.org/10.1142/S0219887819500968}{10.1142/S0219887819500968}},
    issn = {02198878}
}

\end{document}